\documentclass[11pt]{article}

\usepackage{enumerate}
\usepackage{amsmath,xspace,amssymb,stmaryrd}
\usepackage[dvips]{graphics}
\usepackage{proof}
\usepackage{latexsym}  
\usepackage{euscript}  
\usepackage{bussproofs}

\input xy
\xyoption{all}
\xyoption{2cell}
\UseAllTwocells
\title{Differential Restriction Categories}
\author{J.R.B. Cockett\thanks{Partially supported by NSERC, Canada.}, G.S.H. Cruttwell\thanks{Partially supported by PIMS, Calgary.}, 
            and J. D. Gallagher 
\\ Department of Computer Science,
University of Calgary,\\ Alberta, Canada}

\newtheorem{observation}{Remark}[section]
\newtheorem{lemma}[observation]{Lemma}  
\newtheorem{theorem}[observation]{Theorem}
\newtheorem{definition}[observation]{Definition}
\newtheorem{example}[observation]{Example}

\newtheorem{proposition}[observation]{Proposition} 
\newtheorem{corollary}[observation]{Corollary} 
 
\newcommand{\proof}{\noindent{\sc Proof:}\xspace}
\def\endproof{~\hfill$\Box$\vskip 10pt}

\newcommand{\x}{\times}
\newcommand{\<}{\langle}
\renewcommand{\>}{\rangle}

\newcommand{\rfour}{{\bf [R.4]}}
\newcommand{\dr}[1]{{\bf [DR.{#1}]}}

\newcommand{\tup}[5]{\ensuremath{\left(\overrightarrow{x}^{#3} \mapsto \left({#1}_i,{#2}_i\right)_{i=1}^{#4},{\mathcal {#5}}\right)}}
\newcommand{\tupc}[5]{\ensuremath{\left(\overrightarrow{x}^{#3} \mapsto \left({#1}_i,{#2}_i\right)_{i=1}^{#4},\left\langle {#5}\right\rangle\right)}}
\newcommand{\stup}[5]{\ensuremath{\left(\overrightarrow{x}^{#3} \mapsto \left({#1},{#2}\right)_{i=1}^{#4},{\mathcal {#5}}\right)}}
\newcommand{\stupc}[5]{\ensuremath{\left(\overrightarrow{x}^{#3} \mapsto \left({#1},{#2}\right)_{i=1}^{#4},\left\langle {#5}\right\rangle\right)}}

\newcommand{\fr}{\ensuremath{{\mathfrak f}{\mathfrak r}}}
\newcommand{\rzro}[2]{\ensuremath{\stupc{1}{1}{{#1}}{{#2}}{0}}}

\newcommand{\Par}{{\sf Par}\xspace}
\newcommand{\rst}[1]{\ensuremath{\overline{#1}}}  
\newcommand{\rs}[1]{\ensuremath{\overline{#1}\,}}  

\newcommand{\A}{\ensuremath{\mathbb A}\xspace}

\newcommand{\N}{\ensuremath{\mathbb N}\xspace}
\newcommand{\R}{\ensuremath{\mathbb R}\xspace}

\newcommand{\X}{\ensuremath{\mathbb X}\xspace}
\newcommand{\Y}{\ensuremath{\mathbb Y}\xspace}
\newcommand{\Z}{\ensuremath{\mathbb Z}\xspace}

\newcommand{\da}{\ensuremath{\downarrow} \! \!}
\newcommand{\jn}{\ensuremath{\mathbf{Jn}}}
\newcommand{\cl}{\ensuremath{\mathbf{Cl}}}
\newcommand{\op}{\mbox{\scriptsize op}}

\newcommand\nats{\hbox{$I \kern - .38em N$}} 
\newcommand\ints{\hbox{$Z \kern - .65em Z$}} 

\makeatletter

\newdimen\w@dth

\def\setw@dth#1#2{\setbox\z@\hbox{\scriptsize $#1$}\w@dth=\wd\z@
\setbox\@ne\hbox{\scriptsize $#2$}\ifnum\w@dth<\wd\@ne \w@dth=\wd\@ne \fi
\advance\w@dth by 1.2em}

\def\t@^#1_#2{\allowbreak\def\n@one{#1}\def\n@two{#2}\mathrel
{\setw@dth{#1}{#2}
\mathop{\hbox to \w@dth{\rightarrowfill}}\limits
\ifx\n@one\empty\else ^{\box\z@}\fi
\ifx\n@two\empty\else _{\box\@ne}\fi}}
\def\t@@^#1{\@ifnextchar_ {\t@^{#1}}{\t@^{#1}_{}}}

\def\t@left^#1_#2{\def\n@one{#1}\def\n@two{#2}\mathrel{\setw@dth{#1}{#2}
\mathop{\hbox to \w@dth{\leftarrowfill}}\limits
\ifx\n@one\empty\else ^{\box\z@}\fi
\ifx\n@two\empty\else _{\box\@ne}\fi}}
\def\t@@left^#1{\@ifnextchar_ {\t@left^{#1}}{\t@left^{#1}_{}}}

\def\two@^#1_#2{\def\n@one{#1}\def\n@two{#2}\mathrel{\setw@dth{#1}{#2}
\mathop{\vcenter{\hbox to \w@dth{\rightarrowfill}\kern-1.7ex
                 \hbox to \w@dth{\rightarrowfill}}%
       }\limits
\ifx\n@one\empty\else ^{\box\z@}\fi
\ifx\n@two\empty\else _{\box\@ne}\fi}}
\def\tw@@^#1{\@ifnextchar_ {\two@^{#1}}{\two@^{#1}_{}}}

\def\tofr@^#1_#2{\def\n@one{#1}\def\n@two{#2}\mathrel{\setw@dth{#1}{#2}
\mathop{\vcenter{\hbox to \w@dth{\rightarrowfill}\kern-1.7ex
                 \hbox to \w@dth{\leftarrowfill}}%
       }\limits
\ifx\n@one\empty\else ^{\box\z@}\fi
\ifx\n@two\empty\else _{\box\@ne}\fi}}
\def\t@fr@^#1{\@ifnextchar_ {\tofr@^{#1}}{\tofr@^{#1}_{}}}

\newdimen\W@dth
\def\setW@dth#1#2{\setbox\z@\hbox{$#1$}\W@dth=\wd\z@
\setbox\@ne\hbox{$#2$}\ifnum\W@dth<\wd\@ne \W@dth=\wd\@ne \fi
\advance\W@dth by 1.2em}

\def\T@^#1_#2{\allowbreak\def\N@one{#1}\def\N@two{#2}\mathrel
{\setW@dth{#1}{#2}
\mathop{\hbox to \W@dth{\rightarrowfill}}\limits
\ifx\N@one\empty\else ^{\box\z@}\fi
\ifx\N@two\empty\else _{\box\@ne}\fi}}
\def\T@@^#1{\@ifnextchar_ {\T@^{#1}}{\T@^{#1}_{}}}

\def\T@left^#1_#2{\def\N@one{#1}\def\N@two{#2}\mathrel{\setW@dth{#1}{#2}
\mathop{\hbox to \W@dth{\leftarrowfill}}\limits
\ifx\N@one\empty\else ^{\box\z@}\fi
\ifx\N@two\empty\else _{\box\@ne}\fi}}
\def\T@@left^#1{\@ifnextchar_ {\T@left^{#1}}{\T@left^{#1}_{}}}

\def\Tofr@^#1_#2{\def\N@one{#1}\def\N@two{#2}\mathrel{\setW@dth{#1}{#2}
\mathop{\vcenter{\hbox to \W@dth{\rightarrowfill}\kern-1.7ex
                 \hbox to \W@dth{\leftarrowfill}}%
       }\limits
\ifx\N@one\empty\else ^{\box\z@}\fi
\ifx\N@two\empty\else _{\box\@ne}\fi}}
\def\T@fr@^#1{\@ifnextchar_ {\Tofr@^{#1}}{\Tofr@^{#1}_{}}}

\def\Two@^#1_#2{\def\N@one{#1}\def\N@two{#2}\mathrel{\setW@dth{#1}{#2}
\mathop{\vcenter{\hbox to \W@dth{\rightarrowfill}\kern-1.7ex
                 \hbox to \W@dth{\rightarrowfill}}%
       }\limits
\ifx\N@one\empty\else ^{\box\z@}\fi
\ifx\N@two\empty\else _{\box\@ne}\fi}}
\def\Tw@@^#1{\@ifnextchar_ {\Two@^{#1}}{\Two@^{#1}_{}}}

\def\to{\@ifnextchar^ {\t@@}{\t@@^{}}}
\def\from{\@ifnextchar^ {\t@@left}{\t@@left^{}}}
\def\tofro{\@ifnextchar^ {\t@fr@}{\t@fr@^{}}}
\def\To{\@ifnextchar^ {\T@@}{\T@@^{}}}
\def\From{\@ifnextchar^ {\T@@left}{\T@@left^{}}}
\def\Two{\@ifnextchar^ {\Tw@@}{\Tw@@^{}}}
\def\Tofro{\@ifnextchar^ {\T@fr@}{\T@fr@^{}}}

\makeatother

\input{diagxy}

\begin{document}
\maketitle

\begin{abstract}
We combine two recent ideas: cartesian differential categories, and restriction categories.  The result is a new structure which axiomatizes the category of smooth maps defined on open subsets of 
$\R^n$ in a way that is completely algebraic.  We also give other models for the resulting structure, discuss what it means for a partial map to be additive or linear, and show that differential 
restriction structure can be lifted through various completion operations.  
\end{abstract}

\tableofcontents

\section{Introduction}

In \cite{cartDiff}, the authors proposed an alternative way to view differential calculus.  The derivative was seen as an operator on maps, with many of its typical properties (such as the chain 
rule) axioms on this operation.  The resulting categories were called cartesian differential categories, and the standard model is smooth maps between the spaces $\R^n$.  One interesting aspect 
of this project was the algebraic feel it gave to differential calculus.  The seven axioms of a cartesian differential category described all the necessary properties that the standard Jacobian 
has.  Thus, instead of reasoning with epsilon arguments, one could reason about differential calculus by manipulating algebraic axioms.

Moreover, as shown in \cite{resourceCalc}, cartesian (closed) differential categories provide a semantic basis for 
modeling the simply typed differential lambda-calculus described in \cite{diffLambda}.  This latter calculus is linked to various resource 
calculi which, as their name suggests, are useful in understanding the 
resource requirements of programs.  Thus, models of computation in 
settings with a differential operator are of interest in the semantics 
of computation when resource requirements are being considered.

Fundamental to computation is the possibility of non-termination.  Thus, an obvious extension of cartesian differential categories is to allow partiality of maps.  Of course, this has a natural 
analogue in the standard model: smooth maps defined on open subsets of $\R^n$ are a notion of partial smooth map which is ubiquitous in analysis.

To axiomatize these ideas, we combine cartesian differential categories with the restriction categories of $\cite{restrictionI}$.  Again, the axiomatization is completely algebraic: there are 
two operations (differentiation and restriction) that satisfy seven axioms for the derivative, four for the restriction, and two for the interaction of derivation and restriction.

Our goal in this paper is not only to give the definitions and examples of these ``differential restriction categories'', but also to show how natural the structure is.  There are a number of 
points of evidence for this claim.  In a differential restriction category, one can define what it means for a partial map such as 
\[ f(x) = \left\{ \begin{array}{ll}
         2x & \mbox{if $x \ne 5$};\\
        \uparrow & \mbox{if $x = 5$}.\end{array} \right. \]
to be ``linear''.  One can give a similar description for the notion of ``additive''.  The differential interacts so well with the restriction that not only does it preserve the order and 
compatibility relations, it also preserves joins of maps, should they exist.

Moreover, differential restriction structure is surprisingly robust\footnote{With the exception of being preserved when we take manifolds.  Understanding what happens when we take manifolds of a 
differential restriction category will be considered in a future paper: see the concluding section of this paper for further remarks.}.  In the final two sections of the paper, we show that 
differential structure lifts through two completion operations on restriction categories.  The first completion is the join completion, which freely add joins of compatible maps to a restriction 
category.  We show that if differential structure is present on the original restriction category, then one can lift this differential structure to the join completion.

The second completion operation is much more drastic: it adds ``classical'' structure to the restriction category, allowing one to classically reason about the restriction category's maps.  Again, 
we show that if the original restriction category has differential structure, then this differential structure lifts to the classical setting.  This is perhaps the most surprising result of the paper, 
as one typically thinks of differential structure as being highly non-classical.  In particular, it is not obvious how differentials of functions defined at a single point should work.  We show that 
what the classical completion is doing is adding germs of functions, so that a function defined on a point (or a closed set) is defined by how it works on any open set around that point (or closed set).  
It is these germs of functions on which one can define differential restriction structure.

The paper is laid out as follows.  In Section \ref{sectionRes}, we review the theory of restriction categories.  This includes reviewing the notions of joins of compatible maps, as well as the notion 
of a cartesian restriction category.

In Section \ref{sectionDiff}, we define differential restriction categories.  We must begin, however, by defining left additive restriction categories.  Left additive categories are categories in 
which it is possible to add two maps, but the maps themselves need not preserve the addition (for example, the set of smooth maps between $\R^n$).  Such categories were an essential base for defining 
cartesian differential categories, as the axioms need to discuss what happens when maps are added.  Here, we describe left additive restriction categories, in which the maps being added may only be 
partial.  One interesting aspect of this section is the definition of additive maps (those maps which do preserve the addition), which is slightly more subtle than its total counterpart.

With the theories of cartesian restriction categories and left additive restriction categories described, we are finally able to define differential restriction categories.  One surprise is that the 
differential automatically preserves joins.  Again, as with additive maps, the definition of linear is slightly more subtle than its total counterpart.  

In Section \ref{rat}, we develop a family of examples of differential restriction categories: rational functions over a commutative ring.  Rational functions (even over rigs), because of their ``poles'', 
provide a natural source of restriction structure.  We show that the formal derivative on these functions, together with this restriction, naturally forms a differential restriction category.  
The construction of rational functions presented here, is, we believe, novel: it involves the use of weak and rational rigs (described in \ref{rat:frac}).  While one can describe restriction 
categories of rational functions directly, the description of the restriction requires some justification.  Thus, we first characterize the desired categories abstractly, by 
showing they occur as subcategory of a particular, more general, partial map category.  This then makes the derivation of the concrete description straightforward.   Moreover, the theory we develop to 
support this abstract characterization appears to be interesting in its own right.  While many of the ideas of this section are implicit in algebraic geometry, the packaging of differential restriction 
categories makes both the partial aspects of these settings and their differential structure explicit.

In the next two sections, we describe what happens when we join or classically complete the underlying restriction category of a differential restriction category, and show that the differential 
structure lifts in both cases.  Again, this is important, as it shows how robust differential restriction structure is, as well as allowing one to differentiate in a classical setting.

Finally, in section \ref{sectionConclusion}, we discuss further developments.  An obvious step, given a differential restriction category with joins, is to use the manifold completion process of 
\cite{manifolds} to obtain a category of smooth manifolds.  While the construction does not yield a differential restriction category, it is clearly central to developing the differential geometry 
of such settings.  This is the subject of continuing work.

On that note, we would like to compare our approach to other categorical theories of smooth maps.  Lawvere's synthetic differential geometry (carried out in \cite{dubuc}, \cite{kock}, 
and \\ \cite{reyes}) is one such example.  The notion of smooth topos is central to Lawvere's program. A smooth topos is a topos which contains an object of ``infinitesimals''.  One thinks of the 
this object as the set $D = \{x: x^2 = 0 \}$.  Smooth toposes give an extremely elegant approach to differential geometry.  For example, one defines the tangent space of an object $X$ to be the 
exponential $X^D$.  This essentially makes the tangent space the space of all infinitesimal paths in $X$, which is precisely the intuitive notion of what the tangent space is.

The essential difference between the synthetic differential geometry approach and ours is the level of power of the relative settings.  A smooth topos is, in particular, a topos, and so enjoys a 
great number of powerful properties.  The differential restriction categories we describe here have fewer assumptions: we only ask for finite products, and assume no closed structure or subobject 
classifier.   Thus, our approach begins at a much more basic level.  While the standard model of a differential restriction category is smooth maps defined on open subsets of $\R^n$, the standard 
model of a smooth topos is a certain completion of smooth maps between all smooth manifolds.  In contrast to the synthetic differential geometry approach, our goal is thus to see at what minimal 
level differential calculus can be described, and only then move to more complicated objects such as smooth manifolds.

A number of authors have described others notions of smooth space: see, for example, \cite{chen}, \cite{fro}, \cite{sik}.  All have a similar approach, and the similarity is summed up in 
\cite{comparativeSmooth}:
\begin{quote}
``...we know what it means for a map to be smooth between certain subsets of Euclidean space and so in general we declare a function smooth if whenever, we examine it using those subsets, it 
is smooth.  This is a rather vague statement - what do we mean by `examine'? - and the various definitions can all be seen as ways of making this precise.''
\end{quote}
Thus, in each of these approaches, the author assumes an existing knowledge of smooth maps defined on open subsets of $\R^n$.  Again, our approach is more basic: we are seeking to understand 
the nature of these smooth maps between $\R^n$.  In particular, one could define Chen spaces, or Fr\"olicher spaces, based on a differential restriction category other than the standard model, 
and get new notions of generalized smooth space.

Finally, it is important to note that none of these other approaches work with \emph{partial} maps.  Our approach, in addition to starting at a more primitive level, gives us the ability to reason 
about the partiality of maps which is so central to differential calculus, geometry, and computation.  

\section{Restriction categories review}\label{sectionRes}

In this section, we begin by reviewing the theory of restriction categories.  Restriction categories were first described in $\cite{restrictionI}$ as an alternative to the notion of a ``partial map category''.  In a partial map category, one thinks of a partial map from $A$ to $B$ as a span
$$\xymatrix{& A' \ar[dl]_m \ar[dr]^f &\\A & & B\\}$$
where the arrow $m$ is a monic.  Thus, $A'$ describes the domain of definition of the partial map.  By contrast, a restriction category is a category which has to each arrow $f: A \to B$ a ``restriction'' $\rs{f}: A \to  A$.  One thinks of this $\rs{f}$ as giving the domain of definition:  in the case of sets and partial functions, the map $\rs{f}$ is given by 
$$\overline{f}x=
\begin{cases}
x & \mbox{ if} f(x) \text{ defined}\\
\text{undefined} & \text{ otherwise.}
\end{cases}$$
There are then four axioms which axiomatize the behavior of these restrictions (see below).

There are two advantages of restriction categories when compared to partial map categories.  The first is that they are more general than partial map categories.  In a partial map category, one needs to have as objects each of the possible domains of definition of the partial functions.  In a restriction category, this is not the case, as the domains of definition are expressed by the restriction maps.  This is important for the examples considered below.  In particular, the canonical example of a differential restriction category will have objects the spaces $\R^n$, and maps the smooth maps defined on open subsets of these spaces.  This is not an example of a partial map category, as the open subsets are not objects, but it is naturally a restriction category, with the same restriction as for sets and partial functions.

The second advantage is that the theory is completely algebraic.  In partial map categories, one deals with equivalence classes of spans and their pullbacks.  As a result, they are often difficult to work with directly.  In a restriction category, one simply manipulates equations involving the restriction operator, using the four given axioms.  As cartesian differential categories give a completely algebraic description of the derivatives of smooth maps, bringing these two algebraic theories together is a natural approach to capturing smooth maps which are partially defined.

\subsection{Definition and examples}

Restriction categories are axiomatized as follows.  Note that throughout this paper, we are using diagrammatic order of composition, so that ``$f$, followed by $g$'', is written $fg$.  

\begin{definition}
Given a category, $\X$, a {\bf restriction structure} on $\X$ gives for each, $A \stackrel{f}{\longrightarrow} B$, a restriction arrow, $A \stackrel{\rs{f}}{\longrightarrow} A$, that
satisfies four axioms:

\begin{enumerate}[{\bf [R.1]}]
\item $\rs{f} f = f$;
\item If $dom(f) = dom(g)$ then $\rs{g} \rs{f} = \rs{f} \rs{g}$;
\item If $dom(f) = dom(g)$ then $\rs{\rs{g}f} = \rs{g} \rs{f}$;
\item If $dom(h) = cod(f)$ then $f \rs{h} = \rs{fh} f$.
\end{enumerate}

A category with a specified restriction structure is a {\bf restriction category}.
\end{definition}

We have already seen two examples of restriction categories: sets and partial functions, and smooth functions defined on open subsets of $\R^n$.  For more examples see \cite{restrictionI}, as well as \cite{turing}, where restriction categories are used to describe categories of partial computable maps.  

A rather basic fact is that each restriction $\rst{f}$ is idempotent: we will call such idempotents {\bf restriction idempotents}.  We record this 
together with some other basic consequences of the definition:

\begin{lemma}
\label{lemma:restriction}
If \X is a restriction category then: 
\begin{enumerate}[{\em (i)}]
\item $\rs{f}$ is idempotent;
\item $\rs{f}\rs{fg} = \rs{fg}$;
\item $\rs{f\rs{g}}=\rs{fg}$; \label{oldR3}
\item $\rs{\rs{f}} = \rs{f}$;
\item $\rs{\rs{f}\rs{g}} = \rs{f}\rs{g}$;
\item If $f$ is monic then $\rs{f} = 1$ (and so in particular $\rs{1} = 1$);
\label{monics-total}
\item $\rs{f}g = g$ implies $\rs{g} = \rs{f}\rs{g}$.
\end{enumerate}
\end{lemma}

\begin{proof}
Left as an exercise.
\end{proof}


\subsection{Partial map categories}

As alluded to in the introduction to this section, an alternative way of axiomatizing categories of partial maps is via spans where 
one leg is a monic.  We recall this notion here.  These will be important, as we shall see that rational functions over a commutative 
rig naturally embed in a larger partial map category.  

\begin{definition}\label{stablemonics}
Let $\X$ be a category, and $\mathcal{M}$ a class of monics in $\X$.  $\mathcal{M}$ is a {\bf stable system of monics} in case
\begin{enumerate}[{\bf [SSM.1]}]
\item all isomorphisms are in $\mathcal{M}$;
\item $\mathcal{M}$ is closed to composition;
\item for any $m : B' \rightarrow B \in \mathcal{M}$, $f : A \rightarrow B \in \bf{C}$
the following pullback, called an $\mathcal{M}$-pullback, exists and $m' \in \mathcal{M}$:

$$\xymatrix{
A' \ar[r]^{f'} \ar[d]_{m'} & B' \ar[d]^m\\
A \ar[r]_f & B\\
}$$

\end{enumerate}
\end{definition}

\begin{definition}
An {\bf $\mathcal{M}$-Category} is a pair $(\X,\mathcal{M})$ where $\X$ is a category with a specified system of stable monics $\mathcal{M}$.
\end{definition}

Given an $\mathcal{M}$-Category, we can define a category of partial maps.

\begin{definition}
Let $(\X,\mathcal{M})$ be an $\mathcal{M}$-Category.  Define $\Par(\X,\mathcal{M})$ to be the category where
\begin{description}
\item[Obj: ] The objects of $\X$
\item[Arr: ] $A \stackrel{(m,f)}{\longrightarrow} B$ are classes of spans $(m,f)$,
$$\xymatrix{& A' \ar[dl]_m \ar[dr]^f &\\A & & B\\}$$ where $m \in \mathcal{M}$.  The classes of spans are quotiented by the equivalence
relation $(m,f) \sim (m',f')$ if there is an isomorphism, $\phi$, such that both triangles in the following diagram commute.
$$\xymatrix{
& A' \ar@/_1pc/[dl]_m \ar@/_/[drr]^f \ar@{.>}[r]^\phi & A'' \ar@/^/[dll]^{m'} \ar@/^1pc/[dr]^{f'} &\\
A & & & B\\
}$$
\item[Id: ] $A \stackrel{(1_A,1_A)}{\longrightarrow} A$
\item[Comp: ] By pullback; i.e. given $A \stackrel{(m,f)}{\longrightarrow} B, B \stackrel{(m',f')}{\longrightarrow} C$, the pullback

$$\xymatrix{
& & A'' \ar[dl]_{m''} \ar[dr]^{f''}& &\\
& A' \ar[dl]_m \ar[dr]^f & & B' \ar[dl]_{m'} \ar[dr]^{f'} &\\
A & & B & & C\\
}$$

gives a composite $A \to^{(m''m,f''f')} C$.  (Note that without the equivalence relation on the arrows, the associative law would not hold.)  

\end{description}
\end{definition}

Moreover, this has restriction structure: given an arrow $(m,f)$, we can define its restriction to be $(m,m)$.  From \cite{restrictionI}, we have the following completeness result:

\begin{theorem}
Every restriction category is a full subcategory of a category of partial maps.
\end{theorem}

However, it is not true that every full subcategory of a category of partial maps is a category of partial maps, so the restriction notion is more general.

\subsection{Joins of compatible maps}\label{subsecJoins}

An important aspect of the theory of restriction categories is the idea of the join of two compatible maps.  We first describe what it means for two maps to be compatible, that is, equal where they are both defined.  

\begin{definition}
Two parallel maps $f,g$ in a restriction category are \emph{compatible}, written $f \smile g$, if $\rs{f}g = \rs{g}f$.  
\end{definition}
Note that compatibility is \emph{not} transitive.  Recall also the notion of when a map $f$ is less than or equal to a map $g$:
\begin{definition}
$f \leq g$ if $\rs{f}g = f$.
\end{definition}
This captures the notion of $g$ having the same values as $f$, but having a smaller domain of definition.  Note that this inequality is in fact anti-symmetric.

An important alternative characterization of compatibility is the following:
\begin{lemma}\label{lemmaAltComp}
In a restriction category,
	\[ f \smile g \Leftrightarrow \rs{f}g \leq f \Leftrightarrow \rs{g}f \leq g. \]
\end{lemma}
\begin{proof}
If $f \smile g$, then $\rs{f}g = \rs{g}f \leq f$.  Conversely, if $\rs{f}g \leq f$, then by definition, $\rs{\rs{f}g}g = \rs{f}g$, so $\rs{g}f = \rs{f}g$.
\end{proof}

We can now describe what it means to take the join of compatible maps.  Intuitively, the join of two compatible maps $f$ and $g$ will be a map which is defined everywhere $f$ and $g$ are, while taking the value of $f$ where $f$ is defined, and the value of $g$ where $g$ is defined.  There is no ambiguity, since the maps are compatible.  
\begin{definition}
Let \X be a restriction category.  We say that \X is a {\bf join restriction category} if for any family of pairwise compatible maps $(f_i : X \to Y)_{i \in I}$, there is a map $\bigvee_{i \in I} f_i : X \to Y$ such that
\begin{itemize}
	\item for all $i \in I$, $f_i \leq \bigvee_{i \in I} f_i$;
	\item if there exists a map $g$ such that $f_i \leq g$ for all $i \in I$, then $\bigvee f_i \leq g$;
\end{itemize}
(that it, it is the join under the partial ordering of maps in a restriction category) and these joins are compatible with composition: that is, for any $h: Z \to X$,
\begin{itemize}
	\item $h(\bigvee_{i \in I} f_i) = \bigvee_{i \in I}hf_i$.
\end{itemize}
\end{definition}

Note that by taking an empty family of compatible maps between objects $X$ and $Y$, we get a ``nowhere-defined'' map $\emptyset_{X,Y}: X \to Y$ which is the bottom element of the partially ordered set of maps from $X$ to $Y$.

Obviously, sets and partial functions have all joins - simply take the union of the domains of the compatible maps.  Similarly, continuous functions on open subsets also have joins.

Note that the definition only asks for compatibility of joins with composition on the left.  In the following proposition, we show that this implies compatibility with composition on the right. 
\begin{proposition}\label{propJoins}
Let $\X$ be a join restriction category, and $(f_i)_{i \in I}: X \to Y$ a compatible family of arrows.
\begin{enumerate}[(i)]
\item for any $j \in I$, $\rs{f_j} (\bigvee_{i \in I} f_i) = f_j$;
\item $\rs{\bigvee_{i \in I} f_i} = \bigvee_{i \in I} \rs{f_i}$;
\item for any $h: Y \to Z$, $(\bigvee_{i \in I} f_i)h = \bigvee_{i \in I} f_ih$.
\end{enumerate}
\end{proposition}
\begin{proof}
\begin{enumerate}[(i)]
\item This is simply a reformulation of $f_j \leq \bigvee_i f_i$.

\item By the universal property of joins, we always have $\bigvee \rs{f_i} \leq \rs{\bigvee f_i}$.  Note that this also implies that $\bigvee \rs{f_i}$ is a restriction idempotent, since it is less than or equal to a restriction idempotent.  Now, to show the reverse inequality, consider:
\begin{eqnarray*}
&   & \rs{\bigvee_{i \in I} f_i} \bigvee_{j \in I} \rs{f_j} \\
& = & \rs{ \bigvee_{j \in I} \rs{f_j} \bigvee_{i \in I} f_i} \mbox{ since $\bigvee \rs{f_j}$ is a restriction idempotent,} \\
& = & \rs{ \bigvee_{j \in I} f_j} \mbox{ by (i),}
\end{eqnarray*}
as required. 
\item Again, by the universal property of joins, we automatically have $\bigvee (f_ih) \leq (\bigvee f_i)h$.  In this case, rather than show the reverse inequality, we will instead show that their restrictions are equal: if one map is less than or equal to another, and their restrictions agree, then they must be equal.  To show that their restrictions are equal, we first show $\bigvee (f_i\rs{h}) = (\bigvee f_i)\rs{h}$:
\begin{eqnarray*}
&   & \left(\bigvee_{i \in I}f_i\right)\rs{h} \\
& = & \rs{\left(\bigvee_{j \in I} f_j\right)h} \left(\bigvee_{i \in I} f_i\right) \mbox{ by \rfour,} \\
& = & \bigvee_{i \in I} \rs{\left(\bigvee_{j \in I}f_j\right)h} f_i \\
& = & \bigvee_{i \in I} \rs{\rs{f_i}\left(\bigvee_{j \in I}f_j\right)h} f_i \\
& = & \bigvee_{i \in I} \rs{f_ih} f_i \mbox{ by (i),} \\
& = & \bigvee_{i \in I} f_i\rs{h} \mbox{ by \rfour.}
\end{eqnarray*}
Now, we can show that the restrictions of $\bigvee (f_ih)$ and $(\bigvee f_i)h$ are equal:
\begin{eqnarray*}
&   & \rs{ \bigvee_{i \in I}f_ih} \\
& = & \bigvee_{i \in I} \rs{f_ih} \mbox{ by (ii),} \\
& = & \bigvee_{i \in I} \rs{f_i \rs{h}} \\
& = & \rs{ \bigvee_{i \in i} f_i \rs{h}} \\
& = & \rs{ \left(\bigvee f_i \right) \rs{h}} \mbox{ by the result above,}
\end{eqnarray*}
as required.  
    
\end{enumerate}
\end{proof}

\subsection{Cartesian restriction categories}

Not surprisingly, cartesian differential categories involve cartesian structure.  Thus, to develop the theory which combines cartesian differential categories with restriction categories, it will be important to recall how cartesian structure interacts with restrictions.  This was described in \cite{restrictionIII} where it was noted that the resulting structure was equivalent to the P-categories introduced in \cite{pCategories}.  We recall the basic idea here:

\begin{definition}
Let \X be a restriction category.  A {\bf restriction terminal object} is an object $T$ in \X such that for any object $A$, there is a unique total map $!_A : A \longrightarrow T$ which satisfies $!_T = 1_T$.  Further, these maps $!$ must satisfy the property that for any map $f : A \longrightarrow B$, $f!_B \leq !_A$, i.e. $f!_B = \overline{f!_B} \ !_A = \overline{f \overline{!_B}} \ !_A = \overline{f} \ !_A$.

A {\bf restriction product} of objects $A,B$ in \X is defined by total projections
\begin{eqnarray*}
\pi_0 : A \times B \longrightarrow A & & \pi_1 : A \times B \longrightarrow B
\end{eqnarray*}
satisfying the property that for any object $C$ and maps $f : C \longrightarrow A,g:C \longrightarrow B$ there is a unique pairing map, $\<f,g\> : C \longrightarrow A \times B$ such that both triangles below exhibit lax commutativity
$$\xymatrix{
& C \ar[dl]_{f} \ar[dr]^g \ar@{.>}[d]|{\<f,g\>} \ar@{}@<2ex>[dl]|{\geq} \ar@{}@<-2ex>[dr]|{\leq}& \\
A & A \times B \ar[l]^{\pi_0} \ar[r]_{\pi_1} & B\\
}$$
that is,
\begin{eqnarray*}
\<f,g\> \pi_0 = \overline{\<f,g\>}f & \text{and} & \<f,g\> \pi_1 = \overline{\<f,g\>} g.  
\end{eqnarray*}
In addition, we ask that $\rs{\<f,g\>} = \rs{f}\rs{g}$.
\end{definition}

We require \emph{lax} commutativity as a pairing $\<f,g\>$ should only be defined as much as both $f$ and $g$ are.

\begin{definition}
A restriction category \X is a {\bf cartesian restriction category} if \X has a restriction terminal object and all restriction products.
\end{definition}

Clearly, both sets and partial functions, and smooth functions defined on open subsets of $\R^n$ are cartesian restriction categories.  

The following contains a number of useful results.

\begin{proposition}\label{propCart}
In any cartesian restriction category,
\begin{enumerate}[(i)]
	\item $\<f,g\>\pi_0 = \rs{f}g$ and $\<f,g\>\pi_1 = \rs{g}f$;
	\item if $e = \rs{e}$, then $e\<f,g\> = \<ef,g\> = \<f,eg\>$;
	\item $f\<g,h\> = \<fg,fh\>$;
	\item if $f \leq f'$ and $g \leq g'$, then $\<f,g\> \leq \<f',g'\>$;
	\item if $f \smile f'$ and $g \smile g'$, then $\<f,g\> \smile \<f',g'\>$;
	\item if $f$ is total, then $(f \times g)\pi_1 = \pi_1g$.  If $g$ is total, $(f \times g)\pi_0 = \pi_0f$.
\end{enumerate}
\end{proposition}

\begin{proof}
\begin{enumerate}[(i)]
	\item By the lax commutativity, $\<f,g\>\pi_0 = \rs{\<f,g\>} f = \rs{f}\rs{g}f = \rs{g}f$ and similarly with $\pi_1$.
	\item Note that
		\[ e\<f,g\>\pi_0 = e\rs{g}f = \rs{e}\rs{g}f = \rs{\rs{e}g}f = \rs{eg}f = \<f,eg\>\pi_0 \]
              A similar result holds with $\pi_1$, and so by universality of pairing, $e\<f,g\> = \<f,eg\>$.  By symmetry, it also equals $\<ef,g\>$.
	\item Note that 
		\[ f\<g,h\>\pi_0 = f\bar{h}g = \rs{fh}fg = \<fg,fh\>\pi_0 \]
	       where the second equality is by \rfour.  A similar result holds for $\pi_1$, and so the result follows by universality of pairing.  
	\item Consider
		\begin{eqnarray*}
		& & \rs{\<f,g\>}\<f',g'\> \\
		& = & \rs{f}\rs{g}\<f',g'\> \mbox{ by (i)}\\ 
		& = & \<f',\rs{f}\rs{g}g'\> \mbox{ by (ii)}\\
		& = & \<f',\rs{f},g\> \mbox{ since $g \leq g'$} \\
		& = & \<\rs{f}f',g\> \mbox{ by (ii)} \\
		& = & \<f,g\> \mbox{ since $f \leq f'$}.
		\end{eqnarray*}
	      Thus $\<f,g\> \leq \<f',g'\>$. 
	\item By Lemma \ref{lemmaAltComp}, we only need to show that $\rs{\<f,g\>}\<f',g'\> \leq \<f,g\>$.  But, again by Lemma \ref{lemmaAltComp}, we have $\rs{f}f' \leq f$ and $\rs{g}g' \leq g$, so by (iv) we get $\<\bar{f}f',\bar{g}g'\> \leq \<f,g\>$ and thus by (ii) and (i), we get $\rs{\<f,g\>}\<f',g'\> \leq \<f,g\>$.  
	\item \[ (f \times g)\pi_1 = \<\pi_0f,\pi_1g\>\pi_1 = \rs{\pi_0f}\pi_1g = \pi_1g \]
\end{enumerate} 
\end{proof}

If $\X$ is a cartesian restriction category which also has joins, then the two structures are automatically compatible:  

\begin{proposition}
In any cartesian restriction category with joins,
\begin{enumerate}[(i)]
	\item $\<f \vee g, h\> = \<f,h\> \vee \<g,h\>$ and $\<f,\emptyset\> = \<\emptyset,f\> = \emptyset$;
	\item $(f \vee g) \times h = (f \times h) \vee (g \times h)$ and $f \times \emptyset = \emptyset \times f = \emptyset$.
\end{enumerate}
\end{proposition}

\begin{proof}
\begin{enumerate}[(i)]
	\item Since $\rs{\<f,\emptyset\>} = \rs{f}\rs{\emptyset} = \rs{f}\emptyset = \emptyset$, by Proposition \ref{propJoins}, we have $\<f,\emptyset\> = \emptyset$.
	For pairing,
	\begin{eqnarray*}
	\< f \vee g, h \> & = & \rst{\<f \vee g, h \>} \<f \vee g, h \> \\
			  & = & \rst{f \vee g} \rst{h} \<f \vee g, h \> \\
			  & = & (\rst{f} \vee \rst{g})\<f \vee g, h \> \\
			  & = & (\rst{f}\<f \vee g, h\>) \vee (\rst{g}\<f \vee g, h \>) \\
			  & = & \<\rst{f}(f \vee g), h \> \vee \<\rst{g}(f \vee g), h \> \\
			  & = & \<f,h\> \vee \<g,h\>
	\end{eqnarray*}
	as required.	  
		
	\item Using part (a), $f \times \emptyset = \<\pi_0f, \pi_1\emptyset\> = \<\pi_0f, \emptyset\> = \emptyset$ and
	\begin{eqnarray*}
		(f \vee g) \times h & = & \<\pi_0(f \vee g), \pi_1 h \> \\
			            & = & \<(\pi_0f) \vee (\pi_0g), \pi_1h \> \\
			     	    & = & \<\pi_0f, \pi_1h \> \vee \<\pi_0g, \pi_1h \> \\
			     	    & = & (f \times h) \vee (g \times h)
	\end{eqnarray*}
\end{enumerate}
\end{proof}

We shall see that this pattern continues with left additive and differential restriction categories: if the restriction category has joins, then it is automatically compatible with left additive or differential structure.  

\section{Differential restriction categories}\label{sectionDiff}

Before we define differential restriction categories, we need to define left additive restriction categories.  Left additive categories were introduced in \cite{cartDiff} as a precursor to differential structure.  To axiomatize how the differential interacts with addition, one must define categories in which it is possible to add maps, but not have these maps necessarily preserve the addition (as is the case with smooth maps defined on real numbers).  The canonical example of one of these left additive categories is the category of commutative monoids with \emph{arbitrary} functions between them.  These functions have a natural additive structure given pointwise: $(f+g)(x) := f(x) + g(x)$, as well as $0$ maps: $0(x) := 0$.  Moreover, while this additive structure does not interact well with postcomposition by a function, it does with precomposition: $h(f+g) = hf + hg$, and $f0 = 0$.  This is essentially the definition of a left additive category.  

\subsection{Left additive restriction categories} 

To define left additive \emph{restriction} categories, we need to understand what happens when we add two partial maps, as well as the nature of the $0$ maps. Intuitively, the maps in a left additive category are added pointwise.  Thus, the result of adding two partial maps should only be defined where the original two maps were both defined.  Moreover, the $0$ maps should be defined everywhere.  Thus, the most natural requirement for the interaction of additive and restriction structure is that $\rst{f+g} = \rst{f}\rst{g}$, and that the $0$ maps be total.  

\begin{definition}
$\X$ is a {\bf left additive restriction category} if each $\X(A,B)$ is a commutative monoid with 
 $\rst{f+g} = \rst{f}\rst{g}$, $\rst{0} = 1$, and furthermore is left additive: $f(g+h) = fg + fh$ and $f0 = \rs{f}0$.  
\end{definition}
It is important to note the difference between the last axiom ($f0 = \rs{f}0$) and its form for left additive categories ($f0 = 0$).  $f0$ need not be total, so rather than ask that this be equal to $0$ (which is total), we must instead ask that $f0 = \rs{f}0$.  This phenomenon will return when we define differential restriction categories.  In general, any time an equational axiom has a 
variable which occurs on only one side, we must modify the axiom to ensure the variable occurs on both sides, by including 
the restriction of the variable on the other side.

There are two obvious examples of left additive restriction categories: commutative monoids with arbitrary partial functions between them, and the subcategory of these consisting of continuous or smooth functions defined on open subsets of $\R^n$.  \\

Some results about left additive structure:
\begin{proposition}\label{propLA} In any left additive restriction category:
\begin{enumerate}[(i)]
	\item f + g = \rs{g}f + \rs{f}g;
	\item if $e = \rs{e}$, then $e(f+g) = ef + g = f + eg$;
	\item if $f \leq f', g \leq g'$, then $f + g \leq f' + g'$;
	\item if $f \smile f', g \smile g'$, then $(f+g) \smile (f' + g')$.
\end{enumerate}
\end{proposition}

\begin{proof}
\begin{enumerate}[(i)]
	\item \[ f + g = \rs{f + g}(f+g) = \rs{f}\rs{g}(f + g) = \rs{g}\rs{f}f + \rs{f}\rs{g}g = \rs{g}f + \rs{f}g \]
	\item 
		\begin{eqnarray*}
		& & f + eg \\
		& = & \rs{eg}f + \rs{f}eg \mbox{ by (i)} \\
		& = & \rs{e}\, \rs{g}f + \rs{e}\rs{f}g \\
		& = & \rs{e}(\rs{g}f + \rs{f}g) \\
		& = & e(f + g) \mbox{ by (i)}\\
		\end{eqnarray*}
	\item Suppose $f \leq f'$, $g \leq g'$.  Then:
		\begin{eqnarray*}
		& & \rs{f+g}(f' + g') \\
		& = & \rs{f}\rs{g}(f' + g') \\
		& = & \rs{g}\rs{f}f' + \rs{f}\rs{g}g' \\
		& = & \rs{g}f + \rs{f}g \mbox{ since $f \leq f', g \leq g'$} \\
		& = & f + g \mbox{ by (i)}.
		\end{eqnarray*}
	      so $(f+g) \leq (f' + g')$.
	 \item Suppose $f \smile f'$, $g \smile g'$.  By lemma \ref{lemmaAltComp}, it suffices to show that $\rs{f+g}(f' + g') \leq f + g$.  By lemma \ref{lemmaAltComp}, we have $\rs{f}f' \leq f$ and $\rs{g}g' \leq g$, so by (ii), we can start with
	 	\begin{eqnarray*}
	 	\rs{f}f' + \rs{g}g' & \leq & f + g \\
	 	\rs{\rs{g}g'}\rs{f} f' + \rs{\rs{f}f'}\rs{g}g' & \leq & f + g \\
	 	\rs{g}\rs{g'}\rs{f}f' + \rs{f}\rs{f'}\rs{g}g' & \leq & f+ g \mbox{ by R3} \\
	 	\rs{f}\rs{g}(\rs{g'}f' + \rs{f'}g') & \leq & f + g \mbox{ by left additivity} \\
	 	\rs{f + g}(f' + g') & \leq & f + g \mbox{ by (i)}
	 	\end{eqnarray*}
\end{enumerate}
\end{proof}

If $\X$ has joins and left additive structure, then they are automatically compatible:
\begin{proposition}
If $\X$ is a left additive restriction category with joins, then:
\begin{enumerate}[(i)]
	\item $f + \emptyset = \emptyset$;
	\item $(\bigvee_i f_i) + (\bigvee_j  g_j) = \bigvee_{i,j} (f_i+g_j)$.
\end{enumerate}
\end{proposition}

\begin{proof}
\begin{enumerate}[(i)]
	\item $\rs{f + \emptyset} = \rs{f}\rs{\emptyset} = \rs{f}\emptyset = \emptyset$, so by Proposition \ref{propJoins}, $f + \emptyset = \emptyset$.	
	\item Consider:
\begin{eqnarray*}
(\bigvee_i f_i)+ (\bigvee_j g_j) 
& = & \rst{(\bigvee_i f_i)+ (\bigvee_j g_j)} (\bigvee_i f_i)+ (\bigvee_j g_j) \\
& = & (\bigvee_i \rst{f_i})(\bigvee_j \rst{g_j})((\bigvee_i f_i)+ (\bigvee_j g_j)) \\
& = & (\bigvee_{i,j} \rst{f_i}~\rst{g_j}) ((\bigvee_i f_i)+ (\bigvee_j g_j)) \\
& = & \bigvee_{i,j} \rst{g_j}~\rst{f_i} (\bigvee_i f_i)+ \rst{f_i}~\rst{g_j} (\bigvee_j g_j)) \\
& = & \bigvee_{i,j} \rst{g_j}~f_i + \rst{f_i}~g_j \mbox{ by Proposition \ref{propJoins},}\\
& = & \bigvee_{i,j} f_i + g_j, 
\end{eqnarray*}
as required.  
\end{enumerate}
\end{proof}

\subsection{Additive and strongly additive maps}

Before we get to the definition of a differential restriction category, it will be useful to have a slight detour, and investigate the nature of the additive maps in a left additive restriction category.  In a left additive category, arbitrary maps need not preserve the addition, in the sense that 
	\[ (x+y)f = xf + yf \mbox{ and } 0f = 0, \] 
are not taken as axioms.  Those maps which do preserve the addition (in the above sense) form an important subcategory, and such maps are called additive.  Similarly, it will be important to identify which maps in a left additive restriction category are additive.

Here, however, we must be a bit more careful in our definition.  Suppose we took the above axioms as our definition of additive in a left additive restriction category. In particular, asking for that equality would be asking for the restrictions to be equal, so that
\[ \rs{(x+y)f} = \rs{xf + yf} = \rs{xf}\rs{yf} \]
That is, $xf$ and $yf$ are defined exactly when $(x+y)f$ is.  Obviously, this is a problem in one direction: it would be nonsensical to ask that $f$ be defined on $x+y$ implies that $f$ is defined on both $x$ and $y$.  The other direction seems more logical: asking that if $f$ is defined on $x$ and $y$, then it is defined on $x+y$.  That is, in addition to being additive as a function, its domain is also additively closed.

Even this, however, is often too strong for general functions.  A standard example of a smooth partial function would be something $2x$, defined everywhere but $x=5$.  This map does preserve addition, wherever it is defined.  But it is not additive in the sense that its domain is not additively closed.  Thus, we need a weaker notion of additivity: we merely ask that $(x+y)f$ be \emph{compatible} with $xf + yf$.  Of course, the stronger notion, where the domain is additively closed, is also important, and will be discussed further below.    

\begin{definition}
Say that a map $f$ in a left additive restriction category is \textbf{additive} if for any $x,y$,
	\[ (x+y)f \smile xf + yf \mbox{ and } 0f \smile 0 \]
\end{definition}

We shall see below that for total maps, this agrees with the usual definition.  We also have the following alternate characterizations of additivity:

\begin{lemma}
A map $f$ is additive if and only if for any $x,y$,
	\[ \rs{xf}\rs{yf}(x + y)f \leq xf + yf \mbox{ and } 0f \leq 0 \]
or
	\[ (x\rs{f} + y\rs{f})f \leq xf + yf \mbox{ and } 0f \leq 0. \]
\end{lemma}
\begin{proof}
	Use the alternate form of compatibility (Lemma \ref{lemmaAltComp}) for the first part, and then \rfour \ for the second.
\end{proof}

\begin{proposition}
In any left additive restriction category,
\begin{enumerate}[(i)]
	\item total maps are additive if and only if $(x+y)f = xf + yf$;
	\item restriction idempotents are additive;
	\item additive maps are closed under composition;
	\item if $g \leq f$ and $f$ is additive, then $g$ is additive;
	\item 0 maps are additive, and additive maps are closed under addition.
\end{enumerate}
\end{proposition}
\begin{proof}
In each case, the 0 axiom is straightforward, so we only show the addition axiom.  
\begin{enumerate}[(i)]
\item It suffices to show that if $f$ is total, then $\rs{(x+y)f} = \rs{xf + yf}$.  Indeed, if $f$ is total,
       \[ \rs{(x+y)f} = \rs{x+y} = \rs{x}\rs{y} = \rs{xf}\rs{yf} = \rs{xf + yf}. \]
\item Suppose $e = \rs{e}$.  Then by \rfour,
       \[ (xe + ye)\rs{e} = \rs{\rs{xe + ye} \rs{e}} (xe + ye) \leq xe + ye \]
      so that $e$ is additive.
\item Suppose $f$ and $g$ are additive.  Then 
	\begin{eqnarray*}
	&   & \rs{xfg}\rs{yfg}(x+ y)fg \\
	& = & \rs{xfg}\rs{yfg}\rs{xf}\rs{yf}(x + y)fg \\
	& \leq & \rs{xfg}\rs{yfg}(xf + yf)g \mbox{ since $f$ is additive,} \\
	& \leq & xfg + yfg \mbox{ since $g$ is additive,} 
	\end{eqnarray*}
	as required.  
\item If $g \leq f$, then $g = \rs{g}f$, and since restriction idempotents are additive, and the composites of additive maps are additive, $g$ is additive.  
	\item For any $0$ map, $(x+y)0 = 0 = 0 + 0 = x0 + y0$, so it is additive.  For addition, suppose $f$ and $g$ are additive.  Then we have
		\[ (x+y)f \smile xf + yf \mbox{ and } (x+y)g \smile xg + yg.\]
	Since adding preserves compatibility, this gives
		\[ (x+y)f + (x+y)g \smile xf + yf + xg + yg. \]
	Then using \emph{left} additivity of $x, y$, and $x+y$, we get
		\[ (x+y)(f+g) \smile x(f+g) + y(f+g) \]
	so that $f+g$ is additive.
\end{enumerate}
\end{proof}

The one property we do not have is that if $f$ is additive and has a partial inverse $g$, then $g$ is additive.  Indeed, consider the left additive restriction category of arbitrary partial maps from $\Z$ to $\Z$.  In particular, consider the partial map $f$ which is only defined on $\{p,q,r\}$ for $r \ne p + q$, and maps those points to $\{n, m, n + m \}$.  In this case, $f$ is additive, since $(p+q)f$ is undefined.  However, $f$'s partially inverse $g$, which sends $\{n, m, n+m\}$ to  $\{p,q,r\}$, is not additive, since $ng + mg \ne (n+m)g$.  The problem is that $f$'s domain is not additively closed, and this leads us to the following definition.

\begin{definition}
Say that a map $f$ in a left additive restriction category is \textbf{strongly additive} if for any $x,y$,
	\[ xf + yf \leq (x+y)f \mbox{ and } 0f = 0. \]
\end{definition}

An alternate description, which can be useful for some proofs, is the following:

\begin{lemma}\label{strongadd}
$f$ is strongly additive if and only if $(x\rs{f} + y\rs{f})f = xf + yf$ and $0f = 0$.
\end{lemma}
\begin{proof}
\begin{eqnarray*}
&  & xf + yf \leq (x+y)f \\
& \Leftrightarrow & \rs{xf + yf}(x+y)f = xf + yf \\
& \Leftrightarrow & \rs{xf}\rs{yf}(x+y)f = xf + yf \\
& \Leftrightarrow & (\rs{xf}x + \rs{yf}y)f = xf + yf \\
& \Leftrightarrow & (x\rs{f} + y\rs{f})f = xf + yf \mbox{ by \rfour.}
\end{eqnarray*}
\end{proof}

Intuitively, the strongly additive maps are the ones which are additive in the previous sense, but whose domains are also closed under addition and contain $0$.  Note then that not all restriction idempotents will be strongly additive, and a map less than or equal to a strongly additive map need not be strongly additive.  Excepting this, all of the previous results about additive maps hold true for strongly additive ones, and in addition, a partial inverse of a strongly additive map is strongly additive.

\begin{proposition}
In a left additive restriction category,
\begin{enumerate}[(i)]
	\item strongly additive maps are additive, and if $f$ is total, then $f$ is additive if and only if it is strongly additive;
	\item $f$ is strongly additive if and only if $\rs{f}$ is strongly additive and $f$ is additive;
	\item identities are strongly additive, and if $f$ and $g$ are strongly additive, then so is $fg$;
	\item $0$ maps are strongly additive, and if $f$ and $g$ are strongly additive, then so is $f+g$;
	\item if $f$ is strongly additive and has a partial inverse $g$, then $g$ is also strongly additive.
\end{enumerate}
\end{proposition}
\begin{proof}
In most of the following proofs, we omit the proof of the $0$ axiom, as it is straightforward.  
\begin{enumerate}[(i)]

\item Since $\leq$ implies $\smile$, strongly additive maps are additive, and by previous discussion, if $f$ is total, the restrictions 
      of $xf + yf$ and $(x+y)f$ are equal, so $\smile$ implies $\leq$.  
\item When $f$ is strongly additive then $f$ is additive.  To show that $\rs{f}$ is strongly additive we have:
	\begin{eqnarray*}
	&   & (x\rs{f} + y\rs{f})\rs{f} \\
	& = & \rs{(x\rs{f} + y\rs{f})f}(x\rs{f} + y\rs{f}) \mbox{ by \rfour,} \\
	& = & \rs{xf + yf}(x\rs{f} + y\rs{f}) \mbox{by \ref{strongadd} as $f$ is strongly additive,} \\
	& = & \rs{x\rs{f}}\rs{y\rs{f}}(x\rs{f} + y\rs{f}) \\
	& = & x\rs{f} + y\rs{f}
	\end{eqnarray*}
       Together with  $0\rs{f} = \rs{0f}0 = \rs{0}0 = 0$, this implies, using Lemma \ref{strongadd}, that $\rs{f}$ is strongly additive.

      Conversely, suppose $\rs{f}$ is strongly additive and $f$ is additive.  First, observe:
      \begin{eqnarray*}
       \rs{xf + yf} &=& \rs{xf}\rs{yf}\\
                    &=& \rs{x\rs{f}}\rs{y\rs{f}}\\
                    &=& \rs{x\rs{f}+y\rs{f}}\\
                    &=& \rs{(x\rs{f}+y\rs{f})\rs{f}} \mbox{by \ref{strongadd} as $\rs{f}$ is strongly additive}\\
                    &=& \rs{(x\rs{f}+y\rs{f})f}
      \end{eqnarray*}
       This can be used to show:
      \begin{eqnarray*}
        xf+yf &=& \rs{xf+yf}(xf+yf)\\
              &=& \rs{(x\rs{f}+y\rs{f})f}(xf+yf) \mbox{ by the above}\\
              &=& \rs{xf+yf}(x\rs{f}+y\rs{f})f \mbox{as $f$ is additive }\\
              &=& \rs{(x\rs{f}+y\rs{f})f}(x\rs{f}+y\rs{f})f \mbox{ by the above}\\
              &=& (x\rs{f}+y\rs{f})f,
      \end{eqnarray*}
	For the zero case we have:
      \begin{eqnarray*}
       0f &=& \rs{0f}0 \mbox{ since $f$ is additive}\\
          &=& \rs{0\rs{f}}0\\
          &=& \rs{0}0 \mbox{ since $\rs{f}$ is strongly additive}\\
          &=& 0
      \end{eqnarray*}
	Thus, by lemma \ref{strongadd}, $f$ is strongly additive.
	\item Identities are total and additive, so are strongly additive.  Suppose $f$ and $g$ are strongly additive.  Then
	\begin{eqnarray*}
	&   & xfg + yfg \\
	& \leq & (xf + yf)g \mbox{ since $g$ strongly additive,} \\
	& \leq & (x+y)fg \mbox{ since $f$ strongly additive,}
	\end{eqnarray*}
	so $fg$ is strongly additive.  
\item Since any $0$ is total and additive, $0$'s are strongly additive.  Suppose $f$ and $g$ are strongly additive.  Then
	\begin{eqnarray*}
	&   & x(f+g) + y(f+g) \\
	& = & xf + xg + yf + yf \mbox{ by left additivity,} \\
	& \leq & (x+y)f + (x+y)g \mbox{ since $f$ and $g$ are strongly additive,} \\
	& = & (x+y)(f+g) \mbox{ by left additivity,} 
	\end{eqnarray*}
	so $f+g$ is strongly additive.
	\item Suppose $f$ is strongly additive and has a partial inverse $g$.  Using the alternate form of strongly additive,
	\begin{eqnarray*}
	&   & (x\rs{g} + y\rs{g})g \\
	& = & (xgf + ygf)g \\
	& = & (xg\rs{f} + yg\rs{f})fg \mbox{ since $f$ is strongly additive,} \\
	& = & (xg\rs{f} + yg\rs{f})\rs{f} \\
	& = & xg\rs{f} + yg\rs{f} \mbox{ since $\rs{f}$ strongly additive,} \\
	& = & xg + yg
	\end{eqnarray*}
	and $0g = 0fg = 0\rs{f} = 0$, so $g$ is strongly additive. 
\end{enumerate}
\end{proof}

Finally, note that neither additive nor strongly additive maps are closed under joins.  For additive, the join of the additive maps $f: \{n,m\} \to \{p,q\}$ and $g: \{n+m\} \to \{r\}$, where $p+q \ne r$, is not additive.  For strongly additive, if $f$ is defined on multiples of $2$ and $g$ on multiples of $3$, their join is not closed under addition, so is not strongly additive.

\subsection{Cartesian left additive restriction categories}

In a differential restriction category, we will need both cartesian and left additive structure.  Thus, we describe here how cartesian and additive restriction 
structures must interact.  

\begin{definition}
$\X$ is a {\bf cartesian left additive restriction category} if it is both a left additive and cartesian restriction category such that the product functor preserves 
addition (that is $(f+g) \x (h+k) = (f \x h) + (g\x k)$ and $0 = 0 \x 0$) and the maps $\pi_0$,$\pi_1$, and $\Delta$ are additive.
\end{definition}

If $\X$ is a cartesian left additive restriction category, then each object becomes canonically a (total) commutative monoid by $+_X = \pi_0+\pi_1: X \x X \to X$ 
and $0: 1 \to X$.  Surprisingly, assuming these total commutative monoids are coherent with the cartesian structure, one can then recapture the additive structure, 
as the following theorem shows.  Thus, in the presence of cartesian restriction structure, it suffices to give additive structure on the total maps to get a cartesian 
left additive restriction category.

\begin{theorem}\label{thmAddCart}
$\X$ is a left additive cartesian restriction category if and only if $\X$ is a cartesian restriction category in which each object 
is canonically a total commutative monoid, that is, for each object $A$, there are given maps $A \times A \to^{+_A} A$ and $1 \to^{0_A} A$ 
making $A$ a total commutative monoid, such that following exchange\footnote{Recall that the exchange map is defined by  
$\mbox{ex} := \<\pi_0 \x \pi_0,\pi_1 \x \pi_1\>$ and that it satisfies, for example, $\<\<f,g\>,\<h,k\>\>\mbox{ex}  = \<\<f,h\>,\<g,k\>\>$ and 
$(\Delta \x \Delta)\mbox{ex}  = \Delta$.} axiom holds:
	\[ +_{X \x Y} = (X \x Y) \x (X \x Y)\to^{\sf ex} (X \x X) \x (Y \x Y) \to^{+_X \x +_Y} X \x Y.\] 
\end{theorem} 

\begin{proof}
Given a canonical commutative monoid structure on each object, the left additive structure on $\X$ is defined by:
  $$\infer[\mbox{add}]{~~A \to_{f + g := \<f,g\>+_B} B~~}{A \to^f B & A \to^g B} ~~~~~ \infer[\mbox{zero}]{~~A \to_{0_{AB} := !_A 0_B} B~~}{}$$ 
That this gives a commutative monoid on each $\X(A,B)$ follows directly from the commutative monoid axioms on $B$ and 
the cartesian structure.  For example, to show $f + 0 = f$, we need to show $\<f,!_A 0_B \> = f$.  Indeed, we have
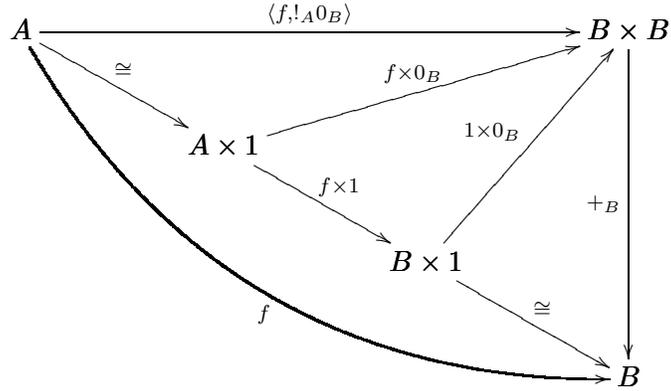
\[
\bfig
	\node a(0,0)[A]
	\node b(2100,0)[B \times B]
	\node c(700,-400)[A \times 1]
	\node d(1400,-800)[B \times 1]
	\node e(2100,-1200)[B]
	\arrow[a`b;\<f,!_A0_B\>]
	\arrow[a`c;\cong]
	\arrow[c`b;f \times 0_B]
	\arrow[c`d;f \times 1]
	\arrow[d`e;\cong]
	\arrow[b`e;+_B]
	\arrow[d`b;1 \times 0_B]
	\arrow|l|/{@{>}@/^-40pt/}/[a`e;f]
\efig
\]
the right-most shape commutes by one of the commutative monoid axioms for $B$, and the other shapes commute by coherences of the cartesian 
structure.  The other commutative monoid axioms are similar. 

For the interaction with restriction,
	\[ \rs{f+g} = \rs{\<f,g\> +_B} = \rs{\<f,g\> \rs{+_B}} = \rs{\<f,g\>} = \rs{f}\rs{g}, \]
and $\rs{0_{AB}} = \rs{!_A 0_B} = 1$ since $!$ and $0$ are themselves total.  

For the interaction with composition,
	\[ f(g+h) = f\<g,h\>+_C = \<fg,fh\>+_C = fg + fh \]
and
	\[ f0_{BC} = f!_B0_C = \rs{f}!_A0_C = \rs{f}0_{AC} \]
as required.

The requirement that $(f + g) \times (h + k) = (f \times h) + (g \times k)$ follows from the exchange axiom:
$$
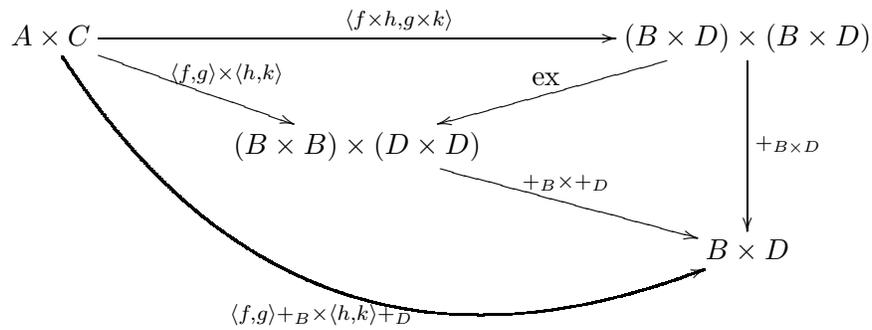
\xymatrix@C=4pc @R=2pc{A \x C \ar@/_5pc/[rrdd]_{\<f,g\> +_B \x \<h,k\> +_D} \ar[rr]^{\<f \x h, g \x k \>} \ar[dr]^{\<f, g\> \x \<h,k\>} 
             & & (B \x D) \x (B \x D) \ar[dl]_{\mbox{ex}} \ar[dd]^{+_{B \x D}} \\
             & (B \x B) \x (D \x D) \ar[dr]^{+_B \x +_D} \\
             & & B \x D}$$
the right triangle is the exchange axiom, and the other two shapes commute by the cartesian coherences.  

Since $\pi_0$ is total, $\pi_0$ is additive in case for all $f,g : A \to B \x C$, $(f+g)\pi_0 = f\pi_0 + g\pi_0$, which is shown 
by the following diagram:
$$\xymatrix@C=4pc @R=2pc{
A \ar@/_5pc/[drr]_{\<f\pi_0,g\pi_0\>}\ar[r]^{\<f,g\>} \ar[dr]|{\<\<f\pi_0,g\pi_0\>,\<f\pi_1,g\pi_1\>\>}
            & (B\x C) \x (B\x C) \ar[r]^{+_{(B\x C)}} \ar[d]|{\mbox{ex}}& B\x C \ar[r]^{\pi_0}& B\\
  & (B\x B) \x (C\x C) \ar[ur]_{+_B\x +_C} \ar[r]_{\pi_0} & B\x B \ar[ur]_{+_B} &
}$$

A similar argument shows that $\pi_1$ is additive. Since $\Delta$ is total, $\Delta$ is additive when for all 
$f,g: A \to B$, $f\Delta + g\Delta = (f+g)\Delta$.  This is shown by the following diagram:

$$\xymatrix@C=4pc @R=2pc{
 A \ar[r]^{\<f\Delta,g\Delta\>~~} \ar[d]_{\<f,g\>} \ar@/_5pc/[dd]_{\<f,g\>+_{B}} 
        & (B\x B)\x(B\x B) \ar[d]|{\mbox{ex}} \ar@/^5pc/@<3ex>[dd]^{+_{B\x B}}  \\
 B\x B \ar[d]_{+_{B}} \ar[ur]^{\Delta \x \Delta} \ar[dr]_{\<+_{B},+_{B}\>} \ar[r]^{\Delta} & (B\x B)\x(B\x B)\ar[d]|{+_{B}\x +_{B}} \\
 B \ar[r]_{\Delta} & B\x B
}$$
\end{proof}

\begin{proposition}\label{propCLA}
In a cartesian left additive restriction category:
\begin{enumerate}[(i)]
	\item $\<f,g\> + \<f',g'\> = \<f + f',g+g'\>$ and $\<0,0\> = 0$;
	\item if $f$ and $g$ are additive, then so is $\<f,g\>$;
	\item the projections are strongly additive, and if $f$ and $g$ are strongly additive, then so is $\<f,g\>$,
	\item $f$ is additive if and only if 
		\[ (\pi_0 + \pi_1)f \smile \pi_0f + \pi_1f \mbox{ and } \, 0f \smile 0; \]
	      (that is, in terms of the monoid structure on objects, $(+)(f) \smile (f \times f)(+)$ and $0f \smile 0$),
	\item $f$ is strongly additive if only if
		\[ (\pi_0 + \pi_1)f \geq \pi_0f + \pi_1f \mbox{ and } \, 0f = 0; \]
	      (that is, $(+)(f) \geq (f \times f)(+)$ and $0f \geq 0$).
\end{enumerate}
\end{proposition}
Note that $f$ being strongly additive only implies that $+$ and $0$ are lax natural transformations.  
\begin{proof}
\begin{enumerate}[(i)]
	\item Since the second term is a pairing, it suffices to show they are equal when post-composed with projections.  Post-composing with $\pi_0$, we get
	\begin{eqnarray*}
		&   & (\<f,g\> + \<f',g'\>)\pi_0 \\
		& = & \<f,g\>\pi_0 + \<f',g'\>\pi_0 \mbox{ since $\pi_0$ is additive,} \\
		& = & \rs{g}f + \rs{g'}f' \\
		& = & \rs{g}\rs{g'}(f + f') \\
		& = & \rs{g + g'}(f + f') \\
		& = & \<f + f',g + g'\>\pi_0
	\end{eqnarray*}
	as required.  The 0 result is direct.
	\item We need to show
		\[ (x+y)\<f,g\> \smile x\<f,g\> + y\<f,g\>; \]
	however, since the first term is a pairing, it suffices to show they are compatible when post-composed by the projections.  Indeed,
		\[ (x+y)\<f,g\>\pi_0 = (x+y)\rs{g}f \smile x\rs{g}f + y\rs{g}f \]
	while since $\pi_0$ is additive,
		\[ (x\<f,g\> + y\<f,g\>)\pi_0 = x\<f,g\>\pi_0 + y\<f,g\>\pi_0 = x\rs{g}f + y\rs{g}f \]
	so the two are compatible, as required.  Post-composing with $\pi_1$ is similar.
	\item Since projections are additive and total, they are strongly additive.  If $f$ and $g$ are strongly additive,
	\begin{eqnarray*}
	&   & x\<f,g\> + y\<f,g\> \\
	& = & \<xf,xg\> + \<yf,yg\> \\
	& = & \<xf + yf, xg + yg\> \mbox{ by (i)} \\
	& \leq & \<(x+y)f, (x+y)g\> \mbox{ since $f$ and $g$ are strongly additive,} \\
	& = & (x+y)\<f,g\>
	\end{eqnarray*}
	so $\<f,g\>$ is strongly additive.  
	\item If $f$ is additive, the condition obviously holds.  Conversely, if we have the condition, then $f$ is additive, since
		\[ (x+y)f = \<x,y\>(\pi_0 + \pi_1)f \smile \<x,y\>(\pi_0f + \pi_1f) = xf + yf \]
	as required.
	\item Similar to the previous proof.  
\end{enumerate}
\end{proof}

\subsection{Differential restriction categories}

With cartesian left additive restriction categories defined, we turn to defining differential restriction categories.  To do this, we begin by recalling the notion of a cartesian differential category.  The idea is to axiomatize the Jacobian of smooth maps.  Normally, the Jacobian of a map $f: X \to Y$ gives, for each point of $X$, a linear map $X \to Y$.  That is, $D[f]: X \to [X,Y]$.  However, we don't want to assume that our category has closed structure.  Thus, uncurrying, we get that the derivative should be of the type $D[f]: X \times X \to Y$.  The second coordinate is simply the point at which the derivative is being taken, while the first coordinate is the direction in which this derivative is being evaluated.  With this understanding, the first five axioms of a cartesian differential category should be relatively clear.  Axioms 6 and 7 are slightly more tricky, but in essence they say that the derivative is linear in its first variable, and that the order of partial differentiation does not matter.  For more discussion of these axioms, see \cite{cartDiff}.
\begin{definition}
A \textbf{cartesian differential category} is a cartesian left additive category with a differentiation operation
$$\infer{X \x X \to_{D[f]} Y}{X \to^f Y}$$
such that 
\begin{enumerate}[{\bf [CD.1]}]
\item $D[f+g] = D[f]+D[g]$ and $D[0]=0$ (additivity of differentiation);
\item $\< g+h,k\>D[f] = \< g,k\>D[f] + \< h,k\>D[f]$ and $\< 0,g\>D[f] = 0$ (additivity of a derivative in its first variable);
\item $D[1] = \pi_0, D[\pi_0] = \pi_0\pi_0$, and $D[\pi_1] = \pi_0\pi_1$ (derivatives of projections);
\item $D[\<f,g\>] = \< D[f],D[g]\>$ (derivatives of pairings);
\item $D[fg] = \<D[f],\pi_1f \>D[g]$ (chain rule);
\item $\<\<g,0\>,\<h,k\>\>D[D[f]] = \<g,k\>D[f]$ (linearity of the derivative in the first variable);
\item $\< \< 0,h\>,\< g,k\>\> D[D[f]] = \<\< 0,g\>,\< h,k\>\> D[D[f]]$ (independence of partial differentiation).
\end{enumerate}
\end{definition}

We now give the definition of a differential restriction category.  Axioms 8 and 9 are the additions to the above.  Axiom 8 says that the differential of a restriction is similar to the derivative of an identity, with the partiality of $f$ now included.  Axiom 9 says that the restriction of a differential is nothing more than $1 \times \rst{f}$: the first component, being simply the co-ordinate of the direction the derivative is taken, is always total.  In addition to these new axioms, one must also modify axioms 2 and 6 to take into account the partiality when one loses maps, and remove the first part of axiom 3 ($D[1] = \pi_0$), since axiom 8 makes it redundant.  

\begin{definition}
A \textbf{differential restriction category} is a cartesian left additive restriction category with a differentiation operation
$$\infer{X \x X \to_{D[f]} Y}{X \to^f Y}$$
such that 
\begin{enumerate}[{\bf [DR.1]}]
\item $D[f+g] = D[f]+D[g]$ and $D[0]=0$;
\item $\< g+h,k\>D[f] = \< g,k\>D[f] + \< h,k\>D[f]$ and $\< 0,g\>D[f] = \rst{gf}0$;
\item $D[\pi_0] = \pi_0\pi_0$, and $D[\pi_1] = \pi_0\pi_1$;
\item $D[\<f,g\>] = \< D[f],D[g]\>$;
\item $D[fg] = \<D[f],\pi_1f \>D[g]$;
\item $\<\<g,0\>,\<h,k\>\>D[D[f]] = \rst{h} \<g,k\>D[f]$;
\item $\< \< 0,h\>,\< g,k\>\> D[D[f]] = \<\< 0,g\>,\< h,k\>\> D[D[f]]$;
\item $D[\rst{f}] = (1 \times \rst{f})\pi_0$;
\item $\rst{D[f]} = 1 \x \rst{f}$.
\end{enumerate}
\end{definition}

Of course, any cartesian differential category is a differential restriction category, when equipped with the trivial restriction structure ($\rs{f} = 1$ for all $f$).  The standard example with a non-trivial restriction is smooth functions defined on open subsets of $\R^n$; that this is a differential restriction category is readily verified.  In the next section, we will present a more sophisticated example (rational functions over a commutative ring).

There is an obvious notion of differential restriction functor:

\begin{definition}
If $\X$ and $\Y$ are differential restriction categories, a \textbf{differential restriction functor} $\X \to^F \Y$ is a restriction functor such that
\begin{itemize}
	\item $F$ preserves the addition and zeroes of the homsets;
	\item $F$ preserves products strictly: $F(A \times B) = FA \times FB, F1 = 1$, as well as pairings and projections,
	\item $F$ preserves the differential: $F(D[f]) = D[F(f)]$.
\end{itemize}
\end{definition}  

The differential itself automatically preserves both the restriction ordering and the compatibility relation:

\begin{proposition}\label{propDiff} In a differential restriction category:
\begin{enumerate}[(i)]
\item $D[\rst{f}g] = (1 \x \rst{f}) D[g]$;
\item If $f \leq g$ then $D[f] \leq D[g]$;
\item If $f \smile g$ then $D[f] \smile D[g]$.
\end{enumerate}
\end{proposition}

\begin{proof}
\begin{enumerate}[(i)]
\item Consider:
\begin{eqnarray*}
&   & D[\rst{f}g] \\
& = & \< D \rs{f}, \pi_1 \rs{f} \> D[g] \mbox{ by [\dr{5}]} \\
& = & \< (1 \times \rs{f}) \pi_0 , \pi_1 \rs{f} \> D[g] \mbox{ by [\dr{8}]} \\
& = & \< (1 \times \rs{f}) \pi_0 , (1 \times \rs{f}) \pi_1  \> D[g] \mbox{ by
naturality} \\
& = & (1 \times \rs{f})D[g] \mbox{ by Lemma \ref{propCart}.}
\end{eqnarray*}
as required.  
\item If $f \leq g$, then
    \[ \rs{D[f]} D[g] = (1 \times \rs{f}) D[g] = D[\rs{f}g] = D[f], \]\
      so $D[f] \leq D[g]$.
\item If $f \smile g$, then
    \[ \rs{D[f]} D[g] = (1 \times \rs{f}) D[g] = D[\rs{f}g] = D[\rs{g}f] = (1
\times \rs{g})D[f] = \rs{D[g]}D[f], \]
      so $D[f] \smile D[g]$.  
\end{enumerate}
\end{proof}

Moreover, just as for cartesian and left additive structure, if $\X$ has joins and differential structure, then they are automatically compatible:

\begin{proposition}\label{propjoindiffcompat} In a differential restriction category with joins, 
\begin{enumerate}[(i)]
	\item $D[\emptyset] = \emptyset$,
	\item $D \left[ \bigvee_i f_i \right] = \bigvee_i D[f_i]$.
\end{enumerate}
\end{proposition}
\begin{proof}
\begin{enumerate}[(i)]
	\item $\rs{D[\emptyset]} = 1 \times \rs{\emptyset} = \emptyset$, so by Lemma \ref{propJoins}, $D[\emptyset] = \emptyset$.
	\item Consider:
	\begin{eqnarray*}
		&  & \bigvee_{i \in I} D[f_i] \\
		& = & \bigvee_{i \in I} D\left[\rs{f_i} \bigvee_{j \in I} f_j \right] \mbox{ by Lemma \ref{propJoins}} \\
		& = & \bigvee_{i \in I} (1 \times \rs{f_i}) D\left[ \bigvee_{j \in I} f_j \right] \mbox{ by Lemma \ref{propDiff}} \\
		& = & \left( 1 \times \bigvee_{i \in I} \rs{f_i} \right) D \left[ \bigvee_{j \in I} f_j \right] \\
		& = & \left( 1 \times \rs{\bigvee_{i \in I} f_i} \right) D \left[ \bigvee_{j \in I} f_j \right] \\
		& = & \rs{D \left[ \bigvee_{i \in I} f_i \right]} D \left[ \bigvee_{j \in I} f_j \right] \mbox{ by \dr{9}} \\
		& = & D \left[ \bigvee_{i \in I} f_i \right]
	\end{eqnarray*}
	as required.  
\end{enumerate}
\end{proof}

\subsection{Linear maps}

Just as we had to modify the definition of additive maps for left additive restriction categories, so too do we have to modify linear maps when dealing with differential restriction categories.  Recall that in a cartesian differential category, a map is linear if $D[f] = \pi_0f$.  If we asked for this in a differential restriction category, we would have
	\[ \rs{\pi_0f} = \rs{D[f]} = 1 \times \rs{f} = \rs{\pi_1f}, \]
which is never true unless $f$ is total.  In contrast to the additive situation, however, there is no obvious preference for one side to be more defined that the other.  Thus, a map will be linear when $D[f]$ and $\pi_0f$ are compatible.  

\begin{definition}
A map $f$ in a differential restriction category is {\bf linear} if 
	\[ D[f] \smile \pi_0 f \]
\end{definition}

We shall see below that for total $f$, this agrees with the usual definition.  We also have the following alternate characterizations of linearity:

\begin{lemma}
In a differential restriction category,
\begin{eqnarray*}
	&   & f \mbox{ is linear} \\
	& \Leftrightarrow & \rst{\pi_1f}\pi_0f \leq D[f] \\
	& \Leftrightarrow &  \rst{\pi_0f}D[f] \leq \pi_0f
\end{eqnarray*}
\end{lemma}
\begin{proof}
	Use the alternate form of compatibility (Lemma \ref{lemmaAltComp}).
\end{proof}

Linear maps then have a number of important properties.  Note one surprise: while additive maps were not closed under partial inverses, linear maps are.  

\begin{proposition}\label{propAdd}
In a differential restriction category:
\begin{enumerate}[(i)]
	\item if $f$ is total, $f$ is linear if and only if $D[f] = \pi_0f$;
	\item if $f$ is linear, then $f$ is additive;
	\item restriction idempotents are linear;
	\item if $f$ and $g$ are linear, so is $fg$;
	\item if $g \leq f$ and $f$ is linear, then $g$ is linear;
	\item $0$ maps are linear, and if $f$ and $g$ are linear, so is $f+g$;
	\item projections are linear, and if $f$ and $g$ are linear, so is $\<f,g\>$;
	\item $\<1,0\>D[f]$ is linear for any $f$;
	\item if $f$ is linear and has a partial inverse $g$, then $g$ is also linear.
\end{enumerate}
\end{proposition}
\begin{proof}
\begin{enumerate}[(i)]
	\item It suffices to show that if $f$ is total, $\rs{D[f]} = \rs{\pi_0f}$.  Indeed, if $f$ is total,
		\[ \rs{D[f]} = 1 \times \rs{f} = \rs{f} \times 1 = \rs{\pi_0f}. \]
	\item For the 0 axiom:
	\begin{eqnarray*}
	0f & = & \rs{0f}0f \\
	& = & \rs{\<0,0\>\pi_1f}\<0,0\>\pi_0f \\
	& = & \<0,0\>\rs{\pi_1f}\pi_0f \mbox{ by \rfour,} \\
	& \leq & \<0,0\>D[f] \mbox{ since $f$ linear,} \\
	& = & \rs{0f}0 \mbox{ by \dr{2},} \\
	& \leq & 0
	\end{eqnarray*}
	and for the addition axiom:
	\begin{eqnarray*}
	\rst{(x+y)f)}(xf + yf) & = & \rst{(x+y)f} (\rst{xf}xf + \rst{y}\rst{xf}\rst{x}yf) \\
			       & = & \rst{(x+y)f} (\rst{xf}xf + \rst{\rst{y}xf}\rst{x}yf) \\
			       & = & \rst{(x+y)f} (\rst{\<x,x\>\pi_1f}\<x,x\>\pi_0f + \rst{\<y,x\>\pi_1f}\<y,x\>\pi_0f) \\
			       & = & \rst{(x+y)f} (\<x,x\>\rst{\pi_1f} \pi_0f + \<y,x\>\rst{\pi_1f} \pi_0f) \\
			       & \leq & \rst{(x+y)f} (\<x,x\>D[f] + \<y,x\>D[f]) \mbox{ since $f$ is linear} \\
			       & = & \rst{\<x+y,x\>\pi_0f} \<x+y,x\>D[f] \mbox{ by \dr{2}} \\
			       & = & \<x+y,x\>\rst{\pi_0f} D[f] \\
			       & = & \<x+y,x\>\rst{\pi_1f} \pi_0f \mbox{ since $f$ is linear} \\
			       & = & \rst{x+y,x\>\pi_1f}\<x+y,x\> \pi_0f \\
			       & = & \rst{\rst{x+y}xf}\rst{x}(x+y)f \\
			       & \leq & (x+y)f
	\end{eqnarray*}
	as required.
	\item Suppose $e = \rs{e}$.  Then consider
		\begin{eqnarray*}
		&   & \rs{\pi_1e}\pi_0 \rs{e} \\
		& = & \rs{\pi_1e}\rs{\pi_0 e} \pi_0 \\
		& \leq & \rs{\pi_1e}\pi_0 \\
		& = & \<\pi_0e,\pi_1e\>\pi_0 \\
		& = & (1 \times e)\pi_0 \\
		& = & D[e]
		\end{eqnarray*}
	so that $e$ is additive.
	\item Suppose $f$ and $g$ are linear; then consider
		\begin{eqnarray*}
		D[fg] & = & \<D[f],\pi_1f\>D[g] \\
		      & \geq & \<\rst{\pi_1f}\pi_0f, \pi_1f\>\rst{\pi_1g} \pi_0g \mbox{ since $f$ and $g$ are linear} \\
		      & = & \rst{\<\rst{\pi_1f}\pi_0f, \pi_1f\> \pi_1g} \<\rst{\pi_1f}\pi_0f, \pi_1f\> \pi_0g \mbox{ by \rfour} \\
		      & = & \rst{\rst{\pi_1f} \rst{\pi_0f} \pi_1 fg} \rst{\pi_1f} \rst{\pi_1f} \pi_0fg \\
		      & = & \rst{\pi_1f}\rst{\pi_1fg}\rst{\pi_0f}\pi_0 fg \\
		      & = & \rst{\pi_1 fg} \pi_0 fg 
		\end{eqnarray*}	
	\item If $g \leq f$, then $g = \rs{g}f$; since restriction idempotents are linear and the composite of linear maps is linear, $g$ is linear.
	\item Since $D[0] = 0 = \pi_00$, $0$ is linear.  Suppose $f$ and $g$ are linear; then consider
	\begin{eqnarray*}
		\rst{\pi_0(f+g)} D[f+g] & = & \rst{\pi_0f + \pi_0g}(D[f] + D[g]) \\
					      & = & \rst{\pi_0f}\rst{\pi_0g} (D[f] + D[g]) \\
					      & = & \rst{\pi_0f}D[f] + \rst{\pi_0g}D[g] \\
					      & = & \rst{\pi_1f}\pi_0f + \rst{\pi_1g}\pi_0g \mbox{ since $f$ and $g$ are linear} \\
					      & = & \rst{\pi_1f}\rst{\pi_1g}\pi_0(f + g) \\
					      & \leq & \pi_0(f+g)
	\end{eqnarray*}
	as required.
	\item By \dr{3}, projections are linear.  Suppose $f$ and $g$ are linear; then consider
		\begin{eqnarray*}
		D[\<f,g\>] & = & \<D[f],D[g] \> \\
	  		   & \geq & \<\rst{\pi_1f}\pi_0f, \rst{\pi_1g}\pi_0g \> \mbox{ since $f$ and $g$ are linear} \\
	  		   & = & \rst{\pi_1f}\rst{\pi_1g} \pi_0\<f,g\> \\
	  		   & = & \rst{\rst{\pi_1f}\pi_1g} \pi_0\<f,g\> \\
	  		   & = & \rst{\rst{\pi_1\rst{f}}\pi_1\rst{g}} \pi_0\<f,g\> \\
	  		   & = & \rst{\pi_1\rst{f}\rst{g}} \pi_0\<f,g\> \mbox{ by \rfour} \\
	  		   & = & \rst{\pi_1 \<f,g\>} \pi_0\<f,g\>
	  	\end{eqnarray*}
	  	as required.
	\item The proof is identical to that for total differential categories:
		\begin{eqnarray*}
		D[\<1,0\>D[f]] & = & \<D[\<1,0\>], \pi_1\<1,0\>\> D[D[f]] \\
		               & = & \< \<\pi_0,0\>, \<\pi_1, 0\>\> D[D[f]] \\
		               & = & \<\pi_0,0\>D[f] \mbox{ by \dr{6}} \\
		               & = & \pi_0\<1,0\>D[f]
		\end{eqnarray*}
		as required.
	\item If $g$ is the partial inverse of a linear map $f$, then
		\begin{eqnarray*}
		D[g] & \geq & (\rs{g} \times \rs{g})D[g] \\
		& = & (gf \times gf)D[g] \\
		& = & (g \times g)(f \times f)D[g] \\
		& = & (g \times g)\<\pi_0f, \pi_1f\>D[g] \\
		& = & (g \times g)\<\rs{\pi_1f}\pi_0f, \pi_1f\>D[g] \\
		& = & (g \times g)\<\rs{\pi_0f}D[f], \pi_1f\>D[g] \mbox{ since $f$ is linear,} \\
		& = & (g \times g)\rs{\pi_0f}\<D[f],\pi_1f\>D[g] \\
		& = & (g \times g)\rs{\pi_0f}D[fg] \mbox{ by \dr{5},} \\
		& = & (g \times g)\rs{\pi_0f}D[\rs{f}] \\
		& = & (g \times g)\rs{\pi_0f}(1 \times \rs{f})\pi_0 \mbox{ by \dr{8},} \\
		& = & \rs{(g \times g)\pi_0f} (g \times g)(1 \times \rs{f})\pi_0 \mbox{ by \rfour,} \\
		& = & \rs{\rs{\pi_1g}\pi_0gf} (g \times g)\pi_0 \\
		& = & \rs{\pi_1g}\rs{\pi_0\rs{g}}\rs{\pi_1g}\pi_0g \\
		& = & \rs{\pi_1g}\pi_0 g
		\end{eqnarray*}
		as required.                
		
\end{enumerate}
\end{proof}

Note that the join of linear maps need not be linear.  Indeed, consider the linear partial maps $2x: (0,2) \to (0,4)$ and $3x: (3,5) \to (9,15)$.   If their join was linear, then it would be additive.  But this is a contradiction, since $2(1.75) + 2(1.75) \ne 3(3.5)$.  However, the join of linear maps is a standard concept of analysis:

\begin{definition}
If $f$ is a finite join of linear maps, say that $f$ is \textbf{piecewise linear}.
\end{definition}

An interesting result from \cite{cartDiff} is the nature of the differential of additive maps.  We get a similar result in our context:
\begin{proposition}
If $f$ is additive, then $D[f]$ is additive and
	\[ D[f] \smile \pi_0\<1,0\>D[f]; \]
if $f$ is strongly additive, then $D[f]$ is strongly additive and
	\[ D[f] \leq \pi_0\<1,0\>D[f]. \]
\end{proposition}
\begin{proof}
The proof that $f$ being (strongly) additive implies $f$ (strongly) additive is the same as for total differential categories (\cite{cartDiff}, pg. 19) with $\smile$ or $\leq$ replacing $=$ when one invokes the additivity of $f$.  The form of $D[f]$ in each case, however, takes a bit more work.  We begin with a short calculation:
	\[ \<0,\pi_1\>\rs{\pi_1f} = \rs{\<0,\pi_1\> \pi_1f} \<0,\pi_1\> = \rs{\pi_1f} \<0,\pi_1\> \]
and
	\[ \<\pi_0,0\> \rs{\pi_1f} = \rs{\<\pi_0,0\> \pi_1 f } \<\pi_0,0\> = \rs{0f}\<\pi_0,0\>. \]
Now, if $f$ is additive, we have:
	\begin{eqnarray*}
	&   & \rs{\pi_0 \<1,0\>D[f]} D[f] \\
	& = & \rs{\<\pi_0,0\>\pi_1f} D[f] \\
	& = & \rs{0f}\rs{\<\pi_0,0\>} D[f] \mbox{ by the second calculation above,} \\
	& = & \rs{0f}D[f] \\
	& = & \rs{0f}\rs{\pi_1f}D[f] \\
	& = & \rs{0f}\rs{\pi_1f}(\<0,\pi_1\> + \<\pi_0,0\>)D[f] \\
	& = & (\rs{\pi_1f}\<0,\pi_1\> + \rs{0f}\<\pi_0,0\>)D[f] \\
	& = & (\<0,\pi_1\>\rs{\pi_1f} + \<\pi_0,0\>\rs{\pi_1f})D[f] \mbox{ by both calculations above,} \\
	& \leq & \<0,\pi_1\>D[f] + \<\pi_0,0\>D[f] \mbox{ since $D[f]$ is additive,} \\
	& = & \rs{\pi_1f}0 + \<\pi_0,0\>D[f] \mbox{ by \dr{2},} \\
	& \leq & 0 + \<\pi_0,0\>D[f] \\
	& = & \pi_0\<1,0\>D[f]
	\end{eqnarray*}
so that $D[f] \smile \pi_0 \<1,0\>D[f]$, as required.  If $f$ is strongly additive, consider
	\begin{eqnarray*}
	&   & \rs{D[f]}\<\pi_0,0\>D[f] \\
	& = & \rs{\pi_1f}\<\pi_0,0\>D[f] \\
	& = & \rs{\pi_1f}0 + \<\pi_0,0\>D[f] \\
	& = & \<0,\pi_1\>D[f] + \<\pi_0,0\>D[f] \\
	& = & (\<0,\pi_1\>\rs{\pi_1f} + \<\pi_0,0\>\rs{\pi_1f}) D[f] \mbox{ since $D[f]$ is strongly additive,} \\
	& = & \rs{\pi_1f}\rs{0f}(\<0,\pi_1\> + \<\pi_0,0\>)D[f] \mbox{ by the calculations above,} \\
	& = & \rs{\pi_1f}\rs{0}(1)D[f] \mbox{ since $f$ strongly additive,} \\
	& = & \rs{\pi_1f}D[f] \\
	& = & D[f]
	\end{eqnarray*}
so that $ D[f] \leq \pi_0\<1,0\>D[f]$, as required.
\end{proof}

%
%

Any differential restriction category has the following differential restriction subcategory:
\begin{proposition}
If $\X$ is a differential restriction category, then $\X_0$, consisting of the maps which preserve 0 if it is in their domain (i.e., satisfying $0f \leq 0$), is a differential restriction subcategory.
\end{proposition}
\begin{proof}
The result is immediate, since the differential has this property:
 \[\<0,0\>D[f] = \rs{0f}0 \leq 0.  \]
\end{proof}

Finally, note that any differential restriction functor preserves additive, strongly additive, and linear maps:
\begin{proposition}\label{diffFunctors}
If $F$ is a differential restriction functor, then
\begin{enumerate}[(i)]
	\item $F$ preserves additive maps;
	\item $F$ preserves strongly additive maps;
	\item $F$ preserves linear maps.
\end{enumerate}
\end{proposition}
\begin{proof}
Since any restriction functor preserves $\leq$ and $\smile$, the result follows automatically.
\end{proof}

\section{Rational functions}\label{rat}



Thus far, we have only seen a single, analytic example of a differential restriction category.  This section rectifies this situation by presenting 
a class of examples of differential restriction categories with a more algebraic flavour.  Rational functions over a commutative ring have an obvious 
formal derivative.  Thus, rational functions are a natural candidate for differential structure.  Moreover, rational functions have an aspect of 
partiality: one thinks of a rational function as being undefined at its poles -- that is wherever the denominator is zero.  

To capture this partiality,  we provide a very general construction of rational functions from which we extract a 
(partial) Lawvere theory of rational functions for any commutative rig and whence, in particular, for any commutative ring. We will 
then show that, for each commutative ring $R$, this category of rational functions over $R$ is a differential restriction category.  

Moreover, we will also show that these categories of rational functions embed into the partial map category of affine schemes with respect to 
localizations.  Thus, we relate these categories of rational functions to categories which are of traditional interest in algebraic geometry.


\subsection{The fractional monad}
\label{rat:frac}


In order to provide a general categorical account of rational functions,  it is useful to first have a monadic construction for fractions.  
When a construction is given by a monad, not only can one recover substitution -- as composition in the Kleisli category -- but also one has the whole 
category of algebras in which to interpret structures.   The main difficulty with the construction of fractions is that, to start with, one has to find both 
an algebraic interpretation of the construction, and a setting where it becomes a monad.  

A formal fraction is a pair $(a,b)$, which one thinks of as $\frac{a}{b}$, with addition and multiplication defined as expected for fractions.  If one 
starts with a commutative ring and one builds these formal fractions the very first peculiarity one encounters is that, to 
remain algebraic, one must allow the pair $(a,0)$ into the construction: that is one must allow division by zero.  Allowing division by zero, $\frac{a}{0}$, introduces a 
number of problems. For example, because $\frac{a}{0} + \frac{-a}{0} = \frac{0}{0}$ and not, as one would like, $\frac{0}{1}$, one loses negatives.    
One can, of course, simply abandon negatives and  settle for working with commutative rigs.  However, this does not resolve all the problems. Without cancellation, fractions 
under the usual addition and multiplication will not be a rig: binary distributivity of multiplication over addition will fail -- as will the nullary distribution 
(that is $0 \cdot x=0$).  Significantly, to recover the binary distributive law,  requires only a limited ability to perform cancellation: one needs precisely the equality 
$\frac{a}{a^2} = \frac{1}{a}$.  By imposing this equality, one can recover,  from the fraction construction applied to a rig, a {\em weak\/} rig  -- weak because the 
nullary distributive law has been lost (although the equalities $0 \cdot 0 = 0$ and $0 \cdot x = 0 \cdot x^2$ are retained).  As we shall show below, this construction 
of fractions does then produce a monad on the category of weak rigs.   Furthermore, the algebras for this monad, {\em fractional rigs\/}, can be used to provide 
a general description of rational functions.

An algebraic structure, closely related to our notion of a fractional rig, which was proposed in order to solve very much the same sort of problems, is that of a 
{\em wheel\/} \cite{wheel}.  Wheels also arise from formal fraction constructions, but the equalities imposed on these fractions is formulated differently.  In particular, this means 
that the monadic properties over weak rigs -- which are central to the development below -- do not have a counterpart for wheels.  Nonetheless, the theory developed here has many 
parallels in the theory of wheels. Certainly the theory of wheels illustrates the rich possibilities for algebraic structures which can result from allowing division by zero,
and there is a nice discussion of the motivation for studying such structures in \cite{wheel}.  

Technically a wheel, as proposed in \cite{wheel}, does not satisfy the binary distributive law (instead, it satisfies $(x+y)z +0z = xz + yz$ -- where notably $0z \not= 0$ in general) and 
in this regard it is a weaker notion than a fractional rig.  A wheel also has an involution with $x^{**} = x$, while fractional rigs have a star operation satisfying
the weaker requirement $x^{***} = x^{*}$.  Thus, the structures are actually incomparable, although they certainly have many common features.

A {\bf weak commutative rig} $R = (U(R),\cdot,+,1,0)$ (where $U(R)$ is the underlying set) is a set with two commutative associative operations, $\_\cdot\_$ with 
unit $1$, and $\_ + \_$ with unit $0$ which satisfies the binary distributive law $x \cdot (y + z) = x \cdot y + x \cdot z$, has $0 \cdot 0 = 0$, and 
$0 \cdot x = 0 \cdot x \cdot x$ (but in which the nullary distributive law fails, so in general $x \cdot 0 \not= 0$).  Weak rigs with evident homomorphisms 
form a category ${\bf wCRig}$. 

For convenience, when manipulating the terms of a (weak) rig, we shall tend to drop the multiplication symbol, writing $x \cdot y$ simply by juxtaposition as $xy$.

Notice that there is a significant difference between a weak rig and a rig: a weak rig $R$ can have a non-trivial ``zero ideal'', $R_0 = \{ 0r |r \in R \}$.  
Clearly $0 \in R_0$, and it is closed to addition and multiplication.  In fact, $R_0$ itself is a weak rig with the peculiar property that $0=1$.  
To convince the reader that weak rigs with $0=1$ are a plausible notion, consider the natural numbers with the addition {\em and\/} multiplication given by maximum: 
this is a weak rig in which necessarily the additive and multiplicative units coincide.   The fact that, in this example, the addition and multiplication are the same   
is not a coincidence: 

\begin{lemma}
\label{zero-ideal}
In a weak commutative rig $R$, in which $0=1$, we have:
\begin{enumerate}[(i)]
\item Addition and multiplication are equal: $x+y = x y$;
\item Addition -- and so multiplication -- is idempotent, making $R$ into a join semilattice (where $x \leq y$ if $x+y=y$).
\end{enumerate}
\end{lemma}

\proof 
When $1=0$, to show that $x+y = x y$ it is useful to first observe that both are addition and multiplication are idempotent:
$$x+x = x0+x0 = x(0+0) = x0 = x ~~~\mbox{and}~~~ xx = 0x0x = 0xx=0x=x.$$
Now we have the following calculation:
\begin{eqnarray*}
x+y & = & (x+y)(x+y) = x^2 +xy +yx +y^2 = x + xy +y \\
       & = & x(1 +y) +y = x(0+y) + y = xy + y \\
       & = & (x+1)y = (x+0)y = xy 
\end{eqnarray*} 
Thus, one now has a join semilattice determined by this operation.
\endproof

Both rigs and rings, of course, are weak rigs in which $R_0 = \{ 0 \}$.  

Define the {\bf fractions}, $\fr(R)$, of a weak commutative rig $R$ as the set of pairs $U(R) \x U(R)$ modulo the equivalence relation generated by
$(r,as) \sim (ar,a^2s)$, with the following ``fraction'' operations:
\begin{center}
\begin{tabular}{ll}
$(r,s) + (r',s') := (r s' + s r',s s'),$~~~~~~~ & $0 := (0,1),$ \\
$(r,s) \cdot (r',s') := (r r',s s'),$ &  $1 := (1,1).$
\end{tabular}
\end{center}

Now it is not at all obvious that this structure is, with this equivalence, a weak commutative rig. To establish this, it is useful to analyze the equivalence 
relation more carefully.

We shall, as is standard, write $a | s$ to mean $a$ divides $s$, in the sense that there is an $s'$ with $s = a s'$. 
We may then write the generating relations for the equivalence above as $(r,s) \rightarrowtriangle_a (r',s')$ where $r' =a r$, $s' = a s$ 
and $a | s$.  Furthermore, we shall say $a$ {\bf iteratively divides} $r$, written 
$a |^{*} r$, in case there is a decomposition $a = a_1  \dots a_n$ such that  $a_1 | a_2 \dots  a_n  r$, and  
$a_2 | a_3 \dots a_n \cdot r$, and ... , and $a_n | r$. Then define $(r,s) \rightarrowtriangle^{*}_a (r',s')$
to mean $r' = a r$, $s' = a s$ and $a |^{*} s$.  

Observe that to say $(r,s) \rightarrowtriangle^*_a (r',s')$ is precisely to say there is a decomposition $a=a_1  \dots  a_n$ such that 
 $$(r,s) \rightarrowtriangle_{a_n} (a_n r, a_n  s) \rightarrowtriangle_{a_{n-1}} \dots 
           \rightarrowtriangle_{a_1} (a_1 \dots a_n  r,a_1 \dots a_n r) = (r',s').$$
Thus $\_ \rightarrowtriangle^{*} \_$ is just the transitive reflexive closure of $\_ \rightarrowtriangle \_$, the generating relation of the equivalence.

Next, say that $(r_1,s_1) \sim (r_2,s_2)$ if and only if there is a $(r_0,s_0)$ and $a,b \in R$ such that 
$$\xymatrix@=10pt{ & (r_0,s_0) & \\ (r_1,s_1) \ar[ur]_{*}^a & & (r_1,s_1) \ar[ul]^{*}_b }.$$
Then we have:

\begin{proposition}
For any weak commutative rig, the relation $(r_1,s_1) \sim (r_2,s_2)$ on $R \x R$ is the equivalence relation generated by $\rightarrowtriangle$.  Furthermore,
it is a congruence with respect to fraction addition and multiplication, turning $R \x R/\sim$ into a weak commutative rig $\fr(R)$.
\end{proposition}

\proof
That $\sim$ contains the generating relations and is contained in the equivalence relation generated by the generating relations is clear.  That it 
is symmetric and reflexive is also clear.  What is less clear is that it is transitive: for that we need the transitivity of $\rightarrowtriangle^*$ -- which is 
immediate -- and the ability to pushout the generating relations with generating relations:
$$\xymatrix@=10pt{ & (r,s) \ar[dl]_a \ar[dr]^b \\ 
             (a r,a s) \ar[dr]_b & & (b r,b s)  \ar[dl]^a \\ 
             & (a b r, a  b s) }$$
This shows that it is an equivalence relation.

To see that this is a congruence with respect to the fraction operations it suffices (given symmetries) to show that if 
$(r,s) \rightarrowtriangle_a (a r,a s) = (r',s')$ that $(r,s) + (p,q) \rightarrowtriangle_a (r',s') + (p,q)$, and similarly for multiplication.
That this works for multiplication is straightforward. For addition we have:
$$(r,s) + (p,q) = (r q + s p, s q) \rightarrowtriangle_a (a (r q + s p),a s q)  = (a r,a s) + (p,q)= (r',s') + (p,q)$$
where $a | s q$ as $a | s$.

Finally, we must show that this is a weak commutative rig.  It is clear that the multiplication has unit $(1,1)$, and is commutative and associative.
Similarly for addition it is clearly commutative, the unit is $(0,1)$ as:
\begin{eqnarray*}
(0,1)+(a,b) & =  & (0b+a,b) \rightarrowtriangle_b ((0b+a)b,b^2) = (0b^2 + ab,b^2) = (0b+ab,b^2) \\
                     & = & ((0+a)b,b^2) = (ab,b^2) \leftarrowtriangle_b (a,b).
\end{eqnarray*}
Furthermore, $(0,1) (0,1) = (0,1)$ and  $(0,1)(a,b)= (0,1)(a,b)^2$ as:
$$(0,1)(a,b) = (0a,b) \rightarrowtriangle_{b^2} (0ab^2,b^3) = (0a^2b,b^3) \leftarrowtriangle_b (0a^2,b^2) = (0,1)(a,b)^2.$$
That addition is associative is a standard calculation.  The only other non-standard aspect is binary distributivity:
\begin{eqnarray*}
(a,b) ((c,d)+(e,f)) & = & (a,b) (c  f + d e,d f) \\
  & = & (a c f + a d e,b d f) \\
  & \rightarrowtriangle_b & (a b  c  f + a  b  d  e,b^2  d f) \\
  & = & (a c,b d) + (a e,b f) \\
  & = & (a,b) (c,d) + (a,b) (e,f).
\end{eqnarray*}
\endproof

Notice that forcing binary distributivity to hold implies
$$(1,b) = (1,b) ((1,1)+(0,1)) \equiv (1,b) (1,1) + (1,b) (0,1) = (1,b)+(0,b) = (b+0b,b^2)=(b,b^2)$$
so that the generating equivalences above {\em must} hold, when distributivity is present.  This means we are precisely forcing binary distributivity of multiplication over
fraction addition with these generating equivalences.  Note also that nullary distributivity, even when one starts with a rig $R$, will not hold 
in $\fr(R)$, as $(0,0) = (0,0)(0,1)$ and $(0,1)$ are distinct unless $0=1$.

It is worth briefly considering some examples:
\begin{enumerate}[(1)]
\item Any lattice $L$ is a rig. $\fr(L)$ has as its underlying set pairs $\{ (x,y) | x,y \in L, x \leq y \}$ as, in this case, $(x,y)\sim (x \wedge y,y)$.
These are the set of intervals of the lattice.  The resulting addition and multiplication are both idempotent and are, respectively, the join and meet 
for two different ways of ordering intervals.  For the multiplication the ordering is $(x,y) \leq (x',y')$ if and only if $x \leq x'$ and $y \leq y'$.  For the addition 
the ordering\footnote{This order is known as the ``modal interval inclusion'' in the rough set literature and the meet with respect to this order is a 
well-known database operation related to ``left outer joins''!} is  $(x,y) \leq (x',y')$ if and only if $x' \wedge y \leq x$ and $y \leq y'$.  

Notice that the zero ideal consists of all intervals $(\bot,a)$.
\item In any unique factorization domain, $R$, such as the integers or any polynomial ring over a unique factorization domain, the equality in $\fr(R)$ 
may be expressed by reduction (as opposed to the expansion given above). This reduction to a canonical form performs cancellation while the factor is not 
eliminated from the denominator.  Thus, in $\fr(\Z)$ we have $(18,36)$ reduces to $(3,6)$ but no further reduction is allowed as this would eliminate a factor 
(in this case $3$) from the denominator.

In any rig $R$, as zero divides zero, we have $(r,0) = (0,0)$ for every $r \in R$.  The zero ideal will, in general, be  
quite large as it is $\fr{R}_0 = \{ (0,r)| r \in R\}$.
\item A special case of the above is when $R$ is a field. In this case $(x,y) = (xy^{-1},1)$ when $y \not= 0$ and when $y=0$ 
then, as above, $(x,0) = (0,0)$.  Thus, in this case the construction adds a single point ``at infinity'', $\infty = (0,0)$.  Note that the zero 
ideal is $\{ 0,\infty \}$.
\item The initial weak commutative rig is just the natural numbers, $\N$.  Thus, it is of some interest to know what 
$\fr(\N)$ looks like as this will be the initial algebra of the monad.  The canonical form of the elements is, as for unique factorization domains, determined by canceling 
factors from the fractions while the denominator remains divisible by that factor. Addition and multiplication are performed as usual 
for fractions and then reduced by canceling in this manner to the canonical form.   The zero ideal consists of $(0,0)$ and element of the form $(0,p_1p_2...p_n)$ where 
the denominator is a (possibly empty) product of distinct primes.
\end{enumerate}

Clearly we always have a weak rig homomorphism:
$$\eta: {\cal R} \to \fr({\cal R}); r \mapsto (r,1)$$  
Furthermore this is always a faithful embedding: if $(r,1) \sim (s,1)$, then we have $(u,v) \rightarrowtriangle^*_a (r,1)$ and  
$(u,v) \rightarrowtriangle^*_b (s,1)$.  This means $a$ iteratively divides $1$ but this means $a = a_1 \cdot \dots \cdot a_n$ where 
$a_n | 1$ which, in turn, means $a_n$ is a unit (i.e. has an inverse). But now we may argue similarly for $a_{n-1}$ and this eventually 
gives that $a$ itself is a unit.  Similarly $b$ is a unit and as $a \cdot v = 1 = b \cdot v$ it follows $a=b$ and whence that $r=s$.

In order to show that $\fr$ is a monad, we will use the ``Kleisli triple'' presentation of a monad.  For this we need a combinator 
$$\infer{\#(f): \fr(R) \to \fr(S)}{f : R \to \fr(S)}$$
such that $\#(\eta) = 1$, $\eta \#(f) = f$ and $\#(f)\#(g) = \#(f\#(g))$. Recall that given this, the functor is defined by $\fr(f) := \#(f \eta)$ and  
the multiplication is defined by $\mu_X := \#(1_{\fr(X)})$.  

We define this combinator as $\#(f)(x,y) := [(x_1 y_2^2,x_2 y_1 y_2)]$, where $[(x_1,x_2)] = f(x)$ and $[(y_1,y_2)] = f(y)$.  To simplify 
notation we shall write $(x_1,x_2) \in f(x)$, rather than $[(x_1,x_2)] = f(x)$, to mean $(x_1,x_2)$ is in the equivalence class determined by 
$f(x)$. 

Our very first problem is to prove that this is well-defined.  That is if $(x,y) \sim (x',y')$ that $\#(f)(x,y) \sim \#(f)(x',y')$ and as this 
is a little tricky we shall give an explicit proof.  First note that it suffices to prove this for a generating equivalence: so we may assume 
that $(x,y) \rightarrowtriangle_a (x',y')=(ax,ay)$ (where this also means $a | y$) and we must prove that $\#(f)(x,y) \sim \#(f)(ax,ay)$.  
Now $\#(f)(ax,ay) = (x'_1(y'_2)^2,x'_2y'_1y'_2)$ where $(x'_1,x'_2) \in f(ax)$ and $(y'_1,y'_2) \in f(ay)$. But we have $f(ax) \sim f(a)f(x)$ and 
$f(ay) \sim f(a)f(y)=f(a)f(a)f(z)$ where $y=az$ thus, letting $(a_1,a_2) \in f(a)$ and $(z_1,z_2) \in f(z)$, there are $\alpha$,$\beta$,$\gamma$, 
and $\delta$ such that 
$$\xymatrix@=10pt{ & (\beta a_1 x_1,\beta a_2 x_2) & & & (\delta a_1 a_1 z_1, \delta a_2 a_2 z_2) \\
            (x'_1,x'_2) \ar[ur]^{\alpha} & & (a_1 x_1,a_2 x_2) \ar[ul]_{\beta} & 
            (y'_1,y'_2)  \ar[ur]^{\gamma} & & (a_1 a_1 z_1,a_2 a_2 z_2)  \ar[ul]_{\delta} }$$
we may now calculate:
\begin{eqnarray*}
(x'_1 {y'_2}^2,y'_1y'_2x'_2) & \rightarrowtriangle_\alpha & (\beta  a_1 x_1  {y'_2}^2,y'_1y'_2 \beta a_2 x_2) \\
   & \rightarrowtriangle_\gamma & (\beta  a_1 x_1 \delta a_2 a_2 z_2 y'_2,\delta a_1 a_1 z_1 y'_2 \beta a_2 x_2) \\
   & \rightarrowtriangle_\gamma & (\beta  a_1 x_1 \delta a_2 a_2 z_2 \delta a_2 a_2 z_2,\delta a_1 a_1 z_1 \delta a_2 a_2 z_2 \beta a_2 x_2) \\
   & = & (\beta \delta^2 a_1 x_1 a_2^4 y_2^2,\beta \delta^2 a_1^2 a_2^3 z_1 z_2 x_2) \\
   & \leftarrowtriangle_{\beta\delta^2a_2^2a_1} &  (x_1 (a_2 z_2)^2,(a_1 z_1)(a_2 z_2) x_2)  \\
   & = & (x_1 y_2^2,y_1 y_2 x_2) \in \#(f)(x,y) 
\end{eqnarray*}
   
This is the first step in proving:

\begin{proposition}
$(\fr,\eta, \mu)$ is a monad, called the {\bf fractional monad}, on ${\bf wCRig}$
\end{proposition}

\proof
It remains to show that $\#(f)$ is a weak rig homomorphism and satisfies the Kleisli triple requirements.
It is straightforward to check that $\#(f)$ preserves the units and multiplication.  The argument for addition 
is a little more tricky. 

First note:
\begin{eqnarray*}
f(rq+sp,sq) & = & [ ( (r_1,r_2)(q_1,q_2) + (s_1,s_2)(p_1,p_2),(s_1,s_2)(q_1,q_2) )] \\
& & \mbox{where}~ (r_1,r_2) \in f(r), (s_1,s_2) \in f(s),(p_1,p_2) \inf(p),(q_1,q_2) \in f(q) \\
                       & = & [( (r_1q_1,r_2q_2) + (s_1p_1,s_2p_2),(s_1q_1,s_2q_2))] \\
                       & = & [(r_1q_1s_2p_2 +r_2q_2s_1p_1,r_2q_2s_2p_2)] \\
f(sq) & = & [(s_1q_1,s_2q_2)] 
\end{eqnarray*} 
so that we now have:
\begin{eqnarray*}
\#(f)([(r,s)]+[(p,q)]) & = & \#(f)([(rq+sp,sq)]) \\
                       & = &[ ((r_1q_1s_2p_2 +r_2q_2s_1p_1)(s_2q_2)^2,s_1p_2q_1r_2(s_2q_2)^2)] \\
                       & \sim & [((r_1q_1s_2p_2+ r_2q_2s_1p_1)s_2q_2,s_1p_2q_1r_2q_2s_2)] \\
                       & = & [(r_1s_2^2p_2q_1q_2+ r_2s_1s_2p_1q_2^2,r_2s_1s_2p_2q_1q_2)] \\
                       & = & [(r_1s_2^2,r_2s_1s_2)] + [(p_1q_2^2,p_2q_1q_2)] \\
                       & = & \#(f)([(r,s)]) + \#(f)([(p,q)])
\end{eqnarray*}

It remains to check the monad identities for the Kleisli triple.  The first two are straightforward we shall illustrate the last identity:
\begin{eqnarray*}
\#(g)(\#(f)([(x,y)])) & = & \#(g)(x_1y_2^2,x_2y_1y_2) ~~~\mbox{where}~~(x_1,x_2) \in f(x), (y_1,y_2) \in g(y) \\
                      & = & [(x_{11}y_{12}^2(x_{22}y_{21}y_{22})^2,x_{21}y_{22}^2x_{12} y_{11}y_{12}x_{22}y_{21}y_{22})] \\
                      & & \mbox{where}~~ (x_{1i},x_{2i}) \in g(x_i), (y_{1j},y_{2j}) \in g(y_j) \\
                      & = & [( x_{11}x_{22}^2(y_{21}y_{22}y_{12})^2,x_{12}x_{22}x_{12}y_{11}y_{22}^2y_{21}y_{22}y_{12})] \\
                      & = & [(x_1'{y'_2}^2,x'_2y_1'y_2')] ~~~\mbox{where}~~~ (x_1',x_2') \in \#(g)(f(x)), (y'_1,y'_2) \in \#(g)(f(y)) \\
& & \mbox{and}~~( x_{11}x_{22}^2,x_{21}x_{22}x_{12}) \in  \#(g)([(x_1,x_2)]) =\#(g)(f(x)) \\
& & ~~~~~~~ (y_{11}y_{22}^2,y_{12}y_{22}y_{12}) \in  \#(g)([(y_1,y_2)]) =\#(g)(f(y)) \\
                      & = & \#(f\#(g))([(x,y)]).
\end{eqnarray*}
\endproof

The algebras for this monad are ``fractional rigs'' as we will now show. A 
{\bf fractional rig} is a weak commutative rig with an operation $(\_)^{*}$ such that 
\begin{itemize}
\item $1^{*} = 1$, $x^{***} = x^{*}$, $(xy)^{*} = y^{*}x^{*}$;
\item $x^{*}xx^{*}= x^{*}$ (that is, $x^{*}$ is regular);
\item $x^{*}x(y+z) = x^{*}xy+z$ (linear distributivity for idempotents).
\end{itemize}
The last axiom is equivalent to demanding $x^{*}xy = x^{*}x0 + y$.  In particular, setting $y=1$, this means that 
$x^{*}x = x^{*}x0 + 1$.

Fractional rigs are of interest in their own right.  Here are some simple observations:

\begin{lemma}
In any fractional rig:
\begin{enumerate}[(i)]
\item $xx^{*}$ is idempotent;
\item If $x$ is a unit, with $xy=1$, then $x^{*}=y$;
\item $x^{*}x^{**}x^{*} = x^{*}$;
\item $xx^{*} =(xx^{*})^{*}$;
\item $e$ is idempotent with $e^{*}=e$ if and only if there is an $x$ with $e=xx^{*}$;
\item $xx^{*}x = x^{**}$;
\item An element $x$ is regular (that is, $xx^{*}x=x$) if and only if $x=x^{**}$.
\item if $0=0^{*}$ then $0=1$, addition equals multiplication, and both operations are idempotent. 
\end{enumerate}
\end{lemma}

We shall call an element $e$  a {\bf ${*}$-idempotent} when $e$ is idempotent and  $e^{*}=e$. 

\proof~
\begin{enumerate}[{\em (i)}]
\item $(xx^{*})(xx^{*}) = x(x^{*}xx^{*}) = xx^{*}$.
\item As $xx^{*}$ is idempotent if it has an inverse it is the identity. However, $yy^{*}$ is its inverse: this means $y=x^{*}$.
\item $x^{*}x^{**}x^{*} = x^{***}x^{**}x^{***} = x^{***} = x^{*}$;
\item $xx^{*} =xx^{*}x^{**}x^{*} = x^{*}xx^{*}x^{**} = x^{*}x^{**} = (xx^{*})^{*}$;
\item If $e$ is idempotent with $e^{*}=e$ then $ee^{*}= ee = e$ and the converse follows from the above.
\item $xx^{*}x = xx^{*}x^{**}x^{*}x = x^{*}xx^{*}x^{**}x = x^{**}x^{*}x = x^{**}x^{*}x^{**}x^{*}x = x^{**}x^{*}x^{**} = x^{**}$;
\item If $x$ is regular in this sense then $x^{**}=xx^{*}x=x$ so $x=x^{**}$ and the converse follows from above.
\item If $0=0^{*}$ then $0=00=00^{*}$ so $0$ is a $*$-idempotent.  This means $1=1+0=0(1+0)=01+ 0 =0$! 
\end{enumerate}
\endproof

In particular, as a consequence of the last observation, it  follows from Lemma \ref{zero-ideal} that for any fractional rig $R$, the fractional rig 
$R_{00^{*}}=\{ 00^{*}r | r \in R\}$, which we discuss further in the next section, is a semilattice.  In fact, we may say more:

\begin{lemma} In any fractional rig in which $0=1$:
\begin{enumerate}[(i)]
\item The addition and multiplication are equal and idempotent, producing a join semilattice;
\item $0=0^{*}$ and $(\_)^{*}$ is a closure operator (that is, it is monotone with $x \leq x^{*}$ and $x^{**}=x^{*}$).
\end{enumerate}
\end{lemma}

\proof  The first part follows from Lemma \ref{zero-ideal}.  For the second part: as $0=1$ and $1^{*}=1$, the first observation is immediate.
Now $x \leq y$ if and only if $y = xy$ but then, as $(\_)^{*}$ preserves multiplication, $y^{*} = x^{*}y^{*}$ so that $x^{*} \leq y^{*}$.  Thus, $(\_)^{*}$ 
is monotone.   $x \leq x^{*}$ if and only if $x^{*}x = x^{*}$ but $x^{*}x = xx^{*}x = x^{**}$, so we are done if we can show $x^{*} = x^{**}$.  
But $x^{*} = x^{*}x^{**}x^{*} = x^{*}x^{**} = x^{**}x^{*}x^{**} = x^{**}x^{***}x^{**} = x^{**}$.
\endproof
      
We observe next that for any weak commutative rig $R$, $\fr(R)$ is a fractional rig with $(\_)^{*}$ defined by $(x,y)^{*} := (y^2,yx)$.  Notice 
first that it is straightforward to check that this is a well-defined operation which is multiplicative and that $x^{***} = x^{*}$. 
Furthermore, $(x,y)^{*}$ is regular in the sense that
$$(x,y)^{*}(x,y)(x,y)^{*} = (y^2,xy)(x,y)(y^2,xy) = (y^4x,y^3x^2) = (y^2,xy)= (x,y)^{*}.$$  
For the linear distribution observe that
$$(z,z)((p,q) + (r,s)) = (z,z)(ps+qr,qs) = (zps+zqr,zqs) = (z,z)(p,q)+(r,s).$$

Note that, for example, in $\fr(\N)$ we have, for any two primes 
$p \not= q$, $(p,q)^{*} = (q^2,pq)$ and $(p,q)^{**}=(p^2q^2,q^3p) = (p^2,pq)$ and $(p,q)^{***} = (p^2q^2,p^3q) = (q^2,pq) = (p,q)^{*}$.  
So here $x^{**} \not= x$ and $x^{**} \not= x^{*}$.   On the other hand, $(0,p)^* = (p^2,0p) = (0,0) = (0,0)^{*} = (0,p)^{**}$, thus the 
``closure'' of everything in the zero ideal is its top element, $(0,0)$.

To show that fractional weak rigs are exactly the algebras for the fractional monad, we need to show that for any fractional rig $R$, there is a structure map 
$\nu: \fr(R) \to R$ such that 
$$\xymatrix{\fr^2(R) \ar[d]_{\fr(\nu)} \ar[r]^\mu & \fr(R) \ar[d]^{\nu} \\ \fr(R) \ar[r]_{\nu} & R}$$
commutes.  Define $\nu: \fr(R) \to R; (r,s) \mapsto rs^{*}$, then:

\begin{lemma} 
For every fractional weak rig, $R$, $\nu$ as defined above is a fractional rig homomorphism.
\end{lemma}

\proof
We must check that $\nu$ is well-defined and is a fractional rig homomorphism.  To establish the former it suffices to prove
$\nu(x,ay) = \nu(ax,a^2y)$ which is so as 
$$\nu(ax,a^2y) = axa^2y({a^{*}}^2y^{*})^2 = xay{a^{*}}^2{y^{*}}^2 = \nu(x,ay).$$
It is straightforward to check that multiplication and the units are preserved by $\nu$.  This leaves addition:
\begin{eqnarray*}
\nu((r,s)+(p,q)) & = & \nu(rq+sp,sq) = (rq+sp)sq{s^{*}}^2{q^{*}}^2 \\
                 & = & s^{*}q^{*}(qq^{*}(rq+sp)ss^{*}) = s^{*}q^{*}(q^2q^{*}r+s^2s^{*}p) \\
                 & = & q^2{q^{*}}^2rs^{*}+s^2{s^{*}}^2pq^{*} = qq^{*}rs^{*}+ss^{*}pq^{*}\\
                 & = & qq^{*}(rs^{*}+pq^{*})ss^{*} = ss^{*}(rs^{*}+pq^{*})qq^{*} \\
                 & = & sr{s^{*}}^2+pq{q^{*}}^2 = \nu(r,s) + \nu(p,q).
\end{eqnarray*}
Finally, $\nu$ preserves the $(\_)^*$ as 
$$\nu((x,y)^{*}) = \nu(y^2,xy) = y^2x^{*}y^{*} = x^{*}(yy^{*}y) = x^{*}y^{**} = (xy^{*})^{*} = \nu(x,y)^{*}.$$ 
\endproof

Now we can complete the story by showing not only that this definition of $\nu$ makes every fractional rig an algebra, but also 
that such an algebra inherits the structure of a fractional rig.

\begin{proposition} 
An algebra for the fractional monad is exactly a fractional rig.
\end{proposition}

\proof
Every fractional rig is an algebra, that is the diagram above commutes: 
\begin{eqnarray*}
\nu(\mu((r,s),(p,q))) & = & \nu(rq^2,spq) =  \\
                      & = & rq^2s^{*}p^{*}q^{*}\\
                      & = & rs^{*}p^{*}qq^{*}q \\
                      & = & rs^{*}p^{*}q^{**} \\
                      & = & \nu(rs^{*},pq^{*}) \\
                      & = & \nu(\fr(\nu)((r,s),(p,q))
\end{eqnarray*}
Conversely an algebra $\nu: \fr(R) \to R$ has $r^* = \nu(\eta(r)^*)$. It remains to check that 
this definition turns $R$ into a fractional rig.  The identities which do not involve nested uses of $(\_)^{*}$ are straightforward.  
For example to show $r^{*}rr^{*} = r^{*}$ we have:
$$r^{*}rr^{*} = \nu(\eta(r)^*)r\nu(\eta(r)^*) = \nu(\eta(r)^*)\nu(\eta(r))\nu(\eta(r)^*) = \nu(\eta(r)^*\eta(r)\eta(r)^*) =\nu(\eta(r)^*) = r^{*}$$
where we use the fact that $\nu$ is a weak rig homomorphism and $\fr(R)$ satisfies the identity.
More difficult is to prove that $r^{***} = r^{*}$: we shall use two facts
$$\xymatrix{\fr(X) \ar@{}[dr]|{(1)} \ar[d]_{*} \ar[r]^{\fr(f)} & \fr(Y) \ar[d]^{*} \\ \fr(X) \ar[r]_{\fr(f)} & \fr(Y)}
~~~ \xymatrix{\fr^2(X)  \ar@{}[dr]|{(2)} \ar[d]_{*} \ar[r]^{\mu} & \fr(X) \ar[d]^{*} \\ \fr^2(X) \ar[r]_{\mu} & \fr(X)}$$
namely (1) the $(\_)^{*}$ on $\fr(X)$ is natural and (2) that $\mu$ preserves the $(\_)^{*}$.  We start by establishing
for any $z \in \fr(R)$ that $\nu(\eta(\nu(z))^*) = \nu(z^{*})$ as:
$$\nu(\eta(\nu(z))^*) = \nu(\fr(\nu)(\eta(z))^{*}) =_{(1)} \nu(\fr(\nu)(\eta(z)^{*})) = \nu(\mu(\eta(z)^{*})) =_{(2)} \nu(\mu(\eta(z))^{*}) = \nu(z^{*})$$
This allows the calculation:
$$r^{***} = \nu(\eta(\nu(\eta(\nu(\eta(r)^{*}))^{*}))^{*}) = \nu(\eta(\nu(\eta(r)^{**}))^{*}) = \nu(\eta(r)^{***}) = \nu(\eta(r)^{*}) = r^{*}.$$
\endproof

Let us denote the category of fractional rigs and and homomorphisms by $\text{\bf fwCRig}$.  Because this is a category of algebras over sets, 
this is a complete and cocomplete category.  Furthermore, we have established:

\begin{corollary}
The underlying functor $V: \text{\bf fwCRig} \to \text{\bf wCRig}$ has a left adjoint which generates the fraction monad on $\text{\bf wCRig}$.
\end{corollary}

This observation suggests an alternative, more abstract, approach to these results: proving that the adjoint between these categories generates 
the fractional monad, in fact, suffices to prove that $\text{\bf fwCRig}$ is monadic over $\text{\bf wCRig}$.  The approach we have followed 
reflects our focus on the fractional monad itself and on its concrete development.


\subsection{Rational functions}


In any fractional rig the $*$-idempotents, $e=e^{*}=ee$, have a special role. If we force the identity $e=1$, this forces all $e'$ with $ee'=e$ 
-- this is the up-set generated by $e$ under the order $e \leq e' \Leftrightarrow ee'=e$ -- to be the identity. For fractional rigs this is an 
expression of localization.  A {\bf localization} 
in fractional rigs is any map which is universal with respect to an identity of the form $e=1$, where $e$ is a ${*}$-idempotent of the domain.  Thus the 
map $\ell_e: R \to R/\<e=1\>$ is a localization at $e$ in case whenever $f: R \to S$ has $f(e)=f(1)$ there is a unique map $f'$ such that:
$$\xymatrix{R \ar[rr]^{\ell_e} \ar[rrd]_f & & R/\<e=1\> \ar@{..>}[d]^{f'} \\ & & S}$$
Here $R/\<e=1\>$ is determined only up to isomorphism, however, there is a particular realization of $R/\<e=1\>$  as the fractional rig $R_e = \{ re |r \in R \}$, 
with the evident addition, multiplication, and definition of $(\_)^{*}$.  This gives a canonical way of representing the localization at any $e$ by the map 
$$\ell_e: R \to R_e; r \mapsto re.$$  
In particular, $\ell_{00^{*}}: R \to R_{00^{*}}$ gives a localization of any fractional rig to one in which $0=1$.

\begin{lemma} 
In $\text{\bf fwCRig}$ the class  of localizations, {\sc loc}, contains all isomorphisms and is closed to composition and 
pushouts along any map.
\end{lemma}

\proof 
All isomorphism are localizations as all isomorphisms are universal solutions to the equation $1=1$. For composition observe, in the canonical representation 
of localizations, $\ell_{e_1}\ell_{e_2} \simeq \ell_{e_1e_2}$.  Finally, the pushout 
of $\ell_e: R \to R_e$ along $f: R \to S$ is given by $\ell_{f(e)}: S \to S_{f(e)}$:
$$\xymatrix{R \ar[d]_{\ell_e} \ar[r]^f & S \ar[d]_{\ell_{f(e)}} \ar[ddr]^{k_1} \\ R_e \ar[r]^{f'} \ar[rrd]_{k_2} & S_{f(e)} \ar@{..>}[dr]|{\hat{k}} \\
            & & K}$$
First note that $f'$ is defined by $f'(er) = f(e)f(r)$, which is clearly a fractional rig homomorphism. Now suppose the outer square commutes.  If we define 
$\hat{k}(f(e)s) = k_2(f(e)s)$, then the right triangle commutes while 
$$\hat{k}(f'(er)) = \hat{k}(f(e)f(r)) = k_2(f(e)f(r)) = k_2(f(er)) = k_1(\ell_e(er)) = k_1(er),$$
showing that the left triangle commutes.  Furthermore, $\hat{k}$ is unique as $\ell_{f(e)}$ is epic, showing that the inner square is a pushout. 
\endproof

This means immediately:

\begin{proposition}
{\sc loc} is a stable system of monics in $\text{\bf fwCRig}^{\rm op}$, so that {\em ({\bf fwCRig}$^{\rm op}$,{\sc loc})} is an ${\cal M}$-category 
and, thus {\em {\sf Par}({\bf fwCRig}$^{\rm op}$,{\sc loc})} is a cartesian restriction category.
\end{proposition}

We shall denote this partial map category $\text{\bf RAT}$ and refer to it as the category of {\bf rational functions}.  Recall that a map $R \to S$ in this category, as 
defined above, is a cospan in $\text{\bf fwCRig}$ of the form:
$$\xymatrix@=12pt{& R_e & \\ R \ar[ur]^{\ell_e} & & & S \ar[ull]_h}$$
where we use the representation $R_e = \{er | r \in R \}$.  This means, in fact, that a map in this category is equivalently a map $h: S \to R$ which preserves 
addition, multiplication, and $(\_)^{*}$ -- but does {\em not} preserve the unit of multiplication, and has $h(0) = h(1) \cdot 0$. These we shall refer to 
as {\bf corational} morphisms. Thus, $\text{\bf RAT}$ can be alternately presented as:

\begin{corollary}
$\text{\bf RAT}$ is precisely the opposite of the category of fractional rigs with corational morphisms. 
\end{corollary}

The advantage of this presentation is that one does not have to contend with spans or pushouts: one can work directly 
with corational maps.  In particular, the corestriction of a corational map $f: S \to R$ is just $f(1) \cdot \_: R \to R$.

This category is certainly not obviously recognizable as a category of rational functions as used in algebraic geometry.  Our 
next objective is to close this gap.  In order to do this we start by briefly reviewing localization in commutative rigs.

The definition of a localization for commutative rigs is a direct generalization of the usual notion of localization for commutative rings, as in \cite{alggeomeis}.  
A {\bf localization} is a rig homomorphism $\phi: R \rightarrow S$ such that there exists a multiplicative set, $U$, with $\phi(U) \subseteq \text{units}(S)$, with 
the property that for any map $f : R \to T$, with $f(U) \subseteq \text{units}(T)$, there is a unique map $k : S \to T$ such that $f = \phi k$:
$$\xymatrix{R \ar[r]^{\phi} \ar[dr]_f & S \ar@{..>}[d]^k \\ & T}$$
A localization is said to be {\bf finitely generated} if there is a finitely generated multiplicative set $U$ for which the map 
is universal.

Denote the class of finitely generated localizations by {\sc Loc}.  We next show that {\sc Loc} is a stable system of monics in ${\bf CRig}^{\op}$, so one may form a partial 
map category for commutative rigs opposite with respect to localizations.

If $R$ is a commutative rig, and $U$ is a multiplicative closed set, $R[U^{-1}]$, is the universal rig obtained with all elements in $U$ turned into into units.  This is called 
the rig of fractions with respect to a multiplicative set $U$, as the operations in the rig are defined as for fractions with denominator chosen from $U$, see for example  
\cite{dumfooteabsalg}.  This is exactly the fractional construction described above except with denominators restricted to $U$ and with the additional ability that 
one may quotient out by arbitrary factors.   There is a canonical localization, $l_U : R \to R[U^{-1}]$; $l_U(r) = \frac{r}{1}$.  It is clear that (finitely 
generated) localizations in {\bf CRig} are epic, contain all isomorphisms, and are closed to composition.  Furthermore, we have the following:

\begin{proposition}
In $\text{\bf CRig}$, the pushout along any map of a (finitely generated) localization exists and is a (finitely generated) localization.  
\end{proposition}

\proof
Let $R,A,S$ be rigs.  Let $\phi : R \to S$ be a localization, let $f : R \to A$ be a rig homomorphism, and 
let $W \subseteq R$ be the factor closed multiplicative set that $\phi$ inverts.  Then $f(W)$ is also a multiplicative set
which is finitely generated if $W$ is, so we can form the canonical localization $l_{f(W)} : A \to A[(f(W))^{-1}]$.  This means that 
$l_{f(W)} f(W) \subset
\text{units}(A[(f(W))^{-1}]$, and so we get a unique $k : S \to A[(f(W))^{-1}]$ such that the following diagram commutes
$$\xymatrix{ R \ar[r]^{\phi} \ar[d]_{f} & S \ar[d]^{k} \\  A \ar[r]_{l_{f(W)}} & A[(f(W))^{-1}]}$$

To show that this square is a pushout, suppose the outer square commutes in:

$$\xymatrix{ R \ar[r]^{\phi} \ar[d]_{f} & S \ar[d]^{k} \ar@/^1pc/[ddr]^{q_1} & \\
             A \ar@/_1pc/[drr]_{q_0} \ar[r]_{l_{f(W)}} & A[(f(W))^{-1}] \ar@{..>}[dr]^{\hat{k}}& \\
             & & Q }$$ 

If we can show that $q_0$ sends $f(W)$ to units, then we get a unique map 
$\hat{k}: A[(f(W))^{-1}] \to Q$.  Now, $q_0(f(W)) = q_1(\phi (W))$ by
commutativity; thus, $q_1 (\phi (W)) \subset \text{units}(Q)$, so
$q_0 (f(W)) \subset \text{units}(Q)$ giving $\hat{k}$.  Next, we must show that $k
\hat{k} = q_1$. However, $\phi q_1 = f q_0  = f l_{f(W)} \hat{k} = \phi k \hat{k}$
and $\phi$ is epic, $q_1 = k \hat{k}$.  Moreover, since $\hat{k}$ is the unique map this makes the bottom triangle 
commute, and the square a pushout.
\endproof

Thus {\sc Loc} is a stable system of monics in {\bf CRig}$^{op}$, and so we can form a partial map category:

\begin{proposition}
$({\bf CRig}^{\op},\text{\sc Loc})$ is an ${\mathcal M}$-category, and $\Par({\bf CRig}^{\op},\text{\sc Loc})$ is a 
cartesian restriction category.
\end{proposition}

We shall call this category $\text{\bf RAT}_{\sf rig}$.  Our next objective is to prove:

\begin{theorem}
\label{rationals-on-rigs}
$\text{\bf RAT}_{\sf rig}$ is the full subcategory of $\text{\bf RAT}$ determined by the objects $\fr(W(R))$ for $R \in \text{\bf CRig}$, 
where $W$ is the inclusion of commutative rigs into weak commutative rigs.
\end{theorem}

To prove  that the induced comparison between the cospan categories is full and faithful it suffices to show that the composite $W \fr$ is a full 
and faithful left adjoint which preserves and reflects localizations.  This because it will then fully represent the maps in the cospan category 
and preserve their composition -- as this is given by a colimit.  That it is a left adjoint follows from \ref{Wleftadj}, the full 
and faithfulness follows from \ref{extracting-rigs}(i) and the preservation and reflection of localizations from \ref{extracting-rigs}(iii),(iv), 
and (v).  We start with:

\begin{lemma}
\label{Wleftadj}
The inclusion functor $W: \text{\bf CRig} \to \text{\bf wCRig}$ has both a left and right adjoint.
\end{lemma}

The left adjoint arises from simply forcing the nullary distributive law to hold.  It is the form of the right adjoint which is of more immediate 
interest to us. Given any weak rig $R$, the set of {\bf rig elements} of $R$ is ${\sf rig}(R) = \{ r | r \cdot 0 = 0 \}$. Clearly rig elements 
include $0$ and $1$, and are closed under the multiplication and addition.  Thus, they form a subrig of any weak rig, and it is this rig which is easily seen 
to give the right adjoint to the inclusion $W$ above.  

This leads to the following series of observations: 

\begin{lemma} 
\label{extracting-rigs}
For any rig $R$:
\begin{enumerate}[(i)]
\item ${\sf rig}(\fr(W(R))) \cong  R$;
\item If $e$ is a ${*}$-idempotent of $\fr(W(R))$ then $e \sim (r,r)$ for some $r \in R$;
\item The up-sets of $*$-idempotents $e \in \fr(W(R))$, $uparrow\! e \{ e' | ee'=e\}$, correspond precisely to finitely 
   generated multiplicative closed subsets of $R$ which are also factor closed, $\Sigma_e= \{ r \in R | (r,r) \geq e\}$;
\item ${\sf rig}(\fr(W(R))/e) = R(\Sigma_e^{-1})$ (where $\Sigma_e^{-1}$ is a rig with $\Sigma_e$ universally inverted);
\item $\fr(W(R))/e \cong \fr(R[\Sigma_e^{-1}])$.
\end{enumerate}
\end{lemma}

\proof ~
\begin{enumerate}[{\em (i)}]
\item Suppose $(r,s)(0,1) \sim (0,1)$ then $(0,1) \rightarrowtriangle_\alpha (p,q) \leftarrowtriangle_\beta (0,s)$.  It follows that 
$\alpha$ is a unit (as it must iteratively divide $1$) and so $p=0$ and $q=\alpha$.  But $q = \beta s$ so that $\beta$ and $s$ are units. 
However then $(r,s) \sim (s^{-1}r,1)$ showing each rig element is (up to equivalence) an original rig element.
\item We must have $(r,s) \rightarrowtriangle_\alpha (p,q) \leftarrowtriangle_\beta (r^2,s^2)$ and 
$(r,s) \rightarrowtriangle_\gamma (p',q') \leftarrowtriangle_\delta (s^2,rs)$ from which we have:
$(r,s) \rightarrowtriangle_{\alpha} (\beta r^2,\beta s^2) \rightarrowtriangle_{\gamma} (\beta\delta rs^2,\beta\delta srs).$
\item $\Sigma_e$ is multiplicatively closed as its idempotents are closed to multiplication, it is factor closed (that $rs \in \Sigma_e$ implies $r,s \in \Sigma_e$) 
provided $e_1e_2 \geq e$ implies $e_1 \geq e$ and $e_2  \geq e$ which is immediate.  Finally, any representative $(r,r)$ for $e$ itself will clearly 
generate the multiplicative set $\Sigma_e$, so it is finitely generated.

A factor closed multiplicative set, $U$, which is finitely generated by $\{ r_1,..,r_n\}$ is generated by a single element, namely the product of the generators, $\prod r_i$, as each 
generator is a factor of this.  However, it is then easy to see that $U = \Sigma_{\prod r_i}$.
\item First observe that forcing $e$ to be a unit forces each $e'\geq e$ to be a unit.  But forcing $e'=(r,r)$ to have $(r,r) \sim (1,1)$ forces $r$ to become 
a unit in $\fr(W(R))/e$ as $(r,1)(1,r) = (r,r) = (1,1)$.  Thus, the evident map $R \to \fr(W(R))/e$ certainly inverts every element in $\Sigma_e$.  However, 
the rig elements of $\fr(W(R))/e$ must, using a similar argument to {\em (i)} above must have their denominators invertible so that they are of the form 
$(r,s)$ where $s \in \Sigma_e$. But these elements give the rig of fractions with respect to $\Sigma_e$ as discussed below.
\item If $l: R \to R[\Sigma_e]$ is the universal map then we have $\fr(W(l)): \fr(W(R)) \to \fr(W(R[\Sigma_e]))$ where this sends $e=(r,r)$ to 
the identity as $r$ becomes a unit.  So this map certainly factors $\fr(W(l)) = \ell_e h$ where $h: \fr(W(R))/e \to \fr(W(R[\Sigma_e]))$.  
However there is also a map (at the level of the underlying sets) in the reverse direction given by $(pu_1^{-1},qu_2^1) \mapsto (pu_1^{-1}r^n,qu_2^1r^n)$ where $e=(r,r)$ and 
a high enough power $n$ is chosen so that $u_1^{-1}$ and $u_2^{-1}$ can be eliminated. This is certainly a section of $h$ as a set map which is enough to show that 
$h$ is bijective and so an isomorphism.
\end{enumerate}
\endproof

We can now complete the proof of Theorem \ref{rationals-on-rigs}:

\proof
As $W$ and $\fr$ are both left adjoints they preserve colimits and thus there is a functor $\text{\bf RAT}_{\sf rig} \to \text{\bf RAT}$ 
which carries an object $R$ to $\fr(W(R))$ and a map $R \to^l R[\Sigma_e^{-1}] \from^h S$  to 
$\fr(W(R)) \to^{\ell_e} \fr(W(R))/e \cong \fr(W(R[\Sigma_e^{-1}]))  \from^h fr(W(S))$. The preservation of colimits ensures composition 
is preserved.

It remains to show that this functor is full.  For this we have to show that given a cospan 
$$\fr(W(R)) \to^{\ell_e}  \fr(W(R))/e \from^h fr(W(S))$$
that it arises bijectively from a span in  $\text{\bf RAT}_{\sf rig}$. For this we note the correspondences:
$$\infer={\fr(W(R)) \to^{\ell_e} \fr(W(R))/e  ~~~~~~ \text{\bf fwCRig}
  }{\infer={ W(R) \to^{(\ell_e)^\flat} \fr(W(R))/e ~~~~~~ \text{\bf wCRig}
  }{R \to^{l_{\Sigma_e}} R[\Sigma_e^{-1}] = {\sf rig}(\fr(W(R))/e) ~~~~~~ \text{\bf CRig}}}$$
and also
$$\infer={\fr(W(S)) \to^h \fr(W(R))/e ~~~~~~ \text{\bf fwCRig}
  }{\infer={W(S) \to^{\eta h} (\fr(W(R))/e) ~~~~~~ \text{\bf wCRig}
  }{\infer={S \to^{(\eta h)^\flat} {\sf rig}(\fr(W(R))/e) ~~~~~~ \text{\bf CRig}
  }{S \to^{(\eta h)^\flat} R[\Sigma_e^{-1}] ~~~~~~ \text{\bf CRig}}}}$$
so that there is a bijective correspondence between the cospans of $\text{\bf RAT}$ from $\fr(W(R))$ to $\fr(W(S))$ and the cospans in 
$\text{\bf RAT}_{\sf rig}$.
\endproof 

We indicated that we had restricted $\text{\bf RAT}$ to rigs by writing $\text{\bf RAT}_{\sf rig}$.  Commutative rings sit inside rigs and, fortuitously, when 
one localizes a ring $R$ in the category of rigs one obtains a ring.  Thus, in specializing this result further to $\text{\bf RAT}_{\sf ring}$ there is nothing 
further to do!

\begin{corollary}
$\text{\bf RAT}_{\sf ring}$ is the full subcategory of $\text{\bf RAT}$ determined by the objects $\fr(W(R))$ where $R \in \text{\bf CRing}$.
\end{corollary}


\subsection{Rational polynomials}


Recall that for any commutative rig $R$, there is an adjunction between {\bf Sets} and $R/\text{\bf CRig}$.  The left adjoint takes a set $B$ to the 
free commutative $R$-algebra on $B$, giving the correspondence
$$\infer={R[x_1,\ldots,x_n] \to_{s^\sharp} S ~~~~ R/\text{\bf CRig}
        }{\{x_1,\ldots,x_n\} \to^s U(S) ~~~~ \text{\bf Sets}}$$
This correspondence gives the morphism, $s^\sharp$, which is obtained by substituting $s_i \in S$ for $x_i$ which we may present as:  
$$s^\sharp: R[x_1,\ldots,x_n] \to S; 
    \sum r_i x_1^{\alpha_{1,i}} ..x_n^{\alpha_{n,i}} \mapsto [s_i/x_i](\sum r_i x_1^{\alpha_{1,i}} \ldots x_n^{\alpha_{n,i}}) = \sum r_i s_1^{\alpha_{1,i}} \ldots s_n^{\alpha_{n,i}}$$ 
(Note that here we identify $r_i$, as is conventional, with its image in an $R$-algebra: strictly speaking we should always write $u(r_i)$ as an $R$-algebra is a map $u: R \to S$.)

The category of finitely generated free commutative $R$-algebras opposite is just the Lawvere theory for $R$-algebras: one may think of it as the category of polynomials over $R$.  
It may be presented concretely as follows: its objects are natural numbers and a map from $n$ to $m$ is an $m$-tuple of polynomials $(p_1,\ldots p_m)$ where each $p_i \in R[x_1,\ldots,x_n]$.  
Clearly the object $n$ is the $n$-fold product of the object $1$ (e.g. the projections $2 \to 1$ are $(x_1)$ and $(x_2)$ and there is only one map $()$ to $0$ making it the final object).
Composition is then given by substituting these tuples:
$$n \to^{(p_1,\ldots, p_m)} m \to^{(q_1,\ldots,q_k)} k = n \to^{([p_1/x_1,\ldots,p_m/x_m]q_1,\ldots,[p_1/x_1,\ldots,p_m/x_m]q_k)} k.$$

The aim of this section is derive a similar concrete description of the category of rational polynomials over a rig (or ring) $R$, which we shall call {\sc Rat$_R$}.  This category will again have natural 
numbers as objects and its maps will involve fractions of the polynomial rigs.  However, before we derive this concrete description, we shall provide an abstract description of this category using our understanding of rational functions developed above.

The category of rational polynomials over a commutative rig $R$ may be described in terms of the partial map category obtained from using localizations in $R/\text{\bf CRig}$.  Recall that 
objects in this coslice category are maps $u: R \to S$, and maps are triangles:
$$\xymatrix@=15pt{ & R \ar[dl]_{u_1} \ar[dr]^{u_2} \\ S_1 \ar[rr]_f & & S_2}$$
A (finitely generated) localization is just a map whose bottom arrow is a localization in $\text{\bf CRig}$. This allows us to form the category of cospans whose left leg 
is a localization: the composition is given as before by pushing out, where pushing outs is the same as in $\text{\bf CRig}$.  We may call this category 
$\text{\bf RAT}_{R/{\sf rig}}$ and, as above, we shall now argue that it is a full subcategory of a larger category of rational functions which we shall call 
$\text{\bf RAT}_{\fr(W(R))}$.  This latter category is formed by taking the cospan category of localizations in the coslice category $\fr(W(R))/\text{\bf fwCRig}$. Thus 
a typical map in this category has the form:
$$\xymatrix{ & \fr(W(R)) \ar@/_/[ddl]_{u_1} \ar[d]|{u'} \ar@/^/[ddrr]^{u_2} \\
             & S_1/e \\
            S_1 \ar[ur]_{\ell_e} & & & S_2 \ar[llu]^{h} }$$
It is now a straightforward observation that:

\begin{proposition}
\label{Rat_R}
$\text{\bf RAT}_{R/{\sf rig}}$ is the full subcategory of $\text{\bf RAT}_{\fr(W(R))}$ determined by the objects $\fr(W(u)): \fr(W(R)) \to \fr(W(S))$ 
for $u \in R/\text{\bf CRig}$.
\end{proposition}

The category of rational polynomials over $R$, {\sc Rat$_R$}, may then be described as the full subcategory of $\text{\bf RAT}_{R/{\sf rig}}$ determined by the objects 
under $R$ given by the canonical (rig) embeddings $u_n: R \to R[x_1,\ldots, x_n]$ for each $n \in \N$.  Thus, the objects correspond to natural numbers.  In $\text{\bf RAT}_R$, 
this is the full subcategory determined by the objects $\fr(W(u_n)$ and the maps are the opposite of the corational functions which fix $\fr(W(R)$.  Unwinding this has the maps 
as cospans:
$$\xymatrix{ & \fr(W(R)) \ar@/_/[ddl]_{\fr(W(u_1))} \ar[d]|{u'} \ar@/^/[ddrr]^{\fr(W(u_2))} \\
             & \fr(W(R[x_1,\ldots,x_n]))/e \\   \fr(W(R[x_1,\ldots,x_n])) \ar[ur]_{\ell_e} & & &  \fr(W(R[x_1,\ldots,x_m])) \ar[llu]^{h} }$$
where we have:
$$\infer={\{  x_1,\ldots,x_m \} \to^{{\sf sub}((\eta h)^\flat)} U({\sf rig}(\fr(W(R[x_1,\ldots,x_n]))/e))  ~~~~~~ \text{\bf Set}
 }{\infer={R[x_1,\ldots,x_m] \to^{(\eta h)^\flat} {\sf rig}(\fr(W(R[x_1,\ldots,x_n]))/e)  ~~~~~~ R/\text{\bf CRig}
 }{\infer={W(R[x_1,\ldots,x_m]) \to^{\eta h} \fr(W(R[x_1,\ldots,x_n]))/e ~~~~~~ W(R)/\text{\bf wCRig}
 }{\fr(W(R[x_1,\ldots,x_m])) \to^h \fr(W(R[x_1,\ldots,x_n]))/e ~~~~~~ \fr(W(R))/\text{\bf fwCRig}}}}$$
Thus, such a map devolves into a selecting, for each variable $x_i$, elements from the underlying set of ${\sf rig}(\fr(W(R[x_1,\ldots,x_n]))/e)$.  To select such elements amounts 
to selecting $m$ fractions from $\fr(W(R[x_1,\ldots,x_n]))$ whose denominator is in the multiplicative set $\Sigma_e$.  Now$\Sigma_e$ is a finitely generated multiplicative set,  
so it can be written as $\Sigma_e := \< p_1,...,p_k \>$, where the $p_i \in R[x_1,\ldots,x_n]$ are the generators.

This allows us to concretely define a category of rational polynomial over a commutative rig $R$.  The objects are natural numbers: the maps $n \to m$ are $m$-tuples of 
rational polynomials in $n$ variables accompanied by a finite set of polynomials called the {\bf restriction set} such that each denominator is in the factor closed 
multiplicative set generated by the restriction set.  For brevity we will write $x_1,\ldots,x_n$ as $\overrightarrow{x}^n$.

\begin{definition}
Let $R$ be a commutative rig.  Define {\sc Rat$_R$} to be the following
\end{definition}
\begin{description}
\item[Objects: ] $n \in \mathbb{N}$
\item[Arrows:  ] $n \rightarrow m$ given by a pair $\tup{f}{g}{n}{m}{U}$ where
\begin{itemize}
\item $(f_i,g_i) \in \fr(W(R[x_1,\ldots,x_n]))$ for each $i$;
\item ${\mathcal U} = \< p_1,\ldots,p_k\> \subseteq R[x_1,\ldots,x_n]$ is a finitely generated factor closed and multiplicatively closed set of polynomials;
\item Each $(f_i,g_i)$ is subject to fractional equality, every denominator $g_i$ is in ${\mathcal U}$, and any $u \in {\cal U}$ can be completely 
eliminated from the fraction (as these are inverted).
\end{itemize}
\item[Identity: ] $\stup{x_i}{1}{n}{n}{\<\>} : n \longrightarrow n$
\item[Composition: ] Given $\tup{f}{g}{n}{m}{U} : n
\longrightarrow m$ and $\tup{f'}{g'}{m}{k}{U'} : m
\longrightarrow k$, then the composition is given by substitution:
\begin{prooftree}
\AxiomC{$\tup{f}{g}{n}{m}{U}$}
\AxiomC{$\tup{f'}{g'}{m}{k}{U'}$}
\BinaryInfC{$\tup{a}{b}{n}{k}{U''}$}
\end{prooftree}

Where
\begin{itemize}
\item $(a_j,b_j) = \left[(f_i,g_i)/x_i\right](f_j',g_j')$,
\item $(\alpha_j,\alpha_j) = \left[(f_i,g_i)/x_i\right] (u_j',u_j')$ where $\< u_1',\ldots,u_w'\> = {\mathcal U}'$,
\item and ${\mathcal U}'' = \left\< u_1,\ldots,u_l,\alpha_1,\ldots,\alpha_w\right\>$, where $\<u_1,\ldots,u_l\> \in {\mathcal U}$.
\end{itemize}
\end{description}

Perhaps the one part of this concrete definition of {\sc Rat$_R$} which requires some explanation is the manner in which ${\mathcal U}''$ is obtained.  To understand what 
is happening, recall that the restriction is determined by a ${*}$-idempotent which for ${\mathcal U}'$ is $e' = (u_1' \ldots u_w',u_1' \ldots u_w')$.  To obtain the new 
${*}$-idempotent we must multiply the ${*}$-idempotent, $e$, obtained from ${\mathcal U}$, with the result of mapping (i.e. substituting) $e'$. 

Here is an example of a composition in {\sc Rat$_\mathbb{Z}$}.  Take the maps 
$$\left(x_1,x_2 \mapsto \left( \frac{5x_1x_2}{x_1},\frac{x_1x_2^2}{x_1+x_2},\frac{(x_1+x_2)^2}{3x_2}\right),\<x_1,x_1+x_2,x_2\>\right) : 2 \to 3,$$
and
$$\left(x_1,x_2,x_3 \mapsto \left(\frac{7(x_1+x_3)}{x_1x_2},\frac{x_1}{1}\right),\<4+x_3+x_1,x_1,x_2\>\right) : 3 \to 2.$$
The composite of the above maps -- without cleaning up any factors -- is:
$$\left(x_1,x_2 \mapsto\!\! \left(\frac{(105x_1x_2^2+7x_1(x_1+x_2)^2)(x_1(x_1+x_2))^2}{15x_1^4x_2^4(x_1+x_2)},\frac{5x_1x_2}{x_1}\right),\<\begin{smallmatrix}x_1,x_1+x_2,x_2,\\ 5x_1^3x_2,x_1x_2^2(x_1+x_2)^2, \\
      \left(\begin{smallmatrix} 15x_1x_2^2+12x_1x_2 \\ +x_1(x_1+x_2)^2\end{smallmatrix} \right) (3x_2x_1)^2 \end{smallmatrix}\>\right)\!\!: 2 \to 2.$$
Recall that here we have used the Kleisli composition of the fractional monad.   This can be cleaned up somewhat by using  properties of fractional rigs:
$$\left(x_1,x_2 \mapsto\!\! \left(\frac{(105x_2^2+7x_1(x_1+x_2)^2)(x_1+x_2)^2}{15x_1^2x_2^4(x_1+x_2)},\frac{5x_1x_2}{x_1}\right),\<\begin{smallmatrix}x_1,x_1+x_2,x_2,\\ 5, 3, \\
      \left(\begin{smallmatrix} 15x_2^2+12x_2 \\ +(x_1+x_2)^2\end{smallmatrix} \right) \end{smallmatrix}\>\right)\!\!: 2 \to 2.$$
Finally, we can actually eliminate factors which are in the multiplicative set:
$$\left(x_1,x_2 \mapsto\!\! \left(\frac{(105x_2^2+7x_1(x_1+x_2)^2)(x_1+x_2)}{15x_1^2x_2^4},\frac{5x_2}{1}\right),\<\begin{smallmatrix}x_1,x_1+x_2,x_2,\\ 5, 3, \\
      \left(\begin{smallmatrix} 15x_2^2+12x_2 \\ +(x_1+x_2)^2\end{smallmatrix} \right) \end{smallmatrix}\>\right)\!\!: 2 \to 2.$$
The hard work we have done with fractional rigs (in particular, Proposition \ref{Rat_R}) can now be reaped to give:

\begin{proposition}
For each commutative rig $R$, {\sc Rat$_R$} is a cartesian restriction category.
\end{proposition}

The restrictions are, in this presentation, given by the multiplicative sets.

A final remark which will be useful in the next section.  If $R$ is a ring, then the rig of polynomials 
$R[x_1,\ldots,x_n]$ is also a ring. Thus, as before, there is nothing extra to be done to define 
{\sc Rat$_R$} for a commutative ring.


\subsection{Differential structure on rational polynomials}


To be a differential restriction category, {\sc Rat$_R$} must have cartesian left additive structure.  

\begin{proposition}
For each commutative rig, $R$, {\sc Rat$_R$} is a cartesian left additive restriction category.
\end{proposition}
\begin{proof}
Each object is canonically a total commutative monoid by the map: 
$$( (x_i)_{i= 1,\dots,2n} \mapsto (x_i+x_{n+i})_{i=1,\dots,n}): 2n \to n$$
and this clearly satisfies the required exchange coherence (see \ref{thmAddCart}).

If $\tup{p}{q}{n}{m}{U} , \tup{p'}{q'}{n}{m}{V} : n \longrightarrow m$ are arbitrary parallel maps then
\begin{eqnarray*}
& &\tup{p}{q}{n}{m}{U} + \tup{p'}{q'}{n}{m}{V} \\
&=& \tupc{p_i{q'}_i+{p'}_iq}{q_iq'}{n}{m}{{\mathcal U}\cup{\mathcal V}}
\end{eqnarray*}
so we are using the addition defined in $\fr(R[x_1,\ldots,x_n])$.  
\end{proof}

It remains to define the differential structure of {\sc Rat$_R$}.   We will use formal partial derivatives to define this structure.  Formal partial derivatives are 
used in many places: in Galois theory the formal derivative is used to determine if a polynomial has repeated roots \cite{galoisstewart}, and in algebraic geometry 
the rank of the formal Jacobian matrix is used to determine if a local ring is regular \cite{alggeomeis}.  Finally, it is also important to note that here we must 
assume we start with a commutative ring, rather than a rig: negatives are required to define the formal derivative of a rational function.

\begin{proposition}
If $R$ is a commutative ring, then {\sc Rat$_R$} is a differential restriction category.
\end{proposition}

Given a ring, $R$, there is a formal partial derivative for elements of
$R[x_1,\ldots,x_n]$.  Let $f = \displaystyle\sum_l a_l x_1^{l_1} \cdots
x_n^{l_n}$ be a polynomial.  Then the formal partial derivative of $f$ with
respect to the variable $x_k$ is 
$$\frac{\partial f}{\partial x_k} = \displaystyle\sum_l l_ka_l x_1^{l_1} \cdots
x_{k-1}^{l_{k-1}} x_k^{l_k-1} x_{k+1}^{l_{k+1}} \cdots x_n^{l_n}$$
Extend the above definition to rational functions,  where $g = \frac{p}{q}$ by
$$\frac{\partial g}{\partial x_k} = \frac{\frac{\partial p}{\partial x_k} q - p
\frac{\partial q}{\partial x_k}}{q^2}.$$
From the above observation, one can show that the unit must have an additive inverse and, thus,  every element 
must have an additive inverse.  This means we need a ring to define the differential structure on rational functions.
Now, if we have $f = \left( f_1,\ldots,f_m\right) =\left( \frac{p_1}{q_1},\ldots,\frac{p_m}{q_m} \right)$, an $m$-tuple of rational
functions in $n$ variables over $R$, then we can define the formal Jacobian at a point of $R^n$ as the $m \times n$ matrix
$$J_f (y_1,\ldots,y_n) = \left[
\begin{matrix}
\frac{\partial f_1}{\partial x_1} (y_1,\ldots,y_n) & \ldots & \frac{\partial f_1}{\partial x_n} (y_1,\ldots,y_n)\\
\vdots & \ddots & \vdots \\
\frac{\partial f_m}{\partial x_1} (y_1,\ldots,y_n)& \ldots & \frac{\partial f_m}{\partial x_n} (y_1,\ldots,y_n)
\end{matrix}
\right]$$
Finally, consider {\sc Rat$_R$} where $R$ is a commutative ring.  Then, define the differential structure to be
\begin{prooftree}
\AxiomC{$\tup{p}{q}{n}{m}{U} : n \to m$}
\RightLabel{$D[\underline{\ \ }]$}
\UnaryInfC{$\left(\overrightarrow{x}^{2n} \mapsto \left(\left(J_{(p_i,q_i)}(x_{n+1},\ldots,x_{2n})\right) \cdot \left( x_1 ,
\ldots , x_n\right)\right) , \left[ x_{n+i} / x_i \right] {\mathcal R}\right): 2n = n \times n \to m$}
\end{prooftree}
For example, consider {\sc Rat$_{\mathbb{Z}}$} and the map $\left(x_1,x_2 \mapsto \left(\frac{1}{x_1},\frac{x_1^2}{1+x_2}\right),\<x_1,1+x_2\>\right)$.  
Then the differential of this map is
$$\left(x_1,x_2,x_3,x_4 \mapsto \left(\frac{-x_1}{x^2_3},\frac{2x_3x_1(x_4+1)-x_3^2x_2}{(x_4+1)^2}\right),\<x_3,1+x_4\>\right)$$

\begin{proof}
In \cite{cartDiff}, the category of smooth functions between finite dimensional $\mathbb{R}$ vector spaces is established
as an example of a cartesian differential category using the Jacobian as the differential structure.  The proof for showing that {\sc Rat$_R$} is a
differential restriction category is much the same, so we will highlight the places 
where the axioms have changed and new axioms have been added.

\medskip

\noindent
{\bf [DR.2]} Consider the second part of {\bf [DR.2]},  $\<0,g\>D[f] = \rs{gf}0$: it has been modified by the addition of the restriction constraint.
Let $f = (\overrightarrow{x}^n \mapsto (f_i,f'_i)_m,{\mathcal V})$ and  $g = (\overrightarrow{x}^k \mapsto (g_i,g'_i)_n,{\mathcal U})$ then it is clear 
that we must show
\begin{eqnarray*}
\left[0 /x_i,(g_i,g_i')/x_{n+i} \right] {\mathcal V}' &=& \left[ (g_i,g_i') /x_i \right] {\mathcal V}
\end{eqnarray*} 
where ${\mathcal V}' = [x_{n+i}/x_i] {\mathcal V}$ so that ${\mathcal V}'$ is just ${\mathcal V}$ with variable indices shifted by $n$.  Thus, these
substitutions are clearly equal.

\medskip

\noindent
{\bf [DR.6]} 
Consider the maps 
	\[ g = \tup{g}{g'}{k}{n}{U}, h = \tup{h}{h'}{k}{n}{W}, \mbox{ and } k = \tup{k}{k'}{k}{n}{W}. \]
The restriction set for $D[f]$ is ${\mathcal V}' = \rst{h}\<
[x_{n+i}/x_i]{\mathcal V}$, and the restriction set for $D[D[f]]$ is ${\mathcal
V}'' = [x_{2n+j}/x_j] {\mathcal V'} = [x_{3n+i}/x_i]{\mathcal V}$.  
We must prove $\< \< g,0\>,\<h,k\>\> D[D[f]] = \rst{h}\< g,k\>D[f]$ which translates to:

\begin{eqnarray*}
\lefteqn{\hspace{-65pt} \left(\overrightarrow{x}^k \mapsto  \left( (g_1,g_1'),\ldots,(g_n,g_n'),0, \ldots ,0,(h_1,h_1'),\ldots,(h_n,h_n'),(k_1,k_1'),\ldots,(k_n,k_n') \right), \left\< {\mathcal U} \cup {\mathcal W} \cup {\mathcal T}\right\>\right)} \\
D[D[f]] & = & \rs{\tup{h}{h'}{k}{n}{T}} \\
& &  ~~~~~~ \left(\overrightarrow{x}^k \mapsto \left( (g_1,g_1') , \ldots , (g_n,g_n'),(k_1,k_1'),\ldots, (k_n,k_n') \right), \left\<{\mathcal U} \cup {\mathcal W}\right\>\right) D[f].
\end{eqnarray*}
The rational functions of the maps are easily seen to be the same.  It remains to prove that the restriction sets are the same, that is:
\begin{align*}
&\left\<\left(U \cup W \cup T \right) \cup
\left[(g_i,g_i')/x_i , 0_i/x_{n+i} , (h_i,h_i')/x_{2n+i},(k_i,k_i')/x_{3n+i}\right] {\mathcal V}''\right\>\\
&= \left\< U \cup \left( W \cup T \cup \left[(g_i,g_i')/x_i , (k_i,k_i')/x_{n+i}\right]{\mathcal
V}'\right)\right\>
\end{align*}
This amounts to showing: 
\begin{eqnarray*}
\left[(g_i,g_i')/x_i , 0_i/x_{n+i} , (h_i,h_i')/x_{2n+i},(k_i,k_i')/x_{3n+i}\right] {\mathcal V}''
&=& \left[(g_i,g_i')/x_i , (k_i,k_i')/x_{n+i}\right]{\mathcal V}'.
\end{eqnarray*}
which is immediate from the variable shifts which are involved.
\medskip

\noindent
{\bf [DR.8]} Let $f = \left(\overrightarrow{x}^n \mapsto (f_i){i=1}^m,{\mathcal V}\right) : n \to m$ then 
\begin{eqnarray*}
(1 \times \rs{f}) \pi_0
&=& \left(\overrightarrow{x}^{2n} \mapsto (x_i)_{i=1}^{2n} , [x_{n+i}/x_i]{\mathcal V} \right) \pi_0 \\
&=& \left(\overrightarrow{x}^{2n} \mapsto (x_i)_{i=1}^{n} , [x_{n+i}/x_i]{\mathcal V} \right) \\
&=& \left(\overrightarrow{x}^{2n} \mapsto I_{n\times n} \vec{x} , [x_{n+i}/x_i]{\mathcal V} \right) \\
&=& \left(\overrightarrow{x}^{2n} \mapsto \left(J_{(x_i)_i}(\overrightarrow{x}_{n+1}^{2n})\right)\cdot
\overrightarrow{x}_{1}^{n},[x_{n+i}/x_i]{\mathcal V} \right) \\
&=& D[\rs{f}].
\end{eqnarray*}

\medskip

\noindent
{\bf [DR.9]}
Considering $f = (\overrightarrow{x}^n \mapsto (f_i,f'_i)_{i=1}^m,{\mathcal V})$, we have 
\begin{eqnarray*}
1 \times \rs{f }
&=& \stup{x_i}{1}{n}{n}{\{\}} \times \stup{x_i}{1}{n}{n}{V} \\
&=& \stupc{x_i}{1}{2n}{2n}{\left[x_{n+i}/x_i\right]{\mathcal V}} \\
&=& \rs{D[f]}.
\end{eqnarray*}
\end{proof}

\subsection{Further properties of {\sc Rat$_R$}}

In this section we will describe three aspects of {\sc Rat$_R$}.  First we will prove that {\sc Rat$_R$} has nowhere defined maps for each $R$.  Next, 
after briefly introducing the definition of $0$-unitariness for restriction categories, we will show that if $R$ is an integral domain, then {\sc Rat$_R$} 
is a $0$-unitary restriction category.  Finally, we will show that {\sc Rat$_R$} does not in general have joins.  

Recall from section \ref{subsecJoins} that a restriction category $\X$ has {\bf nowhere defined maps}, if for each $\X(A,B)$ there is a map $\emptyset_{AB}$ 
which is a bottom element for $\X(A,B)$, and these are preserved by precomposition.  We will show that {\sc Rat$_R$} always has nowhere defined maps.  Intuitively, 
a nowhere defined rational function should be one whose restriction set ${\mathcal U}$ is the entire rig $R[x_1, \ldots, x_n]$.  This can be achieved with a 
finitely generated set by simply considering the set generated by $0$, since any such polynomial is in the factor closure of $0$.    

\begin{proposition}\label{rathaszeros}
For any commutative rig $R$, {\sc Rat$_R$} has nowhere defined maps given by 
	\[ \rzro{n}{m}. \]
\end{proposition}

\begin{proof}
First, note 
$$\rs{\rzro{n}{m}} = \stupc{x_i}{1}{n}{n}{0} = \rzro{n}{n}$$
 since $0x_i=0$.
Next, note that $R[x_1,\ldots,x_n] = \left\<0\right\> = \left\<\left\<0\right\> \cup {\mathcal U}\right\>$.  Let $(a_i,b_i) = \left[(1,1)/x_i\right](p_i,q_i)$; clearly for each $i$, 
$$0 = 0a_i = 0b_i.$$
Thus, the following equalities are clear:
\begin{align*}
&\rzro{n}{n}\tup{p}{q}{n}{m}{U} \\
&= \tupc{a}{b}{n}{m}{\left\<0\right\> \cup {\mathcal U}}\\
&= \rzro{n}{m},
\end{align*}
so that this map is the bottom element.  Now consider
\begin{align*}
& \tup{p}{q}{n}{m}{U}\rzro{m}{k} \\
&= \stupc{1}{1}{n}{k}{{\mathcal U}\cup \left\<0\right\>}\\
&= \rzro{n}{k},
\end{align*}
so that these maps are preserved by precomposition, which completes the proof that {\sc Rat$_R$} has nowhere defined maps.
\end{proof}

Now, if $R$ is an integral domain, we would expect that whenever two rational functions agree on some common restriction idempotent, then they should be equal wherever they are both defined.  To make this idea explicit, we will introduce the concept of $0$-unitary for restriction categories\footnote{This is related to the concept of $0$-unitary from inverse semigroup theory (see \cite{lawson}); the relationship will be explored in detail in a future paper.}.  

Let $\X$ be a restriction category with nowhere defined maps.  To define $0$-unitariness, we first define a relation $\leq_0$ on parallel arrows, called the {\bf $0$-density relation}, as follows:
$$f \leq_0 g \mbox{ if } f \leq g \text { and  } hf = \emptyset \text{ implies } hg = \emptyset.$$
$\X$ is a {\bf $0$-unitary} restriction category when for any $f,g,h$:
$$f \geq_0 h \leq_0 g \text{ implies } f \smile g .$$

\begin{lemma}\label{zerounitaryzeros}
Let $\X$ be a restriction category with nowhere defined maps, and assume $h \leq_0 f$.  Then if $f$ or $h$ equals $\emptyset$, then both $f$ and $h$ equal $\emptyset$.
\end{lemma}
\begin{proof}
Since $h \leq_0 f$, we have $h=\rs{h}f$, and whenever $kh = \emptyset$, $kf=\emptyset$.

First assume that $f = \emptyset$.  Then $h=\emptyset$ since
$$h = \rs{h}f = \rs{h} \emptyset = \emptyset.$$

Next, assume that $h = \emptyset$.  Then by $0$-unitariness, 
$$1h = h = \emptyset \text{ implies } 1f = \emptyset,$$
which completes the proof.
\end{proof}

Now we prove that {\sc Rat$_R$} is a $0$-unitary restriction category when $R$ is an integral domain.

\begin{proposition}\label{ratzerounitary}
Let $R$ be an integral domain.  Then {\sc Rat$_R$} is a $0$-unitary restriction category.
\end{proposition}
\begin{proof}

Consider the maps:
$$\tup{f}{f'}{n}{m}{U}, \tup{g}{g'}{n}{m}{V},~\mbox{ and}~ \tup{h}{h'}{n}{m}{W}$$  
Assume:
\begin{eqnarray*}
\tup{h}{h'}{n}{m}{W} & \leq_0 & \tup{f}{f'}{n}{m}{U} \\ 
\tup{h}{h'}{n}{m}{W} & \leq_0 & \tup{g}{g'}{n}{m}{V}.
\end{eqnarray*}
Now if any of the above maps are $\emptyset$, then lemma (\ref{zerounitaryzeros}) says that all three of the above equal $\emptyset$; therefore,
$$\tup{f}{f'}{n}{m}{U} \smile \tup{g}{g'}{n}{m}{V}.$$
Thus, suppose all three are not $\emptyset$.  Then $0 \not \in {\mathcal U},{\mathcal V},\text{ or }{\mathcal W}$.  Then we have
\begin{eqnarray*}
\lefteqn{\tupc{f}{f'}{n}{m}{{\mathcal W} \cup {\mathcal U}} } \\
&= &  \rs{\tup{h}{h'}{n}{m}{W}}\tup{f}{f'}{n}{m}{U} \\
&= & \rs{\tup{h}{h'}{n}{m}{W}}\tup{g}{g'}{n}{m}{V} \mbox{ since $\rs{h}f = h = \rs{h}g$,} \\
&= & \tupc{g}{g'}{n}{m}{{\mathcal W} \cup {\mathcal V}} .
\end{eqnarray*}
Now, since $R$ is an integral domain, the product of two nonzero elements is nonzero.  Thus, $0 \not \in \left\< {\mathcal W} \cup {\mathcal U}\right\>$.  Thus for each $i$, there is a $W_i \not = 0 \in \left\< {\mathcal W} \cup {\mathcal U}\right\>$ such that $W_if_ig_i' = W_if_i'g_i$.  Moreover, the fact that $R$ is an integral domain also gives the cancellation property: if $a \ne 0$, $ac=ab$ implies $c=b$.  Thus, we have that $f_ig_i' = f_i'g_i$, which proves
$$\tup{f}{f'}{n}{m}{U} \smile \tup{g}{g'}{n}{m}{V}.$$

Thus, when $R$ is an integral domain, {\sc Rat$_R$} is a $0$-unitary restriction category.
\end{proof}

It may seem natural to ask if {\sc Rat$_R$} has finite joins, especially if $R$ has unique factorization.  If $R$ is a unique factorization domain, it is easy to show that 
any two compatible maps in {\sc Rat$_R$} will have the form 
$$\tup{P}{Q}{n}{m}{U} \smile \tup{P}{Q}{n}{m}{V},$$
where $\gcd(P_i,Q_i)=1$.  Thus $Q_i \in U,V$, so $Q_i \in \left\< {\mathcal U} \cap {\mathcal V} \right\>$.  Thus from the order theoretic nature of joins, the only candidate 
for the join is $\tupc{P}{Q}{n}{m}{{\mathcal U} \cap {\mathcal V}}$.  However, reducing the restriction sets of compatible maps by intersection does not define a join restriction 
structure on {\sc Rat$_R$}, as stability under composition will not always hold.  For a counterexample, consider the maps 
$$\left(1,\<x-1\> \right) \smile \left(1,\<y-1\>\right).$$
By the above discussion, $\left(1,\<\<x-1\> \cap \<y-1\>\>\right)$  must be $\left(1,\<1\>\right)$.  We will show that $s(f \vee g) \not = sf \vee sg$.  Consider the map 
$\left((x^2,x^2),\{\}\right)$.  Then 
$$\left((x^2,x^2),\{\}\right) \left(1,\<1\>\right) = \left(1,\<1\>\right).$$
However, 
	\[ \left((x^2,x^2),\{\}\right) \left(1,\<x-1\> \right) = \left(1,\<x+1,x-1\>\right) \]
and
	\[ \left((x^2,x^2),\{\}\right)\left(1,\<y-1\>\right)=\left(1,\<x+1,x-1\>\right).  \] 
The ``join'' of the latter two maps is $\left(1,\<x+1,x-1\>\right) \not = \left(1,\<1\>\right)$.  Thus, in general {\sc Rat$_R$} does not have joins.


\section{Join completion and differential structure}\label{sectionJoins}


In the final two sections of the paper, our goal is to show that when one adds joins or relative complements of partial maps, differential structure is 
preserved.  These are important results, as they show that one can add more logical operations to the maps of a differential restriction category, while 
retaining the differential structure.  

\subsection{The join completion}

As we have just seen, a restriction category need not have joins, but there is a universal construction which freely adds joins to any restriction category.  
We show in this section that if the original restriction category has differential structure, then so does its join completion.  By join completing 
{\sc Rat$_R$}, we thus get a restriction category which has both joins and differential structure, but is very different from the differential 
restriction category of smooth functions defined on open subsets of $\R^n$.  

The join completion we describe here was first given in this form in \cite{boolean}, but follows ideas of Grandis from \cite{manifolds}. 

\begin{definition}
Given a restriction category $\X$, define $\jn(\X)$ to have:
\begin{itemize}
	\item objects: those of $\X$;
	\item an arrow $X \to^A Y$ is a subset $A \subseteq \X(X,Y)$ such that $A$ is down-closed (under the restriction order), and elements are pairwise compatible;
	\item $X \to^{1_X} X$ is given by the down-closure of the identity, $\da 1_X$;
	\item the composite of $A$ and $B$ is $\{fg: f \in A, g \in B \}$;
	\item restriction of $A$ is $\{\rs{f}: f \in A \}$;
	\item the join of $(A_i)_{i \in I}$ is given by the union of the $A_i$. 
\end{itemize}
\end{definition}

From \cite{boolean}, we have the following result:

\begin{theorem}
$\jn(\X)$ is a join-restriction category, and is the left adjoint to the forgetful functor from join restriction categories to restriction categories.  
\end{theorem}

Note that this construction destroys any existing joins.  This can be dealt with: for example, if one wishes to join complete a restriction category which already has empty maps (such as {\sc Rat$_R$}) and one wants to preserve these empty maps, then one can modify the above construction by insisting that each down-closed set contain the empty map.  

Because we will frequently be dealing with the down-closures of various sets, the following lemma will be extremely helpful.

\begin{lemma}\label{lemmaDC}(Down-closure lemma)
Suppose $\X$ is a restriction category, and $A, B \subseteq \X(A,B)$.  Then we have:
\begin{enumerate}[(i)]
	\item $\da A \da B = \da (AB) $;
	\item $\rs{\da A} = \da (\rs{A})$;
	\item if $\X$ is cartesian, $\< \da A, \da B \> = \da \<A, B \>$;   
	\item if $\X$ is left additive, $\da A + \da B = \da(A + B);$
	\item if $\X$ has differential structure, $D[\da A] = \da D[A]$.
\end{enumerate}
\end{lemma}

\begin{proof}
\begin{enumerate}[(i)]
	\item If $h \in \da (AB)$, then $\exists f \in A, g \in B$ such that $h \leq fg$.  So $\rs{h}fg = h$, and $\rs{h}f \in \da A$, $b \in \da B$, so $h \in \da A \da B$. Conversely, if $mn \in \da A \da B$, there exists $f, g$ such that $m \leq f \in A, n \leq g \in B$.  But composition preserves order, so $mn \leq fg$, so $mn \in \da (AB)$.
	
	\item Suppose $h \in \rs{\da A}$.  So there exists $f \in A$ such that $h \leq f$.  Since restriction preserves order, $\rs{h} \leq \rs{f}$.  But since $h \in \rs{\da A}$, $h$ is idempotent, so we have $h \leq \rs{f}$.  So $h \in \da (\rs{A})$.  Conversely, suppose $h \in \da (\rs{A})$, so $h \leq \rs{f}$ for some $f \in A$.  Then we have $h = \rs{h}\, \rs{f} = \rs{\rs{h} f}$, so $h$ is idempotent and $h \leq f$, so $h \in \rs{\da A}$.   
	
	\item Suppose $h \in \da \<A, B \>$, so $h \leq \<f,g\>$ for $f \in A, g \in B$.  Then $h = \rs{h}\<f,g\> = \<\rs{h}f,g\>$, and $\rs{h}f \in \da A$, $g \in \da B$, so $h \in \< \da A, \da B \>$.  Conversely, suppose $h \in \< \da A, \da B \>$, so that $h = \<m,n\>$ where $m\leq f \in A, n \leq g \in B$.  Since pairing preserves order, $h = \<m,n\> \leq \<f,g\>$, so $h \in \da \<A, B \>$.  
	
	\item Suppose $h \in \da A + \da B$, so $h = m + n$, where $m \leq f \in A$, $n \leq g \in B$.  Since addition preserves order, $h = m + n \leq f + g$, so $h \in \da(A + B)$.  Conversely, suppose $h \in \da(A + B)$.  Then there exist $f \in A, g \in B$ so that $h \leq f + g$.  Then $h = \rs{h}(f + g) = \rs{h}f + \rs{h}g$ (by left additivity), so $h \in \da A + \da B$.  
	
	\item Suppose $h \in D[\da A]$.  Then there exists $m \leq f \in A$ so that $h \leq D[m]$.  But differentiation preserves order, so $h \leq D[m] \leq D[f]$, so $h \in \da D[A]$.  Conversely, suppose $h \in D[A]$.  Then there exists $f \in A$ so that $h \leq D[f]$, so $h \in D[\da A]$.

\end{enumerate}
\end{proof}

\subsection{Cartesian structure}

We begin by showing that cartesianess is preserved by the join completion.

\begin{theorem}
If $\X$ is a cartesian restriction category, then so is $\jn(\X)$.  
\end{theorem}
\begin{proof}
We define $1$ and $X \times Y$ as for $\X$, the projections to be $\da \pi_0$ and $\da \pi_1$, the terminal maps to be $\da (!_A)$, and 
	\[ \< A, B \> := \{\<f,g\>: f \in A, g \in B \} \]
This is compatible by Proposition \ref{propCart}, and down-closed since if $h \leq \<f,g\>$, then
	\[ h = \rs{h}\<f,g\> = \<\rs{h}f,g\> \]
so since $A$ is down-closed, this is also in $\<A,B\>$.  

The terminal maps do indeed satisfy the required property, as
	\[ \rs{A}\da (!_A) = \rs{A}!_A = \{ \rs{f}!_A: f \in A \} = \{f : f \in A\} = A, \]
as required.  

To show that $\< - , - \>$ satisfies the required property, consider
	\[ \<A,B\>\da \pi_0 = \{ \<f,g\>\pi_0: f \in A, g \in B \} = \{ \rs{g}f: f \in A, g \in B \} = \rs{B}A \]
and similarly for $\da \pi_1$.  

We now need to show that $\< - , - \>$ is universal with respect to this property.  That is, suppose there exists a compatible down-closed set of arrows $C$ with the property that $C \da \pi_0 = \rs{B}A$ and $C \da \pi_1 = \rs{A}B$.  We need to show that $C = \<A,B\>$.

To show that $C \subseteq \<A,B\>$, let $c \in C$.  Since $\da (C \pi_0) = C \da \pi_0 = \rs{B}A$, there exists $f \in A, g \in B$ such that $c\pi_0 = \rs{g}f$.  Then, since $\da (C \pi_1) = C \da \pi_1 = \rs{A}B$, there exists a $c'$ such that $c'\pi_1 = \rs{f}g$.  Then
	\[ \rs{c'}c\pi_0 = \rs{c'}\rs{c}c\pi_0 = \rs{c'}\rs{c}\, \rs{g}f \]
and since $c \smile c'$,
	\[ \rs{c'}c\pi_1 = \rs{c'}\rs{c}c'\pi_1 = \rs{c'}\rs{c}\rs{f}g \]
Thus, by the universality of $\rs{c'}\rs{c}\<f,g\>$, $\rs{c'}c = \rs{c'}\rs{c}\<f,g\>$.  Thus
	\[ c \leq \rs{c'}c = \rs{c'}\rs{c}\<f,g\> \leq \<f,g\>, \]
so since $\<A,B\>$ is down-closed, $c \in \<f,g\>$.  

To show that $\<A,B\> \subseteq C$, let $f \in A, g \in B$.  Then there exists $c$ such that 
	\[ c\pi_0 = \rs{g}f = \<f,g\>\pi_0. \]  
Thus, there exists $f' \in A,g' \in B$ such that 
	\[ c \pi_1 = \rs{f'}g' = \<f',g'\>\pi_1. \]
Now, we have
	\[ \rs{\<f',g'\>}\<f,g\>\pi_0 = \rs{\<f',g'\>}\, \rs{\<f,g\>}\<f,g\>\pi_0 = \rs{\<f',g'\>}\, \rs{\<f,g\>}c\pi_0 \]
and since $f \smile f'$ and $g \smile g'$, $\<f,g\> \smile \<f',g'\>$, so we also get
	\[ \rs{\<f',g'\>}\<f,g\>\pi_1 = \rs{\<f',g'\>}\, \rs{\<f,g\>}\<f',g'\>\pi_1 = \rs{\<f',g'\>}\, \rs{\<f,g\>}c \pi_1. \]
Thus, by the universality of $\rs{\<f',g'\>}\<f,g\>$,
	\[ \<f,g \> \leq \rs{\<f',g'\>}\<f,g\> = \rs{\<f',g'\>}\, \rs{\<f,g\>}c \leq c\, . \]
Since $C$ is down-closed, this shows $\<f,g\> \in C$, as required.  
\end{proof}

\subsection{Left additive structure}

Next, we show that left additive structure is preserved.  

\begin{theorem}
If $\X$ is a left additive restriction category, then so is $\jn(\X)$, where 
	\[ 0_{\jn(\X)} := \da 0 \mbox{ and } A + B := \{f + g: f \in A, g \in B \} \].  
\end{theorem}
\begin{proof}
By Proposition \ref{propLA}, $A+B$ is a compatible set.  For down-closed, suppose $h \leq f + g$.  Then $h = \rs{h}(f + g) = \rs{h}f + \rs{h}g$.  Since $A$ and $B$ are down-closed, $\rs{h}f \in A$, $\rs{h}g \in B$, so $h \in A + B$. 

That this gives a commutative monoid structure on each hom-set follows directly from Lemma \ref{lemmaDC}, as does $\rs{0} = \da 1$.  Finally,
	\[ \rs{A+B} = \rs{\{f + g: f \in A, g\in B\}} = \{\rs{f+g}: f\in A, g\in B\} = \{\rs{f}\rs{g}: f \in A, g\in B\} = \rs{A}\, \rs{B}. \]
so that $\jn(\X)$ is a left additive restriction category.  
\end{proof}

\begin{theorem}
If $\X$ is a cartesian left additive restriction category, then so is $\jn(\X)$. 
\end{theorem}
\begin{proof}
Immediate from Theorem \ref{thmAddCart}.
\end{proof}

\subsection{Differential structure}

Finally, we show that differential structure is preserved.  There is one small subtlety, however.  To define the pairing or addition of maps in $\jn(\X)$, we merely needed to add or pair pointwise, as the resulting set was automatically down-closed and pairwise compatible if the original was.  However, note that $A$ being down-closed does not imply $\{D[f]: f \in A \}$ down-closed.  Axiom \dr{9} requires that differentials be total in the first component.  However, this is not always true of an arbitrary $h \leq D[f]$.  Thus, to define the differential in the join completion, we make take the down-closure of $\{D[f]: f \in A \}$.

\begin{theorem}
If $\X$ is a differential restriction category, then so is $\jn(\X)$, where
	\[ D[A] := \da \{ D[f]: f \in A \} \]
\end{theorem}
\begin{proof} Checking the differential axioms is a straightforward application of our down-closure lemma.  For example, for \dr{1}, by the down-closure lemmas,
	\[ D[0_{\jn(\X)}] =  D[\da 0] = \da D[0] = \da 0 = 0_{\jn(\X)} \]
and 
	\[ D[A + B] = \da \{D[f+g]: f \in A, g \in B \} = \da \{D[f] + D[g]: f \in A, g \in B \} = D[A] + D[B]\, . \]
Similarly, to check \dr{5}:
	\[ D[AB] = \da \{D[fg]: f \in A, g \in B \} = \da \{\<D[f],\pi_1 f\>D[g]: f \in A, g \in B \} = \<D[A],\da \pi_1 A \> DB \]
where the last equality follows from several applications of the down-closure lemmas.  All other axioms similarly follow.  
\end{proof}

Finally, it is easy to see the following:
\begin{proposition}
The unit $\X \to \jn(\X)$, which sends $f$ to $\da f$, is a differential restriction functor.
\end{proposition}
\begin{proof}
The result immediately follows, given the additive, cartesian, and differential structure of $\jn(\X)$.
\end{proof}

Thus, by Proposition \ref{diffFunctors}, we have
\begin{corollary}
If $\X$ is a differential restriction category, and $f$ is additive/strongly additive/linear, then so is $\da f$ in $\jn(\X)$.
\end{corollary}



\section{Classical completion and differential structure}\label{sectionClassical}

In our final section, we show that differential structure is preserved when we add relative complements to a join restriction category.  This process will greatly expand the possible domains of definition for differentiable maps, even in the standard example.  The standard example (smooth maps on open subsets) does not have relative complements.  By adding them in, we add smooth maps between any set which is the complement of an open subset inside some other open subset.  Of course, this includes closed sets, and so by applying this construction, we have a category of smooth maps defined on all open, closed and half open-half-closed sets.  This includes smooth functions defined on points; as we shall see below, this captures the notion of the germ of a smooth function.

\subsection{The classical completion}

The notion of classical restriction category was defined in \cite{boolean} as an intermediary between arbitrary restriction categories and the Boolean restriction categories of \cite{booleanManes}.  

\begin{definition}
A restriction category $\X$ with restriction zeroes is a \textbf{classical restriction category} if
\begin{enumerate}
	\item the homsets are locally Boolean posets (under the restriction order), and for any $W \to^f X, Y \to^g Z$, 
		\[ \X(X,Y) \to^{f \circ (-) \circ g} \X(W,Z) \]
	is a locally Boolean morphism;
	\item for any disjoint maps $f,g$ (that is, $\rs{f}\rs{g} = \emptyset$), $f \vee g$ exists.
\end{enumerate} 
\end{definition}

\begin{example} 
Sets and partial functions form a classical restriction category.
\end{example}
  
For our purposes, the following alternate characterization of the definition, which describes classical restriction categories as join restriction categories with relative complements, is more useful.

\begin{definition}
If $f' \leq f$, the \textbf{relative complement} of $f'$ in $f$, denoted $f \setminus f'$, is the unique map such that 
\begin{itemize}
	\item $f \setminus f' \leq f$;
	\item $g \wedge (f \setminus f') = \emptyset$;
	\item $f \leq g \vee (f \setminus f')$.
\end{itemize}
\end{definition}

The following can be found in \cite{boolean}:
\begin{proposition}
A classical restriction category is a join restriction category with relative complements $f \setminus f'$ for any $f' \leq f$.
\end{proposition}

Just as one can freely add joins to an arbitrary restriction category, so too can one freely add relative complements to a join restriction category.  We will first describe this completion process, then show that cartesian, additive, and differential structure is preserved when classically completing.  This is of great interest, as classically completing adds in a number of new maps, even to the standard examples.  

\begin{definition}
Let $\X$ be a join restriction category.  A {\bf classical piece} of $\X$ is a pair of maps $(f,f'): A \to B$ such that $f' \leq f$.  
\end{definition}
One thinks of a classical piece as a formal relative complement. 
\begin{definition}
Two classical pieces $(f,f'), (g,g')$ are {\bf disjoint}, written $(f,f') \perp (g,g')$, if $\rs{f}\rs{g} = \rs{f'}\rs{g} \vee \rs{f}\rs{g'}$.  A {\bf raw classical piece} consists of a finite set of classical pieces, $(f_i, f_i')$ that are pairwise disjoint, and is written
	\[ \bigsqcup_{i \in I} (f_i, f_i'): A \to B. \]
\end{definition}

One defines an equivalence relation on the set of raw classical maps by:
\begin{itemize}
	\item {\bf Breaking:} $(f,f') \equiv (ef,ef') \sqcup (f, f' \vee fe)$ for any restriction idempotent $e = \rs{e}$,
	\item {\bf Collapse:} $(f,f) \equiv \emptyset$.
\end{itemize}
The first part of the equivalence relation says that if we have some other domain $e$, then we can split the formal complement $(f,f')$ into two parts: the first part, $(ef, ef')$, inside $e$, and the second, $(f, f' \vee fe)$, outside $e$.  The second part of the equivalence is obvious: if you formally take away all of $f$ from $f$, the result should be nowhere defined.  

\begin{definition}
A {\bf classical map} is an equivalence class of raw classical maps.
\end{definition}

\begin{proposition}
Given a join restriction category $\X$, there is a classical restriction category $\cl(\X)$ with 
\begin{itemize}
	\item objects those of $\X$, 
	\item arrows classical maps,
	\item composition by 
		\[ \bigsqcup_{i \in I}(f_i, f_i') \bigsqcup_{j \in J} (g_j, g_j') := \bigsqcup_{i,j} (f_ig_j, f_i'g_j \vee f_ig_j'), \]
	\item restriction by
		\[ \rs{\bigsqcup_{i \in I} (f_i, f_i')} := \bigsqcup_{i \in I} (\rs{f_i}, \rs{f_i'}) , \]
	\item disjoint join is simply $\sqcup$ of classical pieces,
	\item relative complement is
		\[ (f,f')\setminus (g,g') := (f,f' \vee \rs{g}f) \sqcup (\rs{g'}f, \rs{g'}f'). \]
\end{itemize}
\end{proposition}

In \cite{boolean}, this process is shown to give a left adjoint to the forgetful functor from classical restriction categories to join restriction categories.

We make one final point about the definition.  We defined $(f,f') \perp (f_0, f_0')$ if $\rs{f}\rs{f_0} = \rs{f'}\rs{f_0} \vee \rs{f}\rs{f_0'}$.  Note, however, that it suffices that we have $\leq$, since 
	\[ \rs{f'}\rs{f_0} \vee \rs{f}\rs{f_0'} \leq \rs{f}\rs{f_0} \vee \rs{f}\rs{f_0} = \rs{f}\rs{f_0} \]
We will often use this alternate form of $\perp$ when checking whether maps we give are well-defined.  

\subsection{Cartesian structure}
Our goal is to show that if $\X$ has differential restriction structure, then so does $\cl(\X)$.  We begin by showing that cartesian structure is preserved, and for this we begin by define the pairing of two classical maps.  

\begin{lemma}
Given a join restriction category $\X$ and maps $\bigsqcup (f_i,f_i')$ from $Z$ to $X$ and $\bigsqcup (g_j, g_j')$ from $Z$ to $Y$ in $\cl(\X)$, the following:
	\[ \left< \bigsqcup_{i} (f_i, f_i'), \bigsqcup_j (g_j, g_j') \right> := \bigsqcup_{i,j} \left( \<f_i,g_j\>, \<f_i',g_j\> \vee \<f_i,g_j'\> \right) \]
is a well-defined map from $Z$ to $X \times Y$ in $\cl(\X)$.  
\end{lemma}
\begin{proof}
First, we need to check 
	\[ ( \<f,g\>, \<f',g\> \vee \<f,g'\>) \]
defines a classical piece.  Indeed, since $f' \smile f$ and $g \smile g'$, the two maps being joined are compatible, so we can take the join.  Also, since $f' \leq f$ and $g' \leq g$, the right component is less than or equal to the left component.

Now, we need to check that
	\[ \bigsqcup_{i,j} \left( \<f_i,g_j\>, \<f_i',g_j\> \vee \<f_i,g_j'\> \right) \]
defines a raw classical map.  That is, we need to check that the pieces are disjoint.  That is, we need to show that if 
	\[ (f,f') \perp (f_0, f_0') \mbox{ and } (g,g') \perp (g_0, g_0') \]
then
	\[ (\<f,g\>, \<f',g\> \vee \<f,g'\>) \perp (\<f_0, g_0\>, \<f_0', g_0'\> \vee \<f_0, g_0' \> ). \]
Consider:
\begin{eqnarray*}
&   & \rs{\<f,g\>}   \rs{\<f_0, g_0 \>} \\
& = & \rs{f}\rs{f_0} \rs{g} \rs{g_0} \\
& = & (\rs{f'} \rs{f_0} \vee \rs{f} \rs{f_0'})(\rs{g'}\rs{g_0} \vee \rs{g} \rs{g_0'}) \\
& = & \rs{f'}\rs{f_0}\rs{g'}\rs{g_0} \vee \rs{f}\rs{f_0'} \rs{g'} \rs{g_0} \vee \rs{f'}\rs{f_0} \rs{g} \rs{g_0'} \vee \rs{f} \rs{f_0'} \rs{g} \rs{g_0'} \\
& \leq & \rs{f}\rs{g'}\rs{f_0}\rs{g_0} \vee \rs{f}\rs{g} \rs{f_0} \rs{g_0} \vee \rs{f'}\rs{g} \rs{f_0} \rs{g_0} \vee \rs{f} \rs{g} \rs{f_0} \rs{g_0'} \\
& = & (\rs{f'}\rs{g} \vee \rs{f}\rs{g'})(\rs{f_0}\rs{g_0}) \vee (\rs{f}\rs{g})(\rs{f_0'} \rs{g_0} \vee \rs{f_0}\rs{g_0'}) \\
& = & \rs{\<f',g\> \vee \<f,g'\>} \rs{\<f_0,g_0\>} \vee \rs{\<f,g\>} \rs{\<f_0',g_0\> \vee \<f_0, g_0'\>} 
\end{eqnarray*}
so that \[ (\<f,g\>, \<f',g\> \vee \<f,g'\>) \perp (\<f_0, g_0\>, \<f_0', g_0'\> \vee \<f_0, g_0' \> ), \]
as required.  

Finally, we need to check that this is a well-defined classical map.  Thus, we need to check it is well-defined with respect to collapse and breaking.  For collapse, consider
	\[ \<(f,f'), (g,g) \> = (\<f,g\>, \<f',g\> \vee \<f,g\>) = (\<f,g\>, \<f,g\>) \equiv \emptyset \]
as required.  

For breaking, suppose we have
	\[ (g,g') \equiv (g,g' \vee eg) \perp (eg, eg') \]
Then 
\begin{eqnarray*}
&   & \<(f,f'), (g,g' \vee eg) \perp (eg, eg') \> \\
& = & (\<f,g\>, \<f',g\> \vee \<f,g' \vee eg\>) \perp (\<f,eg\>, \<f',eg\> \vee \<f,eg'\>) \\
& = & (\<f,g\>, \<f',g\> \vee e\<f,g' \vee eg\>) \perp (e\<f,g\>, e(\<f',g\> \vee \<f,g'\>) \\
& \equiv & (\<f,g\>, \<f',g\> \vee \<f,g'\>) \\
& = & \<(f,f'), (g,g') \>
\end{eqnarray*}
as required.  Thus, the above is a well-defined classical map.  
\end{proof} 

We now give some lemmas about our definition.  Note that once we show that this pairing does define cartesian structure on $\cl(\X)$ , these lemmas follow automatically, as they are true in any cartesian restriction category (see Lemma \ref{propCart})  However, we will need these lemmas to establish that this does define cartesian structure on $\cl(\X)$.

\begin{lemma}\label{lemmaCart1}
Suppose we have maps $f: Z \to X$, $g: Z \to Y$, and $e = \rs{e}: Z \to Z$ in $\cl(\X)$.  Then $e\<f,g\> = \<ef,g\> = \<f,eg\>$.
\end{lemma}
\begin{proof}
It suffices to show the result for classical pieces.  Thus, consider
\begin{eqnarray*}
&   & \<(e,e')(f,f'), (g,g') \> \\
& = & \<(ef, e'f \vee ef'), (g,g') \> \\
& = & (\<ef, g\>, \<e'f \vee ef', g\> \vee \<ef,g'\>) \\
& = & (e\<f,g\>, e'\<f,g\> \vee e\<f',g\> \vee e\<f,g'\> \\
& = & (e,e')(\<f,g\>, \<f',g\> \vee \<f,g'\> \\
& = & (e,e')\<(f,f'), (g,g') \>
\end{eqnarray*}
as required.  Putting the $e$ in the right component is similar.  
\end{proof}

\begin{lemma}\label{lemmaCart2}
For any $c$ in $\cl(\X)$, $\<c\pi_0, c\pi_1\> = c$.
\end{lemma}
\begin{proof}
It suffices to show the result for classical pieces.  Thus, consider
\begin{eqnarray*}
&   & \<(c,c')(\pi_0,\emptyset), (c,c')(\pi_1, \emptyset)\> \\
& = & \<(c\pi_0, c'\pi_0), (c\pi_1, c'\pi_1) \> \\
& = & (\<c\pi_0, c\pi_1\>, \<c'\pi_0, c\pi_1\> \vee \<c\pi_0, c'\pi_1\>) \\
& = & (c, \rs{c'}\<c\pi_0, c\pi_1\> \vee \rs{c'}\<c\pi_0, c\pi_1\>) \\
& = & (c, \rs{c'}c \vee \rs{c'}c) \\
& = & (c,c')
\end{eqnarray*}
as required.
\end{proof}

It will be most helpful if we can give an alternate characterization of when two classical maps are equivalent.  To that, we prove the following result:

\begin{theorem}
In $\cl(\X)$, $(f,f') \equiv (g,g')$ if and only if there exist restriction idempotents $e_1, \ldots,  e_n$ such that for any $I \subseteq \{1, \ldots, n\}$, if we define 
	\[ e_I := \left( \bigcirc_{i \in I} e_i, \left(\bigcirc_{i \in I} e_i \right)\left(\bigvee_{j \not\in I} e_j \right) \right) \]
(where $\bigcirc$ denotes iterated composition) then for each such $I$,
	\[ e_I(f,f') = e_I(g,g') \]
or they both collapse to the empty map.  
\end{theorem}
\begin{proof}
As discussed in \cite{boolean}, breaking and collapse form a system of rewrites, so that if two maps are equivalent, they can be broken into a series of pieces, each of which are either equal or both collapse to the empty map.  Thus, it suffices to show that the above is what occurs after doing $n$ different breakings along the idempotents $e_1, \ldots, e_n$.  To this end, note that the two pieces left after breaking $(f,f')$ by $e$ are given by precomposing with $(e,\emptyset)$ and $(1,e)$; indeed:
	\[ (e,\emptyset)(f,f') = (ef,ef') \mbox{ and } (1,e)(f,f') = (f,ef \vee f') \]
Thus, if $n=1$, the result holds.  Now assume by induction that the result holds for $k$.  Then for any subset $I \subseteq \{1, \ldots n\}$, breaking $e_I$ by $(e_{k+1})$ gives the pieces
	\[ (e_{n+1}, \emptyset)(\circ e_i, (\circ{e_i})(\vee e_j) = (e_{n+1} \circ e_i, (e_{n+1} \circ{e_i})(\vee e_j)) \]
and
	\[ (1, e_{n+1})(\circ e_i, (\circ e_i)(\vee e_j) = (\circ e_i, (\circ e_i)(e_{n+1}) \vee (\circ e_i)(\vee e_j)) = (\circ e_i, (\circ e_i)(e_{n+1} \vee e_j)) \]
Thus, we get all possible idempotents $e_{I'}$, where $I' \subseteq \{1, \ldots, n+1 \}$, as required.  
\end{proof}

\begin{theorem}
If $\X$ is a cartesian restriction category, then so is $\cl(\X)$.
\end{theorem}
\begin{proof}
Define the terminal object $T$ as for $\X$, and the unique maps by $!_A := (!_A, \emptyset)$.  Then for any classical map $\bigsqcup (f_i, f_i')$, we have
	\[ \bigsqcup (f_i, f_i') = \bigsqcup (!_A \rs{f_i}, !_A \rs{f_i'}) = \left(\bigsqcup (\rs{f_i}, \rs{f_i'})\right)(!_A, \emptyset) \]
as required.  So $\cl(\X)$ has a partial final object.  

We define the product objects $A \times B$ as for $\X$, the projections by $(\pi_0, \emptyset)$ and $(\pi_1, \emptyset)$, and the product map as above.  To show that our putative product composes well with the projections, consider
\begin{eqnarray*}
&   & \<(f,f'), (g,g')\>(\pi_0, \emptyset) \\
& = & (\<f,g\>, \<f',g\> \vee \<f,g'\>) (\pi_0, \emptyset) \\
& = & (\<f,g\>\pi_0, \<f',g\>\pi_0 \vee \<f,g'\>\pi_0) \\
& = & (\rs{g}f, \rs{g}f' \vee \rs{g'}\rs{f}) \\
& = & (\rs{g}, \rs{g'})(f, f') \\
& = & \rs{(g,g')}(f,f')
\end{eqnarray*}
as required.  Composing with $\pi_1$ is similar.  

Finally, we need to show that the universal property holds.  It suffices to show that if $c\pi_0 \leq f$ and $c\pi_1 \leq g$, then $c \leq \<f,g\>$.  Suppose we have the first two inequalities, so that 
	\[ \rs{c}f \equiv c\pi_0 \mbox{ by breaking with idempotents $(e_1, \ldots, e_n)$} \]
and
	\[ \rs{c}g \equiv c\pi_1 \mbox{ by breaking with idempotents $(d_1, \ldots, d_m)$}. \]
We claim that $\rs{c}\<f,g\> \equiv c$ by breaking with idempotents $(e_1, \ldots e_n, d_1, \ldots, d_m)$.  By the previous theorem, it suffices to show they are equal (or both collapse to the empty map) when composing with an element of the form in the theorem for an arbitrary subset $K \subseteq \{1, \ldots n, n+1, \ldots n+m\}$.  However, if $I = K \cap \{1, \ldots, n\}$ and $J = K \cap \{n+1, \ldots n+m\}$, then such an element can be written as 
	\[ (e_I,e_Ie_{I'})(d_J, d_Jd_{J'}) \]
since that equals
	\[ (e_Id_J, (e_Id_J)(e_{I'} \vee d_{J'})) \]
which is $e_K$.  Thus, writing $e$ for $(e_I,e_Ie_{I'})$ and $d$ for $(d_J, d_Jd_{J'})$, it suffices to show that $ed\rs{c}\<f,g\> = edc$ (or they both collapse to the empty map).  However, we know that 
	\[ e \rs{c}f = ec\pi_0 \mbox { and } d\rs{c}g = dc\pi_1 \]
(or one or the other collapses to the empty map).  Pairing the above equalities, we get
	\[ \<e\rs{c}f,d\rs{c}g\> = \<ec\pi_0, dc\pi_1 \> \]
which, by lemma \ref{propCart}, reduces to
	\[ (ed)\rs{c}\<f,g\> = edc \]
as required.  If either equality has both sides collapsing to the empty map, then both sides of the above collapse to the empty map, since we showed earlier that pairing is well-defined when applied to collapsed maps.  Thus, we have the required universal property, and $\cl(\X)$ is cartesian.
\end{proof}

\subsection{Left additive structure}

Next, we show that left additive structure is preserved.  We begin by defining the sum of two maps.  

\begin{lemma}
Suppose that $\X$ is a left additive restriction category with joins.  Given maps $\bigsqcup (f_i,f_i')$ and $\bigsqcup (g_j, g_j')$ from $X$ to $Y$ in $\cl(\X)$, the following:
	\[ \bigsqcup_{i,j} (f_i + g_j, (f'_i + g_j) \vee (f_i + g_j')) \]
is a well defined map from $X$ to $Y$ in $\cl(\X)$.
\end{lemma}
\begin{proof}
The proof is nearly identical to that for showing that the pairing definition gives a well-defined classical map.

\end{proof}

\begin{theorem}
If $\X$ has the structure of a left additive restriction category, then so does $\cl(\X)$, where addition of maps is defined as above, and the zero map is given by $(0,\emptyset)$. 
\end{theorem}
\begin{proof}
It is easily checked that the addition and zero give each homiest the structure of a commutative monoid.  For the restriction axioms,
\begin{eqnarray*}
&   & \rs{(f,f') + (g,g')} \\
& = & \rs{(f+g, (f'+g) \vee(f+ g'))} \\
& = & (\rs{f+g}, \rs{f'+g} \vee \rs{f + g'}) \\
& = & (\rs{f}\rs{g}, \rs{f'}\rs{g} \vee \rs{f}\rs{g'}) \\
& = & (\rs{f}, \rs{f'})(\rs{g}, \rs{g'}) \\
& = & \rs{(f,f')} \rs{(g,g')}
\end{eqnarray*}
and clearly $(0,\emptyset)$ is total.  For the left additivity, consider
\begin{eqnarray*}
&   & (f,f')(g,g') + (f,f')(h,h') \\
& = & (fg, f'g \vee fg') + (fh, f'h \vee fh') \\
& = & (fg + fh, ((f'g \vee fg') + fh) \vee (fg + (f'h \vee fh'))) \\
& = & (fg + fh, (f'g + fh) \vee (fg' + fh) \vee (fg + f'h) \vee (fg + fh')) \\
& = & (fg + fh, \rs{f'}(fg + fh) \vee \rs{f'}(fg + fh) \vee (fg' + fh) \vee (fg + fh')) \mbox{ since $f' \leq f$} \\
& = & (f(g+h), \rs{f'}f(g+h) \vee f(g'+h) \vee f(g+h')) \\
& = & (f(g+h), f'(g+h) \vee f(g'+h) \vee f(g+h')) \\
& = & (f,f')(g+h, (g'+h) \vee (g+h')) \\
& = & (f,f')((g,g') + (h,h')) 
\end{eqnarray*}
as required.  Thus $\cl(\X)$ is a left additive restriction category. 
\end{proof}

\begin{theorem}
If $\X$ has the structure of a cartesian left additive restriction category, then so does $\cl(\X)$.  
\end{theorem}
\begin{proof}
Immediate from Theorem \ref{thmAddCart}.
\end{proof}

\subsection{Differential structure}

Finally, we show that if $\X$ has differential restriction structure, so does $\cl(\X)$.  We first need to define the differential of a map.

\begin{lemma}
If $\X$ is a differential join restriction category, and $\bigsqcup (f_i, f_i')$ is a map from $X$ to $Y$ in $\cl(\X)$, then the following:
	\[ \bigsqcup (D[f_i], D[f_i']) \]
is a well-defined map in $\cl(\X)$ from $X \times X$ to $Y$.
\end{lemma}

\begin{proof}
If $f' \leq f$, then $D[f'] \leq D[f]$, so it is a well-defined classical piece.  If $(f,f') \perp (g,g')$, then
\begin{eqnarray*}
&   & \rs{Df}\rs{Dg} \\
& = & (1 \times \rs{f})(1 \times \rs{g}) \\
& = & 1 \times \rs{f}\rs{g} \\
& = & 1 \times (\rs{f'}\rs{g} \vee \rs{f}\rs{g'}) \mbox{ since $(f,f') \perp (g,g')$} \\
& = & (1 \times \rs{f'}\rs{g}) \vee (1 \times \rs{f}\rs{g'}) \\
& = & (1 \times \rs{f'})(1 \times \rs{g}) \vee (1 \times \rs{f})(1 \times \rs{g'}) \\
& = & \rs{Df'} \rs{Dg} \vee \rs{Df}\rs{Dg'} 
\end{eqnarray*}
so $(Df, Df') \perp (Dg, Dg')$, so it is a well-defined raw classical map.

That this is well-defined under collapsing is obvious.  For breaking, suppose we have
	\[ (f,f') \equiv (f,f' \vee ef) \perp (ef, ef') \]
for some restriction idempotent $e = \rs{e}$.  Then consider
\begin{eqnarray*}
&   & D[(f,f' \vee ef) \perp (ef, ef')] \\
& = & (Df, Df' \vee D(ef)) \perp (D(ef), D(ef')) \\
& = & (Df, Df' \vee (1 \times e)Df) \perp ((1 \times e)Df, (1 \times e) Df')) \mbox{ by lemma \ref{propDiff}} \\
& \equiv & (Df, Df') \mbox{ by breaking along the restriction idempotent $(1 \times e)$.}
\end{eqnarray*}
Thus the map is well-defined under collapsing and breaking, so is a well-defined classical map.
\end{proof}

\begin{theorem}
If $\X$ is a differential join restriction category, then so is $\cl(\X)$, with the differential of $\bigsqcup (f_i, f_i')$ given above.
\end{theorem}
\begin{proof}
Most axioms involve a straightforward calculation and use of the lemmas we have developed.   We shall demonstrate the two most involved calculations: \dr{2} and \dr{5}.  For \dr{2}, consider
\begin{eqnarray*}
&   & \<(g,g'), (k,k')\>D(f,f') + \<(h,h'),(k,k')\>D(f,f') \\
& = & (\<g,k\>,\<g',k\> \vee \<g,k'\>)(Df,Df') + (\<h,k\>, \<h',k\> \vee \<h,k'\>)(Df,Df') \\
& = & (\<g,k\>Df,\<g',k\>Df \vee \<g,k'\>Df \vee \<g,k\>Df) + (\<h,k\>Df, \<h',k\>Df \vee \<h,k'\>Df \vee \<h,k\>Df') \\
& = & (\<g,k\>Df + \<h,k\>Df, [\<g',k\>Df + \<h,k\>Df] \vee [\<g,k'\>Df + \<h,k\>Df] \vee [\<g,k\>Df' + \<h,k\>Df] \\
&   & \vee [\<g,k\>Df + \<h',k\>Df] \vee [\<g,k\>Df + \<h,k'\>Df] \vee [\<g,k\>Df + \<h,k\>Df']) \\
\end{eqnarray*}
We can simplify a term like $\<g,k'\>Df$ as follows:
	\[ \<g,k'\>Df = \<g,\rs{k'}k\>Df = \rs{k'}\<g,k\>Df \]
And for a term like $\<g,k\>Df'$, we can simplify it as follows:
	\[ \<g,k\>Df' = \<g,k\>D(\rs{f'}f) = \<g,k\>(1 \times \rs{f'})Df = \<g,k\rs{f'}\>Df = \<g,\rs{kf'}k\>Df = \rs{kf'}\<g,k\>Df \]
Thus, continuing the calculation above, we get
\begin{eqnarray*}
& = & (\<g,k\>Df + \<h,k\>Df, [\<g',k\>Df + \<h,k\>Df] \vee \rs{k'}[\<g,k\>Df + \<h,k\>Df] \vee \rs{kf'}[\<g,k\>Df + \<h,k\>Df] \\
&   & \vee [\<g,k\>Df + \<h',k\>Df] \vee \rs{k'}[\<g,k\>Df + \<h,k'\>Df] \vee \rs{kf'}[\<g,k\>Df + \<h,k\>Df']) \\
& = & (\<g,k\>Df + \<h,k\>Df, [\<g',k\>Df + \<h,k\>Df] \vee [\<g,k\>Df + \<h',k\>Df] \\
&   & \vee \rs{k'}[\<g,k\>Df + \<h,k\>Df] \vee \rs{kf'}[\<g,k\>Df + \<h,k\>Df]) \\
& = & (\<g,k\>Df + \<h,k\>Df, [\<g',k\>Df + \<h,k\>Df] \vee [\<g,k\>Df + \<h',k\>Df] \\
&   & \vee [\<g,k'\>Df + \<h,k'\>Df] \vee [\<g,k\>Df' + \<h,k\>Df']) \mbox{ using the above calculations in reverse}\\
& = & (\<g+h,k\>Df, \<g'+h,k\>Df \vee \<g+h',k\>Df \vee \<g+h,k'\>Df \vee \<g+h,k\>Df') \mbox{ by \dr{2} for $\X$} \\
& = & (\<g+h,k\>, \<g'+h,k\> \vee \<g+h',k\> \vee \<g+h,k'\>)(Df, Df') \\
& = & \<(g+h, (g'+h) \vee (g+h')),(k,k')\>(Df, Df') \\
& = & \<(g,g') + (h,h'), (k,k')\>D(f,f') \\
\end{eqnarray*}
as required.  For \dr{5}, consider
\begin{eqnarray*}
&   & \<D(f,f'), (\pi_1, \emptyset)(f,f')\>D(g,g') \\
& = & \<(Df, Df'), (\pi_1 f, \pi_1 f')\>(Dg,Dg') \\
& = & (\<Df, \pi_1f\>, \<Df',\pi_1\> \vee \<Df, \pi_1 f'\>)(Dg, Dg') \\
& = & (\<Df, \pi_1f\>Dg, \<Df',\pi_1f\>Dg \vee \<Df, \pi_1f'\>Dg \vee \<Df, \pi_1f\>Dg')
\end{eqnarray*}

Now, we can simplify 
	\[ \<Df',\pi_1f\> = \<D(\rs{f'}f),\pi_1f\> = \<(1 \times \rs{f'})Df,\pi_1f\> = (1 \times \rs{f'})\<Df,\pi_1f\> \]
(where the second equality is by Lemma \ref{propDiff}), and
	\[ \<Df,\pi_1f'\> = \<Df, \pi_1\rs{f'}f\> = \<Df, (1 \times \rs{f'})\pi_1f\> = (1 \times \rs{f'})\<Df, \pi_1f \> \]
where the second equality is by lemma \ref{propCart}.  Thus, the above becomes
\begin{eqnarray*}
& = & (\<Df, \pi_1f\>Dg, (1 \times \rs{f'})\<Df, \pi_1f\>Dg \vee \<Df, \pi_1f\>Dg') \\
& = & (D(fg), (1 \times \rs{f'}D(fg) \vee D(fg')) \mbox{ by \dr{5} for $\X$} \\
& = & (D(fg), D(f'g) \vee D(fg')) \mbox{ by Lemma \ref{propDiff}} \\
& = & D(fg, f'g \vee fg') \\
& = & D((f,f')(g,g'))
\end{eqnarray*}
as required.  
\end{proof}

Now that we know that the classical completion of a differential restriction category is again a differential restriction category, it will be interesting to see what type of maps are in the classical completion of the standard model.  For example, consider two functions: $f(x) = 2x$ defined everywhere but $x=5$, and $g(x) = 2x$ defined everywhere.  Taking the relative complement of these maps gives a map defined \emph{only} at $x=5$, and has the value $2x = 10$ there.  But if differential structure is retained, in what sense is this map ``smooth''?

Of course, this map is really an equivalence class of maps.  In particular, imagine we have a restriction idempotent $e = \rs{e}$ (that is, an open subset), which includes $5$.  Then we have
	\[ (f,f') \equiv (ef,ef') \sqcup (f,f'\vee ef) = (ef,ef') \sqcup (f,f) \equiv (ef,ef') \]
So that this map is actually equivalent to any other map defined on an open subset which includes $5$.  This is precisely the definition of the \emph{germ} of a function at $5$.  Thus, the classical completion process adds germs of functions at points.

Of course, it also allows us to take joins of germs and regular maps, so that for example we could take the join of the above map, and something like $\frac{x-1}{x-5}$, giving a total map which has ``repaired'' the discontinuity of the second map at $5$.  The fact that this restriction category is a differential restriction category is perhaps now much more surprising.  Clearly, this will be an example that will need to be explored further.  \


Finally, given the additive, cartesian, and differential structure of $\cl(\X)$, the following is immediate:
\begin{proposition}
The unit $\X \to \cl(\X)$, which sends $f$ to $(f,\emptyset)$, is a differential restriction functor.
\end{proposition}

And as a result, we have the following:
\begin{corollary}
Suppose $\X$ is a differential restriction category with joins, and $f' \leq f$.  Then:
\begin{enumerate}[(i)]
	\item if $f$ is additive in $\X$, then so are $(f,\emptyset)$ and $(f,f')$ in $\cl(\X)$;
	\item if $f$ is strongly additive in $\X$, then so is $(f,\emptyset)$ in $\cl(\X)$;
	\item if $f$ is linear in $\X$, then so are $(f,\emptyset)$ and $(f,f')$ in $\cl(\X)$.
\end{enumerate}
\end{corollary}
\begin{proof}
By Proposition \ref{diffFunctors}, $(f,\emptyset)$ retains being additive/strongly additive/linear, and since $(f,f')$ is a relative complement, $(f,f') \leq f$, so is additive/linear if $f$ is.
\end{proof}

	
\section{Conclusion}\label{sectionConclusion}

There are a number of different expansions of this work that are possible; here we mention the most immediate.   A construction given in \cite{manifolds} allows one to build a new restriction 
category of manifolds out of any join restriction category.  For example, applying this construction to continuous functions defined on open subsets of $\R^n$ gives a category of real manifolds.  
An obvious expansion of the present theory is to understand what happens when we apply this construction to a differential restriction category with joins.  Clearly, this will build categories 
of smooth maps between smooth manifolds.  In general, however, one should not expect this to again be a differential restriction category, as the derivative of a smooth manifold map $f: M \to N$ 
is not a map $M \times M \to N$, but instead a map $TM \to TN$, where $T$ is the tangent bundle functor.  Thus, we must show that one can describe the tangent bundle of any object in the manifold 
completion of a differential restriction category.  This leads one to consider using the tangent space as a basis for axiomatizing this sort of differential structure.  This is the subject of a 
future paper, and will allow for closer comparisons between the theory presented here and synthetic differential geometry.

\section*{Acknowlegements}

The authors are grateful to both the referee and the editor for their handling of an error concerning the fractional monad construction.  In the 
first version of the paper we had rather stupidly failed to observe that the fraction construction was not distributive unless one insisted on the equation 
$(a,a^2) = (1,a)$.  The referee had therefore suggested we should add the equation. This resulted in the expanded section on the fractional monad and, the  
non-standard -- but we felt interesting -- presentation of the rational functions where we were forced to use weak rigs.

\end{document}